\documentclass[table,x11names]{amsart} 
\usepackage{va, styf2}

\counterwithin{figure}{section}
\counterwithin{table}{section}
\makeatletter

\let\c@table\c@figure
\makeatother 

\usepackage{subcaption}
\usepackage{textcomp}
\usepackage{array,multirow}
\usepackage{float}
\usepackage{tikz}
\usetikzlibrary{arrows}
\usetikzlibrary{cd}

\usepackage{hyperref}
\hypersetup{colorlinks,citecolor=blue,linktocpage,hyperindex=true,backref=true}
\allowdisplaybreaks

\usepackage{textcomp}
\usepackage{bm}

\usepackage{etoolbox}

\makeatletter
\pretocmd{\section}{\addtocontents{toc}{\protect\addvspace{5\p@}}}{}{}
\makeatother
\usepackage[disable]{todonotes}
\newcommand\va[1]{\todo[color=green!40]{#1}} 

\date{December 1, 2021}
\title[Compactification of moduli of K3 surfaces]{Stable pair 
  compactification of\\ moduli of K3 surfaces of degree 2}
\author{Valery Alexeev}
\email{valery@uga.edu}
\address{Department of Mathematics, University of Georgia, Athens GA
  30602, USA}
\author{Philip Engel}
\email{philip.engel@uga.edu}
\address{Department of Mathematics, University of Georgia, Athens GA
  30602, USA}
\author{Alan Thompson}
\email{a.m.thompson@lboro.ac.uk}
\address{Department of Mathematical Sciences, Loughborough University, 
Loughborough, Leicestershire, LE11 3TU, UK}

\begin{document}
\numberwithin{equation}{section}

\begin{abstract}
  We prove that the universal family of polarized K3 surfaces of
  degree 2 can be extended to a flat family of stable KSBA pairs
  $(X,\epsilon R)$ over the toroidal compactification associated to
  the Coxeter fan. One-parameter degenerations of K3 surfaces in this
  family are described by integral-affine structures on a sphere with
  $24$ singularities.
\end{abstract}

\maketitle 
\tableofcontents

\section{Introduction}
\label{sec:intro}

By the Torelli theorem \cite{piateski-shapiro1971torelli},
the coarse moduli space $F_{2d}$ of primitively polarized K3 surfaces $(X,L)$
of degree $L^2=2d$ is the quotient $F_{2d}=\Gamma\backslash\bD$ of
a 19-dimensional Hermitian symmetric domain by an arithmetic group.
In its capacity as an arithmetic quotient, there are the Baily-Borel
$\oF_{2d}\ubb$ \cite{baily1966compactification-of-arithmetic}
and infinitely many toroidal $\oF_{2d}^\tor$ \cite{ash1975smooth-compactifications} compactifications of $F_{2d}$.
These were unified by the more general semitoric compactifications
$\oF_{2d}^\semi$ \cite{looijenga2003compactifications-defined2} of Looijenga.

The geometry of these Hodge-theoretic compactifications can be described explicitly.
For instance, the incidence structure of the boundary strata is encoded by 
combinatorial information, called a semifan $\mathfrak{F}^\semi$.
But \emph{a priori},  semitoric compactifications are not modular---the boundary
points need not parameterize some geometric
generalization of a K3 surface.

On the other hand, if we canonically choose for every polarized K3 surface $(X,L)$
an effective divisor $R\in |NL|$ in a fixed multiple of the polarization, we get a
geometrically meaningful compactification $F_{2d}\hookrightarrow \oF^{\slc}_{2d}$ by taking the closure
of the space of pairs $(X,\epsilon R)$ in the moduli space of all KSBA stable pairs.
These are pairs with semi log canonical (slc) singularities and ample log
canonical class $K_X+\epsilon R$, see e.g.
\cite{kollar1988threefolds-and-deformations,
 kollar2022book-on-moduli}, \cite{alexeev1996moduli-spaces,
 alexeev2006higher-dimensional-analogues}. Generally, it is very
hard to describe the boundary of $\oF_{2d}^\slc$ and the surfaces
appearing over it.

Thus, finding compactifications of K3 moduli which are
both Hodge-theoretic and algebro-geometric has been
a central, and largely open, motivating question:

\begin{question}\label{que:main-question}
Do $\oF_{2d}^\slc$ and $\oF_{2d}^\tor$ coincide for appropriate choices
of divisor $R$ and fan $\fF$? If so, what are the fibers over the toroidal boundary strata?
\end{question}

For the moduli space~$A_g$ of principally polarized abelian varieties (ppavs), 
these questions were answered affirmatively in \cite{alexeev2002complete-moduli}:
On a ppav $(X,L)$, we choose the unique theta divisor $R=\Theta\in |L|$ in the principal
polarization. Then the closure of the pairs $(X,\epsilon \Theta)$ in the space
of KSBA stable pairs coincides (up to normalization) with the toroidal compactification
associated to the second Voronoi fan.

In this paper, we answer Question~\ref{que:main-question}
affirmatively for the moduli space $F_2$.  A K3 surface
$(X,L)$ of degree $2$ is canonically equipped with an involution
$\iota$, switching the sheets of $\phi_{|L|}$. So its ramification divisor
$R={\rm Fix}(\iota)\in |3L|$ is uniquely determined by $(X,L)$. 
Thus, the closure of the space of pairs $(X,\epsilon R)$ gives a geometric
compactification of $F_2$.

On the other hand, there is a natural choice of toroidal compactification. A
fan is given by an $O(N)$-invariant polyhedral decomposition of the rational closure of
the positive cone in $N:=H\oplus E_8^2\oplus A_1$ which is a hyperbolic lattice of
 signature $(1,18)$.  Then $N$ is a hyperbolic root lattice and we define the Coxeter fan
$\mathfrak{F}_{\cox}$ to have walls equal to the perpendiculars of the roots,
i.e.~ vectors $r\in N$ of norm $-2$. Our main result is:

\begin{theorem}\label{thm:family}
There is a semifan for which $\nu\colon \oF_2^\semi\to  \oF_2^\slc $
is the normalization of the KSBA compactification
associated to the ramification divisor $R$. The Coxeter fan refines
this semifan, and hence there is a family of stable pairs over the associated 
toroidal compactification $\oF_2^\tor$.

The KSBA-stable surfaces over the boundary of $\oF_2^\tor$ admit completely
explicit descriptions, in terms of sub-Dynkin diagrams of the
Coxeter diagram for $N$.
\end{theorem}

For a generic K3 surface of degree 2, the quotient $Y=X/\iota$ is
isomorphic to $\bP^2$, and the double cover is branched in a sextic
curve $B$.  The pair $(X,\epsilon R)$ is stable iff the pair
$(Y, \frac{1+\epsilon}2 B)$ is. Hacking \cite{hacking2004compact-moduli}
defined and studied the stable pair compactification
$\oM(\bP^2, d)$ for the pairs $(\bP^2, \frac{3+\epsilon}{d} C_d)$,
where $C_d$ is a curve of degree~$d$. Then the space
$\oF_2^\slc$ is the special case $d=6$.
Hacking provides a complete description of
$\oM(\bP^2,d)$ for $d=4,5$ and a fairly complete one for $3\nmid~d$.
Some examples of degenerate surfaces for $d=6$ are given
in \cite{hacking2004a-compactification-of-the-space},
but the problem of giving a complete description of $\oM(\bP^2,6)$
remained open. Theorem~\ref{thm:family} provides such a description.

\medskip

A moduli space related to $\oF_{2d}^\slc$ is the compactified
space $\oP_{2d}$ of K3 pairs $(X,\epsilon D)$ where $D\in|L|$
is an \emph{arbitrary} divisor in the polarization class. This space
has dimension $20+d$ versus $19$ for $\oF_{2d}$.  Laza
\cite{laza2016ksba-compactification}, building on the work of
Shah \cite{shah1980complete-moduli} and Looijenga
\cite{looijenga2003compactifications-defined2}, described
~$\oP_2$ and the degenerate pairs at the boundary.
Our constructions are unrelated, since the ramification
divisor $R$ lies in~$|3L|$.

Our compactifications of the universal family over $F_2$
provide toroidal, semitoric, and stable pair compactifications for any subfamily. 
Among them is the Heegner divisor $F_\el\subset F_2$ of elliptic K3 surfaces. 
Theorem~\ref{thm:family} directly generalizes to these
subfamilies. In particular it leads to three compactifications of $F_\el$
which are discussed further in \cite{alexeev2020compactifications-moduli}.

The compactification $\oF_\el^\slc$ induced by
$\oF_2^\slc$ is for the polarizing divisor equal to the trisection of
nontrivial $2$-torsion. Stable pair compactifications of $F_\el$ for different choices of
polarizing divisors, weighted sums of the section and fibers, were
investigated by Brunyate \cite{brunyate15modular-compactification},
Ascher-Bejleri \cite{ascher2019compact-moduli}
and \cite{alexeev2020compactifications-moduli}, with a 
description of the surfaces appearing on the boundary.

\medskip

We now briefly explain our approach and features that parallel or contrast
the case of principally polarized abelian varieties.

One-parameter degenerations of ppavs admit a toric description, due to
Mumford \cite{mumford1972an-analytic-construction}.  Let
$M\simeq\bZ^g$ be a fixed lattice, and $N=M^*$ be its dual.  The Voronoi fan
$\fF^\vor$ is supported on the rational closure $\ocC$ of the cone of
positive definite symmetric forms
$\pC= \{Q\colon M\times M\to\bR,\ Q>0\},$ equivalently of positive
symmetric maps $f_Q\colon M\to N_\bR$.  Classically, a positive
semi-definite quadratic form $Q$ defines two dual polyhedral
decompositions of $M_\bR$, periodic with respect to translation by
$M$: Voronoi and Delaunay,
cf. \cite{voronoi1908nouvelles-applications,
  voronoi1908nouvelles-applications} or
\cite{alexeev1999on-mumfords-construction}.  As $Q$ varies
continuously, so does $\Vor Q$, but the set of possible Delaunay
decompositions is discrete.  Locally closed cones of the fan
$\fF^\vor$ are precisely the subsets of $\ocC$ where the combinatorial
type of $\Vor Q$ stays constant, or equivalently where $\Del Q$ stays
constant.

A one-parameter degeneration $(X_t,\epsilon\Theta_t)$ of ppavs with an
integral monodromy vector $Q\in\cC$ can be written as a
$\bZ^g$-quotient of an infinite toric variety whose fan in
$\bR\oplus N_\bR$ is the cone over a shifted Voronoi decomposition
$\big(1, \ell + f_Q(\Vor Q) \big)$, see
\cite[1.8]{alexeev1999on-mumfords-construction}.  Mikhalkin-Zharkov
\cite{mikhalkin2008tropical-curves} called the quotient
\begin{displaymath} (X_\oldtrop,\Theta_\oldtrop) = \big( N_\bR, \ell +
f_Q(\Vor Q)\big) / f_Q(M)
\end{displaymath} a tropical principally polarized abelian variety. It is an
integral-affine torus $X_\oldtrop = N_\bR/f_Q(M) \simeq (S^1)^g$ with
a tropical divisor $\Theta_\oldtrop$ on it. Then $\Theta_\oldtrop$ induces a
cell decomposition of $X_\oldtrop$ which is the dual complex of the
singular central fiber $(X_0,\epsilon\Theta_0)$. The normalization of
each component of $X_0$ is a toric variety, whose fan is modeled by
the corresponding vertex of $\Theta_\oldtrop$.

\smallskip

Kontsevich and Soibelman proposed in
\cite{kontsevich2006affine-structures} that for K3 surfaces, the real
torus $X_{\rm trop}$ should be replaced by an integral-affine structure
with 24 singular points on a sphere $S^2$ (let us call it an $\ias$
for short). This fits into the general framework of the Gross-Siebert
program \cite{gross2003affine-manifolds}, which seeks to understand
mirror symmetry near a maximally unipotent degeneration of Calabi-Yau
varieties via tropical and integral-affine geometry.

\smallskip 

By work of Kulikov \cite{kulikov1977degenerations-of-k3-surfaces},
Persson-Pinkham \cite{persson1981degeneration-of-surfaces}, and
Friedman-Miranda \cite{friedman1983smoothing-cusp} it is understood
that a triangulated two-sphere is the combinatorial model for a {\it Type III Kulikov degeneration}:
A $K$-trivial, semistable, maximally unipotent, one-parameter family
$\mathcal{X}\rightarrow (C,0)$ of degenerating K3 surfaces.
In fact, the dual complex $\Gamma(\mathcal{X}_0)$ of the central fiber 
admits the structure of a triangulated $\ias$, cf. \cite{engel2018looijenga} and
\cite{gross2015mirror-symmetry-for-log}, which encodes
the combinatorial information of $\cX_0$. As for ppavs, one uses toric geometry
and the triangulation to build the central fiber $\mathcal{X}_0$.
The main complication for K3s is that an integral-affine structure on $S^2$ necessarily
has singularities, whereas an integral-affine structure on $(S^1)^g$ is nonsingular.

Conversely, from a triangulated $\ias$ $B$ one can reconstruct a
surface $\cX_0$ satisfying $\Gamma(\cX_0)=B$,
which smooths to a Type III degeneration by \cite{friedman1983global-smoothings}.
This ``reconstruction" procedure was used in \cite{engel2018looijenga, engel2021smoothings}
to study deformations and smoothings of cusp singularities via a
crepant resolution of the smoothing. The key
innovation in this paper is to introduce a {\it integral-affine divisor} on
an $\ias$: A weighted $1$-dimensional subcomplex $R_\trop\subset B$
which is balanced at its vertices. The Kulikov degenerations in \cite{engel2018looijenga,
  engel2021smoothings} used to study cusp
singularities were only analytic---in fact non-algebraizable
because the central fiber contains a Type VII surface, so
there was no integral-affine divisor.

\smallskip

For each vector in a connected component $\pC\subset\{\vec{a}\in N\otimes \R \mid \vec{a}^2>0\}$,
we construct an $\ias$ $B(\vec{a})$ with up to 24 singularities, together with an integral-affine
divisor $R_\trop$. As $\vec{a}\in\pC$ varies continuously, so does the pair
$\big(B(\vec{a}), R_\trop\big)$. Dual to the polyhedral decomposition
of $B(\vec{a})$ induced by $R_\trop$ is a discrete
subdivision of $S^2$ with 24 singularities.
The set of the
dual subdivisions is discrete. Thus, the family of $(B(\vec{a}), R_\trop)$ varying continuously
over $\pC$ are analogues of $\Vor Q$, and the dual subdivisions
are the analogues of $\Del Q$.

This family of $\ias$ with integral-affine divisors extends over the
rational closure~$\opC$ of the positive cone. As $\vec{a}$ approaches
a cusp of ~$\opC$, the sphere $B$ collapses to a segment, which are
dual complexes of Type II degenerations of K3 surfaces.
The cones of the Coxeter fan are exactly the
subsets of $\opC$ where the combinatorial type of the pair
$(\ias(\vec{a}),R_\trop)$ is constant, resp. where the dual subdivision
is constant, in complete analogy with the second Voronoi fan for
ppavs.

When the vector $\vec{a}$ is integral and satisfies a certain parity
condition, a triangulation
of $B(\vec{a})$ into elementary lattice triangles defines a combinatorial type of
Kulikov model. By surjectivity of an appropriate period map,
cf.~\cite{friedman1986type-III}, these Kulikov models describe all
one-parameter degenerations of K3 surfaces which a given Picard-Lefschetz
transformation, encoded in the vector $\vec{a}$.
The canonical models of these Kulikov models are the stable pairs at
the boundary of KSBA moduli. We
describe explicitly what curves and components get contracted on the
Kulikov model to produce the stable model.

Our $\ias$ are quite different from those appearing in
\cite{odaka2021collapsing-k3}. The main difference is that
our pairs $(B(\vec{a}),R_\trop)$ vary in a PL manner, and so
define a polyhedral decomposition of $\opC$.

\medskip

The plan of the paper is as follows. In Section~\ref{sec:kulikov}, we
recall the definition of Kulikov models and discuss their connection
to integral-affine structures on $S^2$.
Using symplectic geometry, 
we state and prove the Monodromy Theorem, allowing
one to concretely compute the monodromy invariant of a
Kulikov degeneration.

In Section~\ref{sec:compactifications} we recall various
compactifications of moduli spaces as they apply to K3 surfaces of
degree~2, and prove some auxiliary results about
them. Section~\ref{sec:reflection-fan} lays out the combinatorics of
the Coxeter fan and the corresponding toroidal compactification
$\oF_2^\tor$ in detail, along with a semitoric compactification $\oF_2^\semi$.

In Section~\ref{sec:mirror-k3} we discuss a one-dimensional family of
K3 surfaces with Picard rank 19 that is mirror-symmetric to~$F_2$. For
a general surface in this family its nef cone is isomorphic to a
fundamental chamber of the Coxeter fan.

In Section~\ref{sec:ias-over-coxeter} we apply the general theory
of polarized $\ias$ to the case at hand, building the family of pairs
$(B(\vec{a}),R_\trop)$ over the Coxeter fan $\fF^\cox$. We interpret an
integral vector $\vec{a}$ in this fan as a combinatorial type of Kulikov
model of K3 surfaces with the monodromy vector~$\vec{a}$. 
In Section~\ref{sec:degenerate-k3s} we describe explicitly the resulting
stable models, in terms of the $ADE$ and $\wA\wD\wE$ surfaces of
\cite{alexeev17ade-surfaces}.

Finally, in Section~\ref{sec:family}, we prove Theorem~\ref{thm:family}.
Throughout, we work over~$\bC$.

\section{Kulikov models and $\ias$}
\label{sec:kulikov}

\subsection{Kulikov models and anticanonical pairs}\label{subsec:kulikov}
One of the first results about degenerations of K3 surfaces is the
well-known theorem of Kulikov and Persson-Pinkham
\cite{kulikov1977degenerations-of-k3-surfaces,
  persson1981degeneration-of-surfaces}.

\begin{theorem}\label{kpp-thm} Let $\mathcal{X}\rightarrow (C ,0)$ be a flat proper family
  over a germ of a curve such that the fibers of $\mathcal{X}^*\rightarrow
  C^*=C\setminus 0$
  are projective K3 surfaces. Then there is a finite ramified base change
  $(C',0) \to (C,0)$ and a birational modification
  $\cX'\to \cX\times_C C'$ such that $\pi\colon\cX'\to C'$ is
  semistable (a smooth threefold with $\cX_0'$ a reduced normal
  crossing divisor) with $\omega_{\cX'/C'}\simeq \cO_{\cX'}$.
\end{theorem}

Moreover, by Shepherd-Barron
\cite{shepherd-barron1981extending-polarizations}, for a relatively
nef line bundle $\cL^*$ on $\mathcal{X}^*\rightarrow C^*$ there is a model as above
to which $\cL^*$ extends as a nef line bundle $\cL$.
 
\begin{definition}\label{Kulikov-def} A degeneration $\mathcal{X}\rightarrow (C,0)$
  satisfying the conclusion of the theorem is a {\it Kulikov
    degeneration} and we call the central fiber a {\it Kulikov surface}. \end{definition}

Let $\log T$ be
the nilpotent logarithm of the unipotent Picard-Lefschetz transformation
$T\,:\,H^2(\mathcal{X}_t,\Z)\rightarrow H^2(\mathcal{X}_t,\Z)$. There
are three possible cases for the order of $\log T$, called Types
I, II, III:

\begin{enumerate}

\item[(I)] If $\log T=0$, then $\mathcal{X}_0$ is a smooth K3 surface. \vspace{2pt}

\item[(II)] If $(\log T)^2=0$ but $\log T\neq 0$, then
  $\mathcal{X}_0=\bigcup_{i=1}^n V_i$ is a chain of surfaces with
  dual complex a segment. The ends $V_1$ and $V_n$ are
  rational and $V_i$ for $i\neq 1,n$ are birational
  to $E\times \mathbb{P}^1$ for a fixed elliptic curve $E$. The
  double curves $D_{i,i+1}:=V_i\cap V_{i+1}$ are isomorphic to $E$,
  the union of the double curves lying on $V_i$ is an
  anticanonical divisor,
  and $$D_{i,i+1}\big{|}_{V_i}^2+D_{i,i+1}\big{|}_{V_{i+1}}^2=0.$$

\item[(III)] If $(\log T)^3=0$ but $(\log T)^2\neq 0$, then
  $\mathcal{X}_0=\bigcup_{i=1}^n V_i$ is a union of rational surfaces
  whose dual complex is a triangulation of the sphere. The union of
  all double curves $D_{ij}:=V_i\cap V_j$ lying on (the normalization of) $V_i$
  form an anticanonical cycle of rational curves. Declaring $D_{ij}
  \subset V_i$ and $D_{ji}\subset V_j$ and $d_{ij} := -2p_a(D_{ij})-D_{ij}^2$,
  we have $$d_{ij}+d_{ji}= -2.$$
\end{enumerate}

Note that $p_a(D_{ij})=0$ unless $D_{ij}\subset V_i$ is an anticanonical cycle of length $1$,
i.e. an irreducible nodal anticanonical divisor $D_{ij}\in |-K_{V_i}|$.

Every natural compactification of the moduli space
of K3 surfaces has strata of Types I, II, III, with Types II,
III on the boundary. The three cases are distinguished by the
property that for Type I, the central fiber is smooth, for Type II,
the central fiber has double curves but no triple points, and for Type
III the central fiber has triple points.

\begin{definition}\label{def:antican-pair}
  An {\it anticanonical pair} $(V,D)$ is a smooth rational surface $V$
  together with a cycle of smooth rational curves $D\in |-K_V|$.
\end{definition}

\begin{definition}
  Let $(V,D)$ be an anticanonical pair, with $D=D_1+\cdots +D_n$. The
  {\it charge} is $Q(V,D):=12- \sum (D_i^2+3).$
\end{definition}

\begin{definition} A {\it corner blow-up} of $(V,D)$ is the blow-up at
  a node of the cycle $D$ and an {\it internal blow-up} is a blow-up
  at a smooth point of $D$. In both cases, the blow-up has an
  anticanonical cycle mapping to $D$. The corner blow-up leaves the
  charge invariant, while the internal blow-up increases the charge by
  $1$.
\end{definition}

For the internal blow-up, the resulting anticanonical cycle is the strict transform
of $D$, whereas for the corner blow-up, it is the reduced inverse image of $D$.

\begin{remark} By \cite[2.7]{friedman2015on-the-geometry}
  the pair $(V,D)$ is {\it toric}, in the sense that $V$
  is toric and $D$ is the toric boundary, if and only if
  $Q(V,D)=0$. Otherwise, $Q(V,D)>0$. 
\end{remark}

We have the following proposition:

\begin{proposition}[Conservation of Charge]
  Let $\mathcal{X}\rightarrow (C,0)$ be a Type III Kulikov
  degeneration. Then
  $\sum_{i=1}^n Q(V_i,\textstyle\sum_j D_{ij})=24.$ In particular,
  at most $24$ components of $\mathcal{X}_0$ are
  non-toric.
\end{proposition}

\begin{proof} See Proposition 3.7 of Friedman-Miranda
  \cite{friedman1983smoothing-cusp}. 
\end{proof}

As we will see in Section \ref{subsec:pseudo}, this proposition presaged the
existence of an integral-affine structure on the dual complex
$\Gamma(\mathcal{X}_0)$ of the central fiber.

The {\it combinatorial type} of a Kulikov degeneration is
the combinatorial information of the simplicial complex $\Gamma(\cX_0)$,
together with the deformation type of each component $(V_i,\sum_j D_{ij})$,
which, in particular, determines (but is not always determined by) the collection
of integers $d_{ij}$.

The remaining data is continuous: One must choose a point in the
deformation space of anticanonical pairs for each component, and
choose how to glue double curves $D_{ij}$. These moduli are
parametrized by a torus $(\C^*)^N$ of some large dimension, but for
$\mathcal{X}_0$ to be smoothable we must choose the gluings and moduli
of $V_i$ carefully. A theorem of Friedman
\cite{friedman1983global-smoothings} states that
{\it
  $d$-semistability}
$$\mathcal{E}xt^1(\Omega^1_{\mathcal{X}_0},\mathcal{O}_{\mathcal{X}_0})=1\in
\textrm{Pic}^0((\mathcal{X}_0)_{\rm sing})\simeq (\C^*)^{\#\{V_i\}-1}
$$
is a necessary and sufficient condition for smoothability.

By \cite{friedman1986type-III}, the logarithm of monodromy in Types II
and III is given by
$$\log T\,:\,x\mapsto (x\cdot \delta)\lambda-(x\cdot \lambda)\delta$$
for elements $\delta,\lambda\in H^2(\mathcal{X}_t,\Z)$ satisfying
$\delta^2=\delta\cdot \lambda =0$, $\delta$ primitive, and
$\lambda^2=\#\{\textrm{triple points of }\mathcal{X}_0\}$. Thus
$\lambda^2=0$ if the degeneration is Type II. 

\begin{definition}\label{def:monodromy} Let $\mathcal{X}\rightarrow C$
  be a Type III degeneration. We call
  $\delta\in H^2(\mathcal{X}_t,\Z)$ the {\it vanishing cycle} and the
  vector $\lambda\in \delta^\perp/\delta$ the {\it monodromy
    invariant}. If the family $\mathcal{X}\rightarrow C$ is
  polarized by $L$, the vanishing cycle and monodromy
  invariant are defined similarly, but with reference to the ambient
  lattice $c_1(L)^\perp\subset H^2(\cX_t,\Z)$. \end{definition}

Any degeneration of K3 surfaces determines a primitive isotropic sublattice of
$H^2(\mathcal{X}_t,\Z)$ by taking a Kulikov model and
setting
\begin{align*}  I&=\Z \delta &&\textrm{if }\mathcal{X}\rightarrow
  C\textrm{ is Type III,} \\ 
  J&=(\Z
  \delta\oplus \Z\lambda)^{\rm sat} &&\textrm{if }\mathcal{X}\rightarrow C\textrm{ is Type
    II.}\end{align*}
    
\subsection{Integral-affine structures: general definitions} 

\begin{definition} An {\it integral-affine structure} on a real surface $S$
  is a collection of charts from $S$ to $\R^2$ such that the
  transition functions lie in $\SL_2(\Z)\ltimes
  \R^2$. \end{definition}
  
\begin{definition} 
  The {\it monodromy representation}
  $\rho\,:\,\pi_1(S,*)\rightarrow \SL_2(\Z)\ltimes \R^2$ is constructed
  by patching together charts along a loop $\gamma\in \pi_1(S,*)$
  in the unique way such that they glue on
  overlaps, then comparing the chart at the end of the loop with the
  one at the beginning. This process of patching charts together
  defines the {\it developing map} from the universal cover
  $\widetilde{S}\rightarrow \R^2$ which is equivariant with respect to
  $\rho$. Usually, we further project the monodromy to the group
  $\SL_2(\Z)$.

\end{definition}

As defined, the two-sphere admits no integral-affine structures. One
must introduce a reasonable class of singularities of such structures.

\begin{figure}[!h]
  \includegraphics[width=345pt]{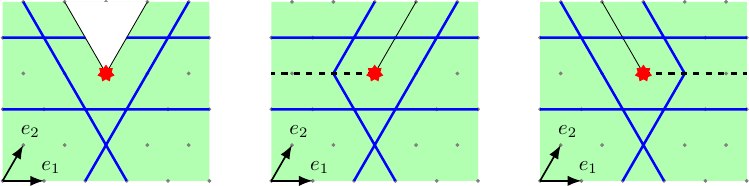} 
  \caption{Three representations of the $I_1$ singularity.}
  \label{fig:i1}
\end{figure}

\begin{definition}
  An {\it $I_1$ singularity} is the germ of a singular integral-affine
  surface isomorphic to the following basic example:

  Cut from $\R^2=\R e_1\oplus \R e_2$ the convex cone with the sides $\bR_{\ge0}e_2$
  and $\bR_{\ge0}(e_2-e_1)$, as on the left in Fig.~\ref{fig:i1}, and
  glue one boundary ray to another by a shear in the $e_1$-direction,
  i.e. by the rule $e_1\mapsto e_1$, $e_2\mapsto -e_1+e_2$.
\end{definition}

  Three straight lines in the affine structure are shown in
  bold blue.
  The second and third figures in Fig.~\ref{fig:i1} also represent the $I_1$
  singularity, with a dashed ray in the monodromy-invariant direction removed. The image
  of the developing map is $\R^2$ minus the ray. We can visualize this presentation
  as taking the standard affine structure on $\R^2$ minus the ray,
  then gluing across the ray by a shear.

\begin{remark}\label{rem:flip-gap}
  The $I_1$ singularity can be presented by removing any ray
  emanating from the singularity. When this ray is not in a
  monodromy-invariant direction, the two sides of the ray separate to
  produce a gap as in the left-hand figure. \end{remark}

\begin{definition}\label{intaffsing}
  Let $\vec{v}_1,\dots,\vec{v}_k$ be a sequence of primitive integral
  vectors, ordered cyclically counterclockwise around the
  origin. Define an {\it integral-affine singularity}
  $(S,p)=I(n_1\vec{v}_1,\dots,n_k\vec{v}_k)$ to be the result of
  shearing the affine structure of $\R^2$ a total of $n_i$ times along
  $\R_{\geq 0}\vec{v}_i$.
\end{definition}

  Let $M(\vec{v})$ be the unique matrix conjugate in
  $\SL_2(\Z)$ to $(e_1,e_2)\mapsto (e_1,e_1+e_2)$ such that $\vec{v} M(\vec{v})=\vec{v}$, 
  i.e. $M(\vec{v})$ is the unit shear along $\vec{v}$. Then, the
  $\SL_2(\Z)$ monodromy of a counterclockwise loop around the singularity $(S,p)$
  is the product $M(S,p)=M(\vec{v}_1)^{n_1} \cdots M(\vec{v}_k)^{n_k}.$

We can view $I(n_1\vec{v}_1,\dots,n_k\vec{v}_k)$ as the collision of $n_1+\cdots+n_k$
$I_1$ singularities, with monodromy invariant directions along the $\vec{v}_i$.

\begin{definition}  The {\it charge} of a singularity $(S,p)$ is the number $\sum_{i=1}^k n_i$
  of rays sheared to produce it, counted with multiplicity. For
  instance, the $I_1$ singularity $I(\vec{v})$ has charge one. \end{definition}

\begin{definition} An integral-affine sphere $B$, or $\ias$ for short,
  is a sphere $B=S^2$ and a finite set $\{p_1,\dots,p_n\}\in B$ such
  that $B\setminus\{p_1,\dots,p_n\}$ has a non-singular
  integral-affine structure, and a neighborhood of each $p_i$ is
  modeled by some integral-affine singularity $I(n_1\vec{v}_1,\dots,
  n_k\vec{v}_k)$.
\end{definition}

\begin{proposition}
  Let $B$ be an integral-affine structure with singularities on a
  compact oriented surface of genus $g$. Then, the sum of the charges
  is $12(2-2g)$.
\end{proposition}

\begin{proof} See \cite{kontsevich2006affine-structures} or
  \cite{engel2018moduli-space}.
\end{proof}

\begin{remark}
  The shearing directions $\vec{v}_i$ used to construct each
  singularity form part of the definition of $B$. Thus, two $\ias$ be
  may not be isomorphic even if there is a homeomorphism
  $B_1\rightarrow B_2$ which is an integral-affine isomorphism away
  from the singular sets. We discuss the appropriate equivalence
  relation below.
\end{remark}

\begin{definition}\label{def:Ipqr}
  Let $\vec u, \vec v, \vec w\in\bZ^2$ be three vectors so that
  $(\vec u, \vec v)$ form an oriented basis and $\vec u+ \vec v + \vec w=0$. As a
  further shortcut, we define $I(p) = I(p\vec u)$, called an {\it $I_p$ singularity}.
  Let $I(p,q) = I(p\vec u, q\vec v)$, and
  $I(p,q,r) = I(p\vec u, q\vec v,r \vec w)$. Up to the action of $\SL_2(\Z)$, this
  notation is symmetric under cyclic rotations.  Finally, we set
  $I(p,q,r,s) = I(p\vec u, q\vec v, r(-\vec u), s(-\vec v))$, also symmetric
  up to cyclic rotation.\end{definition}

\subsection{Pseudo-fans and Kulikov models}\label{subsec:pseudo}
In this section we describe how to encode a deformation type of
anticanonical pairs as an integral-affine surface singularity, and in
turn how to encode a Type III Kulikov model as an $\ias$. 

\begin{definition}\label{def:pseudofan}
  The {\it pseudo-fan} of an anticanonical pair $\mathfrak{F}(V,D)$,
  see \cite[Sec.1.2]{gross2015mirror-symmetry-for-log} or \cite[Def.3.8]{engel2018looijenga}, is a triangulated
  integral-affine surface with boundary constructed as follows:
  
  As a
  PL surface, $\mathfrak{F}(V,D)$ is the cone over the dual complex of
  $D$. The affine structure on each triangle in this cone is declared
  integral-affine equivalent to a lattice triangle of lattice
  volume~$1$. Two adjacent triangles are glued by the following rule:
  Let $\vec{e}_j$ be the directed edge of $\mathfrak{F}(V,D)$
  emanating from the cone point and pointing towards the vertex
  corresponding to $D_j$. In a chart containing the union of the two
  adjacent triangles containing $\vec{e}_j$ we have
  $\vec{e}_{j-1}+\vec{e}_{j+1}=d_j\vec{e}_j$, where $d_j=-D_j^2$ if
  $D_j$ is smooth and $d_j=-D_j^2+2$ if $D_j$ is a rational nodal
  curve.
\end{definition}

\begin{remark}
  When $(V,D)$ is a toric pair, the pseudo-fan $\mathfrak{F}(V,D)$ is
  a non-singular integral-affine surface with a single chart to a
  polygon in $\R^2$. The vertices of this polygon are the endpoints of
  the primitive integral vectors pointing along the $1$-dimensional
  rays of the fan of $(V,D)$.
\end{remark}

\begin{remark}\label{allpseudo}
  A {\it toric model}
  $\pi:(V,D) \rightarrow (\overline{V},\overline{D})$ is a blow-down
  to a toric pair. After some corner blow-ups, every anticanonical
  pair admits a toric model, see
  \cite[Prop.~1.3]{gross2015moduli-of-surfaces}. Assume that $\pi$
  consists only of internal blow-ups, as corner blow-ups don't affect
  toricity. Then \cite[Prop.~3.13]{engel2018looijenga} implies
  $\mathfrak{F}(V,D)$ is the result of shearing along the rays of the
  fan of $(\overline{V},\overline{D})$ corresponding to components
  which get blown up. Hence, by Definition \ref{intaffsing}, every
  integral-affine surface singularity is the cone point of the
  pseudo-fan of some anticanonical pair, and by subdividing the
  singularity into standard affine cones, the converse is also true.
  
  Furthermore, the charge $Q(V,D)$ coincides with the charge of
  the corresponding singularity $\mathfrak{F}(V,D)$. It is the
  number of internal blow-ups of the toric model.
\end{remark}

Let $\mathcal{X}\rightarrow C$ be a Type III degeneration. We label
the vertices of the dual complex $\Gamma(\mathcal{X}_0)$ by $v_i$, the edges by
$e_{ij}$, and the triangles by $t_{ijk}$, corresponding respectively
to the components, double curves, and triple points of
$\mathcal{X}_0$. Let $\Star(v_i)$ be the union of the triangles
containing~$v_i$.

\begin{proposition}\label{spheretok3}
  The dual complex $\Gamma(\mathcal{X}_0)$ of a Type III degeneration
  of K3 surfaces admits a natural integral-affine structure such
  that $$\Star(v_i)=\mathfrak{F}(V_i,\,\textstyle \sum_j D_{ij}).$$
  Conversely, given an integral-affine structure $B$ on the two-sphere
  with a triangulation into lattice triangles of lattice volume
  $1$ and singularities at the vertices, there is a Type III degeneration
  $\mathcal{X}\rightarrow C$ such that
  $\Gamma(\mathcal{X}_0)=B$.
\end{proposition}

Here {\it lattice volume} means twice the Euclidean area.

\begin{proof}
  See \cite{engel2018looijenga} or \cite[Rem1.11v1]{gross2015mirror-symmetry-for-log}.
  The key point is that the pseudo-fans of the components compatibly glue
  to form a well-defined integral affine structure on any quadrilateral formed from
  two adjacent triangles of $\Gamma(\mathcal{X}_0)$. 
  This follows from the formula $d_{ij}+d_{ji}=2$ in (III), below
  Definition \ref{Kulikov-def}.
  \end{proof}

\begin{definition}\label{cornereq}
  Two anticanonical pairs $(V_1,D_1)$ and $(V_2,D_2)$ lie in the same
  {\it corner blow-up equivalence class} (c.b.e.c.) if they are related by
  a sequence corner blow-ups and blow-downs, 
  and a topologically trivial deformation.
  A {\it toric model} of a c.b.e.c. is a representative $(V,D)$ of the equivalence
  class, and a toric model
  $$(V,D)\rightarrow
  (\overline{V},\overline{D}).$$
\end{definition}

Note that all topologically trivial deformations of $(V,D)$ are the result
of deforming the points on $\overline{D}$ which are blown up.

By Remark \ref{allpseudo}, a toric model of an anticanonical pair
$(V,D)$ determines an integral-affine singularity at the cone point of
$\mathfrak{F}(V,D)$. Corner blow-ups subdivide the pseudo-fan,
which do not affect the singularity. Neither do topologically trivial deformations.
We conclude that there is a bijection between presentations
$I(n_1\vec{v}_1,\dots,n_k\vec{v}_k)$ of integral-affine
singularities by shears and toric models of c.b.e.c.s.
We now forget the dependence on the toric model:

\begin{definition}\label{cluster} Two integral-affine singularities
  are {\it equivalent}
  $$(S_1,p_1)=I(n_1\vec{v}_1,\dots,n_k\vec{v}_k)\sim
  I(m_1\vec{w}_1,\dots,m_\ell\vec{w}_\ell)=(S_2,p_2)$$ if 
  the corresponding c.b.e.c.s are equal $[(V_1,D_1)]=[(V_2,D_2)]$.
\end{definition} 

By choosing a single anticanonical pair $(V,D)$ which admits both toric
models corresponding to $\vec{v}_i$ and to $\vec{w}_j$, and building
$\mathfrak{F}(V,D)$ by the recipe in Definition \ref{def:pseudofan}
(which does not use a toric model), an equivalence of integral-affine
singularities provides a homeomorphism
$(S_1,p_1)\rightarrow (S_2,p_2)$ which is an integral-affine
isomorphism away from the $p_i$. But the converse is false, see
Example 4.13 of \cite{engel2021smoothings}. Such examples
explain why it does not suffice to define an integral-affine
singularity as purely a geometric structure---the presentation via
shears (at least up to equivalence) is part of the definition.

\begin{remark}\label{cluster-mutation} Each toric model of the c.b.e.c. of $(V,D)$ defines a
  Zariski open subset of the open Calabi-Yau
  $(\C^*)^2\hookrightarrow V\backslash D.$ One may choose a
  different toric model by changing exactly one exceptional curve $E$
  blown down in the toric model---to a curve $F$ such that $E+F$ is
  the fiber of a toric ruling. The change-of-coordinates to the new
  inclusion $(\C^*)^2\hookrightarrow V\backslash D$ is a birational
  map called a {\it cluster mutation}. 
  It is almost always the case that there are infinitely many such cluster charts.
 Any two toric models of a c.b.e.c. are connected by a series of cluster mutations,
 by a theorem of Blanc \cite{blanc2013symplectic}.
\end{remark}

\begin{example}\label{exa:equiv-ia-sings}
  Start with the toric pair $(\bP^2, L_1+L_2+L_3)$ and make a
  corner blowup to get $(\bF_1, s_0+f+s_\infty+f)$, with $s_0^2=1$,
  $s_\infty^2=-1$. Blow up one point on $s_0$ then contract one exceptional
  curve intersecting $s_\infty$ to obtain $\bP^1\times\bP^1$.
  This corresponds to a single cluster mutation as in Remark \ref{cluster-mutation}.
  We may
  also blow up $p$ points on the first copy of $f$, $q$ points on
  the second copy of $f$, and one more point on $s_0$.
  In this way, we see the equivalences
  \begin{displaymath}
    I(2,p,q) \sim I(1,p,1,q) \sim I(p,1,q,1) \sim I(2,q,p).
  \end{displaymath}
\end{example}

\subsection{Birational modifications and base change}\label{sec:bir-mod}

All Kulikov models $\cX\to (C,0)$ completing a punctured
family $\cX^*\to C^*$ are related by flops along smooth rational curves.
The modifications which change the isomorphism type of $\cX_0$ are:

\begin{enumerate}
\item {\it M1 modifications} are Atiyah flops along an exceptional curve $E\subset V_i$
meeting a double curve $D_{ij}$ at a single point $p$. The effect on $\cX_0$ is to blow down $E$
on $V_i$ and blow up $V_j$ at $p$. \vspace{2pt}
\item {\it M2 modifications} are Atiyah flops along an exceptional double curve $E=D_{ij}=V_i\cap V_j$.
The effect on $\cX_0$ is to blow down $E$ on both $V_i$ and $V_j$, blow up the two
triple points $T_{ijk}$ and $T_{ij\ell}$ contained in $E$, on the components $V_k$ and $V_l$,
and then glue the resulting exceptional curves.
\end{enumerate}

\begin{definition}\label{nodal-slide} Let $(S,p)=I(n_1\vec{v}_1,\dots,n_k\vec{v}_k)$ be an integral-affine singularity.
A {\it nodal slide along $\vec{v_i}$ of length $t$}, cf. \cite[Def.~6.1]{symington2003four-dimensions},
is a surgery on the integral affine structure $(S,p)$
which translates by $t\vec{v}_i$ the originating point of one shearing ray in the direction $\vec{v}_i$.
\end{definition}

  Note that nodal slides are called {\it moving worms} in the mirror
symmetry literature, see e.g. \cite{kontsevich2006affine-structures} or
\cite{gross2015mirror-symmetry-for-log}. 

Starting with the single singularity $(S,p)$, the nodal slide results in an integral-affine
surface with two singularities $I(n_1\vec{v}_1,\dots,(n_i-1)\vec{v}_i,\dots, n_k\vec{v}_k)$
and an $I_1$ singularity at the endpoint of $t\vec{v}_i$. The result
is an integral-affine surface which is isomorphic to the original one
on the complement of the segment $t\vec{v}_i$. Thus, the operation is purely
local and can be done independently of the rest of the integral affine surface.
For appropriately large $t$, a nodal slide may result in the $I_1$
singularity sent off colliding into another singularity.

In fact, any integral-affine singularity can be defined as the result
of colliding a collection of $I_1$ singularities moving along nodal slides. 

\begin{proposition}\label{nodal-slide-dual}
\cite[Prop.~4.5, 4.6]{engel2021smoothings} An M2 modification does not change 
the $\ias$ structure on $B=\Gamma(\cX_0)$, but retriangulates $B$ by cutting along the opposite
diagonal of an integral-affine unit square.

An M1 modification preserves the triangulation of $\Gamma(\cX_0)$ but changes 
the $\ias$ $B$ by a unit length nodal slide,
moving an $I_1$ singularity along $\vec{e}_{ij}$ from $v_i$ to $v_j$.
\end{proposition}

Thus, a sequence of M1 and M2 modifications connecting two Kulikov
surfaces $\cX_0\dashrightarrow \cX_0'$
are modeled as a sequence
of retriangulations and integer length nodal slides
$\Gamma(\cX_0)\dashrightarrow \Gamma(\cX_0')$ 
on the corresponding dual complexes.

\begin{proposition} \cite{Friedman1983base-change}\label{refine-base-change} Let $\cX\to (C,0)$
be a Kulikov model, and consider the base change $\cX'\to (C',0)$ ramified
over $0$ to order $N$. There is a standard resolution $\cX[N]\to \cX'$, 
producing a new Kulikov model whose central fiber $\mathcal{X}_0[N]$
is the result of inserting ``special bands of hexagons" of width $N$ between
all the components of $\cX_0$. The effect on the dual complex $\Gamma(\cX_0)$
is to take the standard refinement every triangle into $N^2$ triangles
(see also \ref{cla:kulikov-avoid-strata} below). \end{proposition}

In fact, the integral-affine structure on the dual complex $B[N]:=\Gamma(\mathcal{X}_0(N))$
is the result of post-composing the integral-affine charts $U\to \R^2$ on $B=\Gamma(\cX_0)$
with multiplication by $N$, cf. \cite[Prop.~4.3]{engel2021smoothings}. We call this the {\it order $N$
refinement} of $B$. Note that the base change
multiplies the monodromy invariant $\lambda\mapsto N\lambda$.

\subsection{Integral-affine divisors}

In this section, we define an integral-affine divisor on an
$\ias$. For motivation, consider a line bundle
$\mathcal{L}\rightarrow \mathcal{X}$ on a Kulikov model.  Let
$L_i:=\mathcal{L}\big{|}_{V_i}\in \textrm{Pic}(V_i)$. These line
bundles automatically satisfy a compatibility
condition
$L_i\cdot D_{ij}= L_j\cdot D_{ji}.$
Thus, we define:

\begin{definition}\label{IAdivisor}
  Let $B$ be an $\ias$. An {\it integral-affine divisor} $R_\ia$ on
  $B$ consists of two pieces of data:
  \begin{enumerate}
  \item A weighted graph $R_\ia\subset B$ with vertices $v_i$,
    straight line segments as edges $e_{ij}$, and integer labels
    $n_{ij}$ on each edge. 

  \item Let $v_i\in R_\ia$ be a vertex and $(V_i ,D_i)$ be an
    anticanonical pair such that $\mathfrak{F}(V_i ,D_i)$ models $v_i$
    and contains all edges of $e_{ij}$ coming into $v_i$. We require
    the data of a line bundle $L_i \in \textrm{Pic}(V_i)$ such that
    $\deg{L_i}\cdot D_{ij}=n_{ij}$ for the components $D_{ij}$ of
    $D_i$ corresponding to edges $e_{ij}$ and $L_i$ has degree zero on
    all other components of $D_i$.
  \end{enumerate}
\end{definition} 

\begin{definition}
  Given a line bundle $\mathcal{L}\rightarrow \mathcal{X}$ on a
  Kulikov degeneration, the intersection numbers $n_{ij}=L_i\cdot D_{ij}$
  define an integral-affine divisor
  $R_\ia\subset B=\Gamma(\mathcal{X}_0)$ supported on the $1$-skeleton.
  If $\cL$ is nef then $R_\ia$ is effective i.e. $n_{ij}\geq 0$.
\end{definition}

\begin{remark}\label{balancing1}
  When $v_i\in R_\ia$ is non-singular, the pair $(V_i,D_i)$ is toric,
  and the labels $n_{ij}$ uniquely determine $L_i$. They must satisfy
  a balancing condition. If $v_{ij}$ are the primitive integral
  vectors in the directions $e_{ij}$ then one must have
  $\sum n_{ij}v_{ij}=0$ for such a line bundle $L_i\to V_i$ to exist.

  Similarly, if $I_1=\mathfrak{F}(V_i,D_i)=I(\vec{v})$
  i.e. $(V_i,D_i)$ is the result of a single internal blow-up of a
  toric pair, the $n_{ij}$ determine a unique line bundle $L_i$ so
  long as $\sum n_{ij}v_{ij} \in \Z\vec{v}$. This condition is
  well-defined as the $v_{ij}$ are well-defined up to shears in the
  $\vec{v}$ direction.
\end{remark}

\begin{definition}\label{def:polarizing-ia-div}
  We say that a divisor on $B$ is {\it polarizing} if each line bundle
  $L_i$ is nef and at least one $L_i$ is big. The {\it
    self-intersection} of an integral-affine divisor is
  $R_\ia^2:=\sum_i L_i^2\in \Z.$
\end{definition}

\begin{definition}
  An $\ias$ is {\it generic} if it has $24$ distinct $I_1$
  singularities.
\end{definition}

\begin{remark}\label{generic} Let $B$ be a lattice triangulated $\ias$ or equivalently,
$B=\Gamma(\mathcal{X}_0)$ is the dual complex of a Type III
degeneration. Then $B$ is generic if and only if
$Q(V_i,D_i)\in\{0,1\}$ for all components $V_i\subset
\mathcal{X}_0$. When $B$ is generic, an integral-affine divisor
$R_\ia\subset B$ is uniquely specified by a weighted graph satisfying
the balancing conditions of Remark \ref{balancing1}, so the extra
data (2) of Definition \ref{IAdivisor} is unnecessary. \end{remark}

\begin{definition}
  An integral-affine divisor $R_\ia\subset B$ is {\it compatible} with
  a triangulation if every edge of $R_\ia$ is formed from edges of the
  triangulation.
\end{definition}

If $B$ comes with a triangulation, we require the integral-affine
divisor to be compatible with it.

\subsection{Integral-affine structures from Lagrangian torus fibrations}
The reference for this section is Symington
\cite{symington2003four-dimensions}.  Let $(S,\omega)$ be a smooth
symplectic $4$-manifold. Given a Lagrangian torus fibration
$\mu\,:\,(S,\omega)\rightarrow B$
with only nodal singularities, the base $B$ inherits a natural
integral-affine structure with an $I_n$ singularity under a necklace
of $n$ two-spheres:

\begin{definition} Let $C_\alpha$ and $C_\beta$ be cylinders in $S$
fibering over a path from a fixed base point $*\in B$ to a point
$p\in B$, such that the ends of the cylinders over $*$ are
homologous to $\alpha$ and $\beta$, an oriented basis of
$H_1(S_*,\Z)$. The {\it induced integral affine structure} on $B$ 
is the collection of charts of the form
$$p\mapsto (x(p),y(p))=\left( \textstyle
  \int_{C_\alpha} \omega, \,\int_{C_\beta} \omega\right)\in \R^2.$$
  \end{definition}
  
These charts are only defined up to monodromy in
$\SL(2,\bZ)\ltimes \R^2$, by choosing a path in a different homotopy class
and moving the base point $*$.

Let $\oT$ be a complex toric surface, $\oL\in\Pic(\oT)\otimes\bR$ an
ample class, and $\oomega$ a symplectic form with $[\oomega]=\oL$. The
moment map
$\mu_\oT\,:\, (\oT,\ov{\omega}) \to \oP$ is a Lagrangian torus
fibration which induces the integral-affine structure on the moment
polytope~$\oP$ coming from its embedding into $\R^2$.
It degenerates over the toric boundary $\oD\subset \oT$,
and sends the components of $\oD$ to the boundary components of $\oP$.

Now let $\phi\colon T\to\oT$ be a blowup at a smooth point of the
boundary $\oD$, with exceptional divisor $E$. Symington
\cite{symington2003four-dimensions} constructed a Lagrangian torus
fibration $\mu_T\,:\, (T,\omega)\to P$ satisfying
$[\omega]=\phi^*[\oomega]-aE$ over a singular integral-affine disk
 $P$ (a ``Symington polytope'') obtained as follows:

\begin{definition} A {\it Symington surgery} is the result of cutting a triangle of lattice
size $a$ (and lattice volume $a^2$) from the side of the moment polytope $\oP$ corresponding to
the component blown up, then gluing the two remaining edges,
introducing an $I_1$ singularity $p\in P$ at the interior corner of
the triangle. \end{definition}

The fiber over $p$ is an irreducible nodal $I_1$ fiber of
the torus fibration. In symplectic geometry, this procedure is called
an \emph{almost toric blowup.} The monodromy axis of the singularity
is parallel to the side of $\oP$ on which the surgery triangle rests and the location of
the cut on the side of $\oP$ is essentially arbitrary.

\begin{construction}\label{visibleconstruction}
  Let $B$ be a generic $\ias$ and let $B^o=B\setminus\{p_1,\dots,p_{24}\}$
  be its nonsingular locus. Let $\gamma =\sum_i (\gamma_i, \alpha_i) \subset B$ be a $1$-chain with
  values in the constructible sheaf
  \begin{math}
    T_\Z:= i_*(T_\Z B^o),  
  \end{math}
  where $i\,:\,B^o\hookrightarrow B$ is the
  inclusion. This sheaf is a $\Z$-local system of rank $2$ on
  $B^o$ and has rank $1$ at the $I_1$ singularities.
  
  Concretely, $\gamma$ is a collection of oriented paths
  $\gamma_i\rightarrow B$ and a (constant) integral vector field
  $\alpha_i$ on each path. There is a boundary map $\partial$ to
  $0$-chains with values in $T_\Z$ gotten by taking an oriented sum
  of the tangent vectors $\alpha_i$ at the endpoints of $\gamma_i$. We
  say that $\gamma$ is a {\it $1$-cycle} if $\partial\gamma=0$. Some care
  must be taken at the singularities, where the rank of $T_\Z$
  drops. Here the condition that the boundary is zero means that
  $\sum \alpha_i$ is parallel to the monodromy-invariant direction of
  the singularity.

From such a $1$-cycle $\gamma$, we may construct a PL surface
$\Sigma_\gamma\subset S$ inside the symplectic $4$-manifold with a
Lagrangian torus fibration $\mu\,:\,(S,\omega)\rightarrow B$. We take
a cylinder in $S$ which maps to $\gamma_i$ whose fibers are the
circles in the torus fiber that correspond to $\alpha_i$ via the
symplectic form. The condition that $\partial \gamma=0$ is exactly the
condition that the ends of these cylinders over the points in
$\cup_i\,\partial \gamma_i$ are null-homologous in the fiber. Thus, we
may glue in a (Lagrangian) $2$-chain contained in the fiber over
$\cup_i\,\partial \gamma_i$ and produce a closed PL surface
$\Sigma_\gamma$.
\end{construction}

\begin{definition}
  The surfaces $\Sigma_\gamma$ constructed as above are the {\it
    visible surfaces}.
\end{definition}

\begin{example}\label{minus-two-visible}
Given a path $\gamma$ connecting two $I_1$ singularities $p$ and $q$,
such that the monodromy-invariant directions at both $p$ and $q$ are parallel to $\alpha$,
the $1$-cycle $(\gamma,\alpha)$ defines a visible surface, which we denote
$E_{(\gamma,\alpha)}$. It satisfies $E_{(\gamma,\alpha)}^2=-2$ because
$E_{(\gamma,\alpha)}$ is attached to each nodal fiber $S_p$, $S_q$ by 
a $(-1)$-framed $2$-handle. \end{example}

Note that $\Sigma_\gamma$ is non-canonical even on the level of its
homology class: There are many choices of Lagrangian $2$-chains in the
fibers over $\cup_i\,\partial \gamma_i$. But, they all differ by some
multiple of the fiber class $f=[\mu^{-1}(p)]$. Note that also
$[\Sigma_\gamma]\cdot f=0$.  We do have a well-defined class
$[\Sigma_\gamma]\in f^\perp/f$.

We note an important special case of the above construction.

\begin{definition} Suppose
that all $\gamma_i$'s are straight line segments $e_{ij}$ forming a
graph in the integral-affine structure on $B$, and that the
tangent vector field is an integer multiple $n_{ij}$ of the primitive
integral tangent vector along $\gamma_i$. Then the cylinder lying over
$e_{ij}$ can be made Lagrangian and the
surface $\Sigma_\gamma$ is a PL Lagrangian surface in $(S,\omega)$.
We call the result a {\it Lagrangian visible surface}.
\end{definition}

 In particular, the class of a Lagrangian visible surface
 satisfies $[\Sigma_\gamma]\cdot [\omega]=0$.
Observe that the condition that $\gamma$ is a
$1$-cycle is exactly the balancing condition of Remark
\ref{balancing1}. Thus an integral-affine divisor $R$ on $B$ in the sense
of Definition \ref{IAdivisor} corresponds to a Lagrangian visible
surface $\Sigma_R$.

\subsection{The Monodromy Theorem}\label{sec:monodromy}

Our goal now is to understand the vanishing cycle $\delta$, monodromy
invariant $\lambda$, and polarization of a Kulikov degeneration
$\mathcal{X}\rightarrow C$, see Definition \ref{def:monodromy}, in terms of
$\ias$ and symplectic geometry. We now prove a version of Proposition
3.14 of \cite{engel2021smoothings}, the key new ingredient being the
presence of a polarizing divisor $R$.

\begin{theorem}\label{thm:monodromy}
  Let $B$ be a generic $\ias$, together with a triangulation into
  lattice triangles of lattice volume~$1$. 
  \begin{enumerate}
  \item Let $\cX\to C$ be a type III Kulikov degeneration such that
    $\Gamma(\mathcal{X}_0)=B$.
  \item Let $\mu\colon(S,\omega)\to B$ be a Lagrangian torus fibration
    over the same $B$.
  \end{enumerate}
  Then there exists a diffeomorphism $\phi\colon S\to \cX_t$ to a
  nearby fiber $t\ne0$ such that
  \begin{enumerate}\renewcommand{\theenumi}{\alph{enumi}}
  \item $\phi_*f = \delta$, 
  \item $\phi_*[\omega] = \lambda$ in $\delta^\perp/\delta\otimes \R$.\end{enumerate}
  Moreover, suppose that $\mathcal{L}\rightarrow \mathcal{X}$ is a
  line bundle, which defines the integral-affine divisor $R_{\ia}$ on
  $B$. Let $\Sigma_{R_\ia}$ be the corresponding Lagrangian visible surface
  in $S$. Then, we have
  \begin{enumerate}
  \item[(c)] $\phi_*[\Sigma_{R_\ia}]=c_1(\mathcal{L}_t)$ in $\delta^\perp/\delta$.
  \end{enumerate}
  \end{theorem}

\begin{proof} We first prove (a) and (b) following Proposition 3.14 of
\cite{engel2021smoothings} closely. There, an almost exactly analogous
statement is proved for Type III degenerations of anticanonical pairs,
so we only describe the minor modifications necessary. We ignore the parts
of the proof in \cite{engel2021smoothings} which refer to $D$, and
similarly the special component of $\mathcal{X}_0$ equal to the
hyperbolic Inoue surface, instead treating all surfaces $V_i\subset
\mathcal{X}_0$ on equal footing. Then, the construction of $\phi$
proceeds the same way, by using the Clemens collapse to show that
$(S,\omega)$ and $\mathcal{X}_t$ can be written as the same fiber
connect-sum of $2$-torus fibrations over the intersection complex
$\Gamma(\mathcal{X}_0)^\vee$. Statement (a) follows immediately.

Again following \cite{engel2021smoothings}, we consider the collection
of Lagrangian visible surfaces $\Sigma_\gamma$ which fiber over the
$1$-skeleton $\Gamma(\mathcal{X}_0)^{[1]}$. The images under $\phi_*$
of the classes $[\Sigma_\gamma]$ generate a $19$-dimensional lattice
in $\delta^\perp/\delta$ invariant under the Picard-Lefschetz
transformation $H^2(\mathcal{X}_t;\Z)\rightarrow
H^2(\mathcal{X}_t,\Z)$. Since $[\omega]\cdot [\Sigma_\gamma]=0$, we
conclude that the monodromy invariant $\lambda$ and $\phi_*[\omega]$
are proportional in $\delta^\perp/\delta$. By
\cite{friedman1986type-III}, $$\lambda^2=\#\{\textrm{triple points of
}\mathcal{X}_0\} = \vol(\Gamma(\mathcal{X}_0))=[\omega]^2.$$ We
conclude that $\lambda= \phi_*[\omega]\textrm{ mod }\Z\delta$,
i.e. (b).

Now suppose that $\mathcal{X}$ admits a line bundle
$\mathcal{L}$. There is an integral-affine divisor $R$ on
$\Gamma(\mathcal{X}_0)$ whose defining line bundles
$L_i\in \textrm{Pic}(V_i)$ are $\mathcal{L}\big{|}_{V_i}$. Since
$\Gamma(\mathcal{X}_0)$ is generic, these line bundles are uniquely
determined by the integer weights $n_{ij}= L_i\cdot D_{ij}$ on the edges of
$\Gamma(\mathcal{X}_0)^{[1]}$. By construction, the Lagrangian visible
surface $\Sigma_{R_\ia}\subset S$ fibering over the weighted balanced graph
$R_\ia$ is sent by $\phi$ to a surface whose Clemens collapse is a union
of surfaces $\Sigma_i\subset V_i$ satisfying
\begin{enumerate} 
\item $\Sigma_i\cap D_{ij} = \Sigma_j\cap D_{ji}$
\item $\Sigma_i\cdot D_{ij} =L_i\cdot D_{ij}$.
\end{enumerate} These conditions uniquely determine the class $\phi_*[\Sigma_{R_\ia}]$. We conclude (c).  \end{proof}

\begin{remark} 
  Statements similar to Theorem \ref{thm:monodromy} (a) and (b) have
  appeared in the mirror symmetry literature. For instance, Theorem
  5.1 of \cite{gross2010mirror-symmetry} computes the monodromy of the
  Picard-Lefschetz transform of a toric degeneration of Calabi-Yau
  varieties in terms of cup product with the {\it radiance
    obstruction} $$c_B\in H^1(B,i_*(T_\Z B^o)),$$ a
  cohomology class canonically associated to an integral-affine
  structure, first studied in \cite{goldman1984radiance-obstruction}. The class $c_B$
  is identified with $[\omega]$ via the Leray spectral sequence of the
  map $\mu\,:\,(S,\omega)\rightarrow B$.  These monodromy formulas
  verify the prediction of topological SYZ mirror symmetry that the
  Picard-Lefschetz transformation is cup product with a section of the
  SYZ fibration.
  See also \cite[Cor. 4.24]{odaka2021collapsing-k3}.
\end{remark}

\section{Compactifications of $F_2$}
\label{sec:compactifications}

We first recall the basics about the moduli spaces of K3 surfaces as
they apply to the degree 2 case. For the Baily-Borel and toroidal
compactifications, a convenient reference is
\cite{scattone1987on-the-compactification-of-moduli}.  Then we
describe a compactification via stable pairs and prove some auxiliary
results about it.

\subsection{Period domain and moduli space}
\label{subsec:periods}

Let $\Lambda_{\rm K3} \simeq H^3\oplus E_8^2$ be a fixed lattice of
signature $(3,19)$ isomorphic to $H^2(S,\bZ)$ for a K3 surface $S$. Here,
$H$ is the hyperbolic plane, and the lattice $E_8$ for convenience is
negative definite. All primitive vectors of square $h^2=2d$ lie in the
same orbit of the isometry group of $\Lambda_{\rm K3}$. The lattice
$h^\perp$ is isometric to $H^2\oplus E_8^2\oplus \la -2d\ra$.  The
period domain for the polarized K3 surfaces of degree $2d$ is a connected
component of
\begin{displaymath}
  \bD = \mathbb{D}_{2d}:= \mathbb{P}
  \big\{ x\in h^\perp \otimes \C\,\big{|}\,
  x\cdot x =0,\,x \cdot \overline{x}>0\big\},
\end{displaymath}
a Hermitian symmetric domain associated to the group
$O^+(2,19)$. On it, we have the action of the group $\Gamma =
\Gamma_{2d}$ which is the spinor norm $1$ subgroup of the
stabilizer of $h$ in the isometry group
$O(\Lambda_{\rm K3})$. By the Torelli theorem, the quotient space $F_{2d} =
\Gamma\backslash\bD$ is the coarse moduli space of polarized K3
surfaces $(X,L)$, where $X$ is a K3 surface with ADE (Du Val) singularities,
and $L$ is an ample line bundle with $L^2=2d$. 
One has $\dim F_{2d} = \dim \bD_{2d} = 19$.

The moduli stack $\cF_{2d}$ of polarized K3 surfaces of degree $2d$
is a smooth DM stack. This stack and its coarse moduli space $F_{2d}$ are incomplete,
and $F_{2d}$ is quasiprojective.

\subsection{Baily-Borel compactification}
\label{sec:baily-borel}

Let $\mathbb{D}^\vee$ denote the compact dual of $\mathbb{D}$---it is
the quadric defined by dropping the condition $x\cdot \overline{x}>0$. Let
$\overline{\mathbb{D}}\subset \mathbb{D}^\vee$ be the topological closure.
Let $I$ be a primitive isotropic sub-lattice of $h^\perp$. Then $I$
has rank one or two. One calls the former Type III and the latter
Type II. The {\it boundary component} associated to $I$ is by
definition
$$F_I:= \mathbb{P}\{x\in \overline{\bD} \,\big{|}\,\textrm{span}\{\textrm{Re}(x),\,\textrm{Im}(x)\}=
I\otimes \R\}\subset \mathbb{D}^\vee$$ which is either a $0$-cusp, a
point for Type III or a $1$-cusp, a copy of
$\mathbb{H}$ for Type II ($\bH$ is the
upper-half plane).

\begin{notation} To distinguish the ranks, we henceforth use $I$ or $J$ for rank $1$ or $2$
primitive isotropic lattices, respectively. \end{notation}

Then, the Baily-Borel compactification is, topologically,
$$\overline{F}_{2d}\ubb:=\Gamma \backslash (\mathbb{D}\cup_J F_J\cup_I F_I).$$
In $\oF_2\ubb$, the boundary consists of four curves, meeting
at a single point, see \cite{scattone1987on-the-compactification-of-moduli}.
The point is the Type III boundary while the
curves (minus the point) are the Type II boundary. The curves
correspond to four distinct orbits of rank $2$ primitive isotropic
sublattices $J\subset h^\perp$.
For each of them, $J^\perp/J$ contains a finite index root sublattice,
which can be used as a label for this 1-cusp:
$$A_{17},\,\, D_{10}\oplus E_7, \,\,E_8^2\oplus A_1
,\,\textrm{ and } D_{16}\oplus A_1.$$

The stabilizer $\textrm{Stab}_\Gamma(J)\subset \Gamma$ acts on $J \simeq \Z^2$
by a finite index subgroup $\Gamma_J\subset \SL_2(\Z)$ and the boundary
component $ \Gamma_J\backslash F_J$ is a modular curve corresponding
to a Type II boundary curve. One has a natural finite morphism
$\Gamma_J\backslash F_J\to \SL(2,\bZ)\backslash \bH=\bA^1_j$ to the
$j$-line. Thus, the boundary of the Baily-Borel compactification has
codimension $18$.

\subsection{Toroidal compactifications}
\label{sec:toroidal-comp}

Toroidal compactifications 
$F_{2d}\hookrightarrow \overline{F}_{2d}^{\mathfrak{F}}$ have
divisorial boundary, but depend on a {\it $\Gamma$-admissible collection of fans}.
This is a choice of a fan $\fF=\{\fF_I\}$ for each cusp of the Baily-Borel
compactification, satisfying conditions described below.
For the 1-cusps, the fans are 1-dimensional and no choice
is involved; they are automatically compatible with the fans for the
0-cusps.

Each 0-cusp corresponds to a primitive isotropic line
$I\subset h^\perp$.  Consider the lattice
$\Lambda_I:=I^\perp/I$ whose intersection form has
signature $(1,18)$. Let
$\Gamma_I:=\textrm{Stab}_\Gamma(I)/U_I$ where
$U_I\subset \textrm{Stab}_\Gamma(I)$ is the unipotent
subgroup, isomorphic to a translation subgroup of $I^\perp/I$.
Let $\cC_I$ denote the
positive cone of $\Lambda_I\otimes \R$ and let
${\ov\cC}_I$ denote its {\it rational closure}---the union of
the positive cone and the rational null rays on its boundary.
Then the fan $\fF_I=\{\tau_i\}$ is a collection of
closed, convex, rational polyhedral cones in ${\ov\cC}_I$,
 closed under taking intersections and faces, such that:
\begin{enumerate}
\item $\Supp\fF_I = {\ov\cC}_I$ and $\fF_I$ is
  locally finite in the positive cone $\cC_I$.
\item $\fF_I$ is invariant under the action of $\Gamma_I$
  with only finitely many orbits.
\end{enumerate}
Then for each $0$-cusp $I$,
the infinite type toric variety $X(\fF_I)$ contains an analytic open subset $V_I$
satisfying the following conditions:
 \begin{enumerate}
 \item $V_I$ contains all toric boundary strata of
   $X(\fF_I)$ which correspond to cones of $\fF_I$ that
   intersect $\cC_I$ (the only strata it does not fully contain
   are those corresponding to null rays and the origin).
 \item $V_I$ is $\Gamma_I$-invariant and the action of
   $\Gamma_I$ is properly discontinuous.
 \item The open stratum of $V_I$ modulo $\Gamma_I$ is
   the intersection of a neighborhood of the Type III
   point $P_I$ of $\overline{F}_{2d}\ubb$ with $F_{2d}$.
 \end{enumerate} 
 
Taking the union of $F_{2d}$ with the open sets from (3), for all $I$, 
we get a map
 \begin{displaymath}
   F_{2d} \cup_I (\Gamma_I\backslash V_I)
   \rightarrow \overline{F}_{2d}\ubb   
 \end{displaymath}
 with complete fibers. It surjects onto the union of $F_{2d}$ with an
 open neighborhood of the Type III boundary point. This map
 extends over the Type II boundary as a fibration in 
 finite quotients of abelian varieties. More explicitly,
 the preimage of the Type II boundary component
 $\Gamma_J\backslash F_J\subset \oF_{2d}^{\bb}$ in the
 toroidal compactification is the quotient by a subgroup of
 $O(J^\perp/J)$ of a family of abelian varieties isogenous to
 $J^\perp/J\otimes \cE$, the self product of the universal elliptic
 curve over $\Gamma_J\backslash F_J$.

The {\it toroidal compactification}
 $\overline{F}_{2d}^{\mathfrak{F}}$ associated to the $\Gamma$-admissible collection
 of fans $\mathfrak{F}$ is then the result of extending these abelian variety families
 from $F_{2d} \cup_I (\Gamma_I\backslash V_I)$, along
 all orbits of rank $2$ isotropic lattices $J$.
 The toroidal compactification admits a birational morphism 
 $\overline{F}_{2d}^\mathfrak{F}\to \overline{F}_{2d}\ubb$ which is an isomorphism on $F_{2d}$.

 For degree 2 K3 surfaces, there is only one 0-cusp, and the fan for this unique
 $0$-cusp has the support on $\opC=\opC_I$ in the vector space
 $N\otimes\bR$, where
 $N = I^\perp/I = H\oplus E_8^2 \oplus A_1$ is a lattice of
 signature $(1,18)$. The fan must be
 $\Gamma_I=O^+(N)$-invariant, where $O^+(N)$
 is the index 2 subgroup of the isometry group
 $O(N)$ preserving the positive cone $\pC$.
For us, the critical fact is:

\begin{proposition}\label{right-monodromy} The unipotent
$U_I\backslash \bD$ embeds into $I^\perp/I\otimes \C^*$
and the period map $C^*\to U_I\backslash \bD$ of a Kulikov model
$\cX\to (C,0)$ with
monodromy invariant $\lambda$ is well-approximated
by a translate of the cocharacter $\lambda\otimes \C^*$ near $0\in C$. \end{proposition}

\begin{proof} This is a direct consequence of Schmid's nilpotent orbit theorem. \end{proof}

 \subsection{Stable pair compactification}
\label{sec:ksba-comp}

First, we recall the definitions:
\va{I reworked this section to apply to $\oP_{2d}$ as well}

\begin{definition}\label{def:lc}
  A pair $(X,B = \sum b_iB_i)$ consisting of a normal variety and a
  $\bQ$-divisor with $0\le b_i\le 1$, $b_i\in\bQ$ is \emph{log
    canonical (lc)} if the divisor $K_X+B$ is $\bQ$-Cartier and for a
  resolution $f\colon Y\to X$ with a divisorial exceptional locus
  $\Exc(f) = \cup E_j$ and normal crossing
  $\cup f_*\inv B_i\cup\Exc(f) $, in the natural formula
  \begin{displaymath}
    f^*(K_X+B) = K_Y + \sum_i b_i f_*\inv B_i + \sum_j b_j E_j
    \quad\text{one has } b_j\le 1.
  \end{displaymath}
\end{definition}

\begin{definition}\label{def:slc}
  A pair $(X,B = \sum b_iB_i)$ consisting of a reduced variety and a
  $\bQ$-divisor is \emph{semi log canonical (slc)} if $X$ is $S_2$, has at
  worst double crossings in codimension 1, and for the normalization
  $\nu\colon X^\nu \to X$ writing
  \begin{displaymath}
    \nu^*(K_X+B) = K_{X^\nu} + B^\nu,
  \end{displaymath}
  the pair $(X^\nu, B^\nu)$ is log canonical. Here, $B^\nu = D + \sum b_i
  \nu\inv(B_i)$, and $D$ is the double locus.   
\end{definition}

\begin{definition}\label{def:stable-pair}
  A pair $(X,B)$ consisting of a connected projective
  variety $X$ and a $\bQ$-divisor is \emph{stable} if
  \begin{enumerate}
  \item $(X,B)$ has semi log canonical singularities, in particular
    $K_X+B$ is $\bQ$-Cartier.
  \item The $\bQ$-divisor $K_X+B$ is ample. 
  \end{enumerate}
\end{definition}

Next, we introduce the objects that we are interested in here:

\begin{definition}
  For a fixed degree $e\in\bN$ and fixed rational number
  $0<\epsilon\le 1$, a \emph{stable $K$-trivial pair} of type $(e,\epsilon)$ is
  a pair $(X,\epsilon R)$ such that
  \begin{enumerate}
  \item $X$ is a Gorenstein surface with $\omega_X\simeq\cO_X$,
  \item The divisor $R$ is an ample Cartier divisor of degree $R^2=e$.
  \item The surface $X$ and the pair $(X,\epsilon R)$ are slc. In
    particular, the pair $(X,\epsilon R)$ is stable in the sense of
    Definition~\ref{def:stable-pair}.
  \end{enumerate}
\end{definition}

\begin{definition}\label{def:family-stable-k3}
  A family of stable $K$-trivial pairs of type $(e,\epsilon)$ is a flat
  morphism $f\colon (\cX,\epsilon\cR)\to S$ such that
  $\omega_{\cX/S}\simeq \cO_X$ locally on $S$, the divisor $\cR$ is a
  relative Cartier divisor, such that every fiber is a stable $K$-trivial pair
  of type $(e,\epsilon)$.
\end{definition}

\begin{lemma}\label{lem:indep-of-epsilon}
  For a fixed degree $e$ there exists an $\epsilon_0(e)>0$ such that
  for any $0<\epsilon\le \epsilon_0$ the moduli stacks
  $\cM^\slc(e,\epsilon_0)$ and $\cM^\slc(e,\epsilon)$ coincide.
\end{lemma}
\begin{proof}
  For a fixed surface $X$, there exists an $0<\epsilon_0\ll 1$ such
  that the pair $(X,\epsilon_0 R)$ is slc iff $R$ does not contain any
  centers of log canonical singularities: images of the divisors with
  codiscrepancy $b_i=1$ on a log resolution of singularities
  $Y\to X^\nu\to X$ as in Definitions~\ref{def:lc}, \ref{def:slc}.
  There are finitely many of such centers.  Then for any
  $\epsilon<\epsilon_0$, the pair $(X,\epsilon_0 R)$ is slc iff
  $(X,\epsilon R)$ is.

  Now since $\omega_X\simeq\cO_X$ and $R$ is ample Cartier of a fixed degree,
  the family of the pairs $(X,R)$ is bounded, and the number
  $\epsilon_0$ with this property can be chosen universally.
\end{proof}

We will be interested in the moduli space $M^\slc_e$ of such pairs,
and more precisely in the closure of $F_{2d}$ in $M_e^\slc$ for a chosen
\emph{intrinsic} polarizing divisor $R\in|NL|$.

\smallskip

We refer to \cite{kollar1988threefolds-and-deformations,
  kollar2022book-on-moduli},
\cite{alexeev2006higher-dimensional-analogues} for the existence and
projectivity of the moduli space of stable pairs $(X,\sum b_iB_i)$. In
general, when some coefficients $b_i$ are $\le\frac12$, there are
delicate problems with the definition of a family since a flat limit
of divisors may happen to be a nonreduced scheme with embedded
components. In our case the situation is much easier since $R$ is Cartier.

\begin{definition}
  A family of stable $K$-trivial pairs of degree $e$ is a family of type
  $(e,\epsilon_0)$, where $\epsilon_0(e)$ is chosen as in
  Lemma~\ref{lem:indep-of-epsilon}. We denote the corresponding
  moduli functor by $M^\slc_e$.  For a scheme $S$,
  \begin{math} 
    M^\slc_e(S) = 
    \{ \text{families of type $(e,\epsilon_0(e))$ over } S\}.
  \end{math}
\end{definition}

\begin{proposition}
  There is a Deligne-Mumford stack $\cM^\slc_e$ and a coarse moduli
  space $M^\slc_e$ of stable $K$-trivial pairs.
\end{proposition}
\begin{proof}
The spaces $\cM^\slc_e$ and $M^\slc_e$ are constructed by standard
methods, as quotients of appropriate Hilbert schemes by a $\PGL$ group
action. Again, for general stable pairs there are delicate questions
of the formation of $(\omega^{\otimes n}_{\cX/S}(nR))^{**}$ commuting
with base changes. But in our case both $\omega_{\cX/S}$ and
$\cO_\cX(\cR)$ are invertible, so these questions disappear.
\end{proof}

We do not prove that the moduli space $M_e^\slc$ is proper
but we do prove below that it provides a compactification for the
moduli spaces of ordinary K3 surfaces. 
(The components of $M_e^\slc$ where $X$ is generically non-normal
require additional arguments.)
A related moduli space is:

\begin{definition}
  Let $N\in\bN$. The moduli stack $\cP_{N,2d}$ parameterizes proper
  flat families of pairs $(X, R)$ such that $(X,L)$ is a polarized K3
  surface with $ADE$ singularities and a primitive ample line bundle
  $L$, $L^2=2d$, and $R\in |NL|$ is an arbitrary divisor. One has
  $R^2=2dN^2$. In particular, one defines $\cP_{2d}:= \cP_{1,2d}$.
\end{definition}

If we take $\epsilon_0(e)$ as in Lemma~\ref{lem:indep-of-epsilon},
then the pair $(X,\epsilon_0 R)$ is stable. Obviously, the stack
$\cP_{N,e}$ is fibered over the stack $\cF_{2d}$ with fibers
isomorphic to $\bP^{dN^2+1}$. The automorphism groups of stable pairs
are finite, and it is easy to see that the stack $\cP_{N,2d}$  is
coarsely represented by a scheme, denoted $P_{N,2d}$.

\begin{definition}
  One defines $\oP_{N,2d}$ to be the closure of the coarse moduli
  space $P_{N,2d}$ in $M^\slc_e$ for $e=2dN^2$.  A \emph{canonical
    choice} of a divisor $R\in |NL|$ for each $(X,L)\in F_{2d}$ gives an embedding $F_{2d}
  \subset P_{N,2d}$. Its closure in $\oP_{N,2d}$ is denoted by
  $\oF_{2d}^\slc$. 
\end{definition}

\begin{theorem}
  $\oP_{N,2d}$ and thus also $\oF_{2d}^\slc$ are proper and
  projective. 
\end{theorem}

\begin{proof}
  Properness follows from the following theorem. Projectivity follows
  from it by results of Fujino and Kov\'acs-Patakfalvi
  \cite{kovacs2017projectivity}.
\end{proof}

\begin{theorem}\label{thm:extend-family-stable-k3}
  For a fixed degree $e$, every family
  $f^*\colon \cX^* \to C^*=C\setminus 0$ over a smooth curve of K3
  surfaces with ADE singularities and ample $R$, $R^2=e$, can be
  extended to a family of stable K3 pairs
  $(\cX',\epsilon_0(e) \cR')\to C'$ of type $(e,\epsilon_0(e))$
  possibly after a ramified base change $C'\to C$.
\end{theorem}
\begin{proof}
  (Cf. \cite[Thm.2.11, Rem.2.12]{laza2016ksba-compactification})
  The proof is achieved by modifying that of a theorem of
  Shepherd-Barron
  \cite[Thm.1]{shepherd-barron1981extending-polarizations}.  His
  theorem says that if $\cX\to C$ is a semistable model with
  $K_{\cX}=0$ and $\cL^*$ is a relatively nef line bundle on
  $\cX^*$ of positive degree then there exists another
  semistable model to which $\cL^*$ extends as a relatively nef line
  bundle $\cL$.  This is done over $C$, without a base change.  Then
  \cite[Thm.2]{shepherd-barron1981extending-polarizations} says that
  $\cL^n$ for $n\ge 4$ gives a contraction $\pi\colon\cX\to \ov{\cX}$
  so that $\omega_{\ov\cX}\simeq\cO_{\cX}$, with $\ov{\cL}$ an ample line
  bundle on $\ov{\cX}$ and $\cL = \pi^*(\ov{\cL})$.

  Now let $f\colon (\cX^*,\epsilon \cR^*)\to C^*$ be a family
  of K3 surfaces with ADE singularities and a relatively ample Cartier
  divisor $\cR$. After shrinking the base, we can simultaneously
  resolve the singularities to obtain a family of smooth K3s
  $(\cX^*_1,\cR^*_1)$ with a relatively big and nef effective
  divisor. By Theorem \ref{kpp-thm}, after a further base change we get a semistable
  model $\cX_2\to C$ with $K_{\cX_2}=0$. We are now in a situation
  where Shepherd-Barron's theorem applies.  However, first we make another
  base change that exists by Claim~\ref{cla:kulikov-avoid-strata} to
  obtain a Kulikov model $\cX_3$ satisfying condition~(*):
  
  \begin{itemize}
    \item[(*)] The closure of $\cR^*$ in $\mathcal{X}$
    does not contain any strata (double curves or triple points) of
    the central fiber $\cX_0$.
    \end{itemize}
  
  The proof of
  \cite[Thm.1]{shepherd-barron1981extending-polarizations} proceeds by
  starting with a divisor $\cR_3$ which does not contain an entire
  component of the central fiber. One then makes it nef using flops
  along curves $E$ with $\cR_3\cdot E<0$. The flops are called elementary
  modifications. They are of three types: (0) along an interior
  $(-2)$-curve, (I) along a curve intersecting a double curve, and
  (II) along a double curve.

  But with condition (*) achieved, the divisor $\cR_3$ already
  intersects the double curves non-negatively, and the flops of type
  II are not needed. The flops of types 0 and I preserve (*). Thus,
  the end result is a model $\cX_4\to C'$ with an effective,
  relatively nef divisor $\cR_4$ satisfying (*).

  Since the central fiber $(\cX_4)_0$ is normal crossing, for
  $0<\epsilon\ll1$ the pair $(\cX_4,\epsilon \cR)$ is slc. Then the
  contraction $\cX_4\to \ov{\cX}_4$ provided by
  \cite[Thm.2]{shepherd-barron1981extending-polarizations} gives a
  family $(\ov{\cX}_4,\epsilon \ov{\cR}_4)$ of stable pairs extending
  the original family $(\cX^*, \epsilon\cR^*)\to C^*$ after a base
  change $C'\to C$.
\end{proof}

\begin{claim}\label{cla:kulikov-avoid-strata}
  For any Kulikov model $\cX\to C$ there exists a finite base change
  $C'\to C$ and birational modification to a Kulikov model
  $\cX'\to C'$ of $\cX\times_C C'$ satisfying {\rm (*)}.
\end{claim}
\begin{proof}
  This is a local toric computation. We give an argument for a triple
  point which by simplification also applies to double curves. An
  obvious modification of this proof works for a semistable
  degeneration in any dimension.

  Let the triple point be the origin with a local equation $xyz=t$. A
  local toric model of it is $\bA^3_{x,y,z}$. Its fan is the cone
  $\sigma$ that is the first octant in $\bR^3$ with the lattice
  $N=\bZ^3$. In the lattice of monomials $M=N^*\simeq\bZ^3$ the dual
  cone $\sigma^\vee$ is the first octant as well. A ramified base
  change $t=s^d$ means choosing the new lattices
  \begin{displaymath}
    M' = M + \bZ \frac{(1,1,1)}{d} , \quad
    N\supset N' = \{n = (a,b,c) \mid n \cdot \frac{(1,1,1)}{d} \in \bZ  \}.
  \end{displaymath}
  Choosing a Kulikov model locally at this $0$-stratum is equivalent
  to choosing a triangulation $\cT$ of the triangle $\sigma\cap \{a+b+c =
  d\}$ with the vertices $(d,0,0)$, $(0,d,0)$, $(0,0,d)$ into elementary
  triangles of lattice volume 1. Then the new fan is obtained by
  subdividing $\sigma$ into the cones over these elementary triangles.

  We note that an arbitrary triangulation $\cT$ will \emph{not}
  achieve the condition (*). Instead, it has to be chosen
  carefully. Using $x,y,z$ as local parameters, the equation of the
  divisor is a power series $f\in k[[x,y,z]]$.  Let $\{m_j\}$ be the
  set of the monomials appearing in $f$. Let $P$ be the convex hull of
  $\cup_j \big( m_j + \sigma^\vee \big)$. This is an infinite
  polyhedron but it has only finitely many vertices, say $m_j$ for
  $1\le j\le r$. 

  Let $\Nor(P)$ be the normal fan of $P$; it is a refinement of the
  cone $\sigma$. Let~$X'$ be the toric variety, possibly singular, for
  this fan. We have a toric blowup $X'\to\bA^3$ modeling a blowup
  $f\colon \cX'\to\cX$. The strict preimage of the divisor $\cR$ has
  the same equation~$f$ which still makes sense for each of the
  standard open sets $\bA^3$ that cover $X'$. The reason for taking the
  convex hull was this: The vertices of $P$ correspond to the
  $0$-dimensional strata $x'_j$ of $X'$ and the fact that for each of
  them the corresponding monomial has a nonzero coefficient means that
  the divisor does not pass through $x'_j$. These points are in a
  bijection with the maximal-dimensional cones $\sigma'_j$ of
  $\Nor(P)$. Subdividing these cones further means blowing up at the
  points $x'_j$ further. The preimage of the divisor under these
  blowups will not contain any strata on the blowup.

  So the final recipe is this: From the equation of $f$ obtain the
  polyhedron $P$ and its normal fan $\Nor(P)$. It has finitely many
  rays $\bR_{\ge0}(a_i,b_i,c_i)$, $(a_i,b_i,c_i)\in \bZ^3_{\ge0}$. Let
  $d_i = a_i+b_i+c_i$ and let $d$ be the $\operatorname{gcd}(d_i)$ so
  that these rays are cones over some integral points of the triangle
  $\sigma\cap \{a+b+c = d\}$. The fan $\Nor(P)$ gives a subdivision of
  this triangle. Refine it arbitrarily to a triangulation $\cT$ into
  volume~1 triangles. This defines a Kulikov model locally.  Then
  repeat this procedure at all the $0$-strata of $\cX$. The resulting
  Kulikov model satisfies the condition (*).
\end{proof}

\begin{remark}
  Difficulties with the moduli spaces of stable pairs $(X,B=\sum
  b_iB_i)$ arise when $K_X+B$ is $\bQ$- or
  $\bR$-Cartier but $K_X$ and $B$ by themselves are
  not. One solution was proposed in
  \cite[Sec. 1.5]{alexeev2015moduli-weighted}: choose the coefficients
  $b_i$ so that $(1,b_1,\dotsc, b_n)$ are $\bQ$-linearly
  independent. In the situation at hand this means picking 
  $\epsilon$ to be irrational. We do not need this trick for the K3
  surfaces, however, since by the above the divisor $R$ remains
  Cartier in the interesting part of the compactified moduli space. 
\end{remark}

\begin{theorem}\label{thm:slc-over-bb}
  The rational maps $(\oP_{N,2d})^\nu \dashrightarrow
  \oF_{2d}\ubb$ and, for a
  canonical choice of a polarizing divisor,
  $\big(\oF_{2d}^\slc\big)^\nu \dashrightarrow \oF_{2d}\ubb$ from the
  normalizations of $\oP_{N,2d}$ and $\oF_{2d}^\slc$ to the
  Baily-Borel compactification are regular.
\end{theorem}
\begin{proof}
  We apply Lemma~\ref{lem:extend-rat-map} with
  $X=(\oP_{N,2d})^\nu$ resp. $X=\big(\oF^\slc_{2d}\big)^\nu$,
  and $Y=\oF_2\ubb$.  We claim that
  the condition of \eqref{lem:extend-rat-map} is satisfied.
  Namely, for a one-parameter family of stable K3 surfaces over
  $(C,0)$, the central fiber uniquely determines if the limit in the
  Baily-Borel compactification is of Type II or Type III, and if it is
  of Type II then the $j$-invariant of the elliptic curve is uniquely
  determined.

  As in the proof of Theorem~\ref{thm:extend-family-stable-k3}, we get
  a Kulikov model $\cX$ to which a big and nef line bundle
  $\cL=\cO_\cX(\cR)$ extends and then a contraction $\cX\to\ov{\cX}$
  to the canonical model. If $\cX$ is of Type III then $\ov{\cX}$ is a
  union of rational surfaces with rational singularities, glued along
  rational curves. If $\cX$ is of Type II then either some components
  of $\ov{\cX}$ are glued along an elliptic curve $E$ or, if all the
  elliptic curves that constitute the double locus of $\cX$ are contracted, a
  component of $\ov{\cX}$ has an elliptic singularity, resolved by
  inserting $E$. So the Type, and for Type II the $j$-invariant of $E$,
  can be recovered from the central fiber $\ov{\cX}_0$. 
\end{proof}

\begin{lemma}\label{lem:extend-rat-map}
  Let $X$ and $Y$ be proper varieties, with $X$ normal.
  Let $\varphi\colon X\ratmap Y$ be a rational map,
  regular on an open dense subset $U\subset X$.  Let $(C,0)$ be a
  regular curve and $f\colon C\to X$ a morphism whose image meets $U$.
  Let $g\colon C\to Y$ be the unique extension of $f\circ\varphi$
  which exists by the properness of $Y$.

  Assume that for all $f$ with the same $f(0)$, there are only
  finitely many possibilities for $g(0)$.  Then $\varphi$ can be
  extended uniquely to a regular morphism $X\to Y$.
\end{lemma}
\begin{proof}
  Let $Z\subset X\times Y$ be the closure of the graph of $U\to
  Y$. The projection $Z\to Y$ is a morphism extending
  $\varphi$. The morphism $Z\to X$ is birational, and the condition is
  that it is finite. By Zariski's Main Theorem $Z\to X$ is an
  isomorphism.
\end{proof}

\section{The Coxeter fan and compactifications of $F_2$}
\label{sec:reflection-fan}

\subsection{The Coxeter fan}

For $F_2$, a toroidal compactification depends on a single fan,
supported on the rational closure $\opC$ of the positive cone in
the space $N_\bR$ for the hyperbolic lattice $N=H\oplus E_8^2\oplus
A_1$. We now describe a particularly nice fan on $N$,
cf. \cite[6.2]{scattone1987on-the-compactification-of-moduli}. 

\begin{definition}
  The Coxeter fan $\fF^\cox$ is obtained by cutting $\opC$ by the mirrors
  $r^\perp$ to the roots of $N$, i.e. the vectors $r\in N$ with
  $r^2=-2$. 
\end{definition}

The Weyl group $W(N)$ generated by reflections in the roots has finite
index in the isometry group $O^+(N)$, with the quotient
$O^+(N)/W(N) = S_3$. Here $O^+(N)$ is the index 2 subgroup of $O(N)$
fixing the positive cone. Reflection groups acting on hyperbolic
spaces we studied by Vinberg, see e.g.
\cite{vinberg1985hyperbolic-groups, vinberg1973some-arithmetic}. Note
that in those papers a hyperbolic space has signature $(r-1,1)$ vs. our
$(1,18)$.

A fundamental chamber $\fund$ of $W(N)$ is described by a
Coxeter diagram given in Fig.~\ref{fig:vinberg}. The nodes represent
$24$ roots $r_i$ that generate $N$, with the index $i$ given by the label
in Fig.~\ref{fig:vinberg}.  We have $(r_i,r_j)=0,1,2,6$ depending
on whether there is no line, a single line, a doubled line, or a dashed line
connecting $i$ to $j$, respectively. The fundamental chamber is
$$\fund =\{\lambda\in \opC\,:\, \lambda\cdot r_i\geq 0\textrm{ for }0\leq i\leq
23\}.$$ The group $S_3$ acts on the fundamental chamber by symmetries
of the diagram. The projectivization $P = \bP(\fund)$ is a hyperbolic polytope with
cusps: it has infinite vertices corresponding to null rays $v\in \fund$ with
$v^2=0$. However, it has finite hyperbolic volume.

\begin{figure}[!htp]
  \begin{center}
    \begin{minipage}{175pt}
      \includegraphics[width=1.0\linewidth]{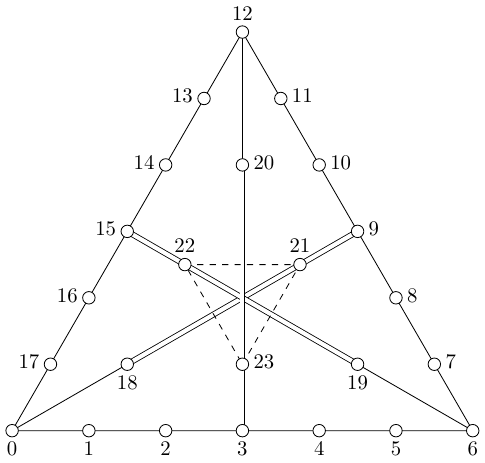}
    \end{minipage}
    \begin{minipage}{175pt}
      \begin{minipage}{.45\linewidth}
        \includegraphics[width=\linewidth]{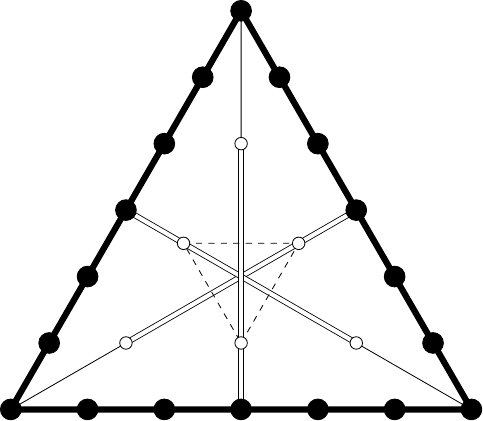}
        \vskip 10pt
        \includegraphics[width=\linewidth]{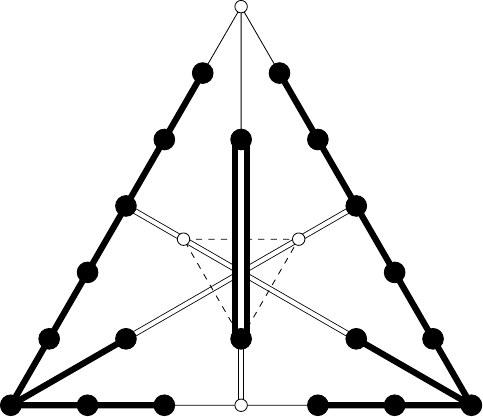}
      \end{minipage}
      \begin{minipage}{.45\linewidth}
        \includegraphics[width=\linewidth]{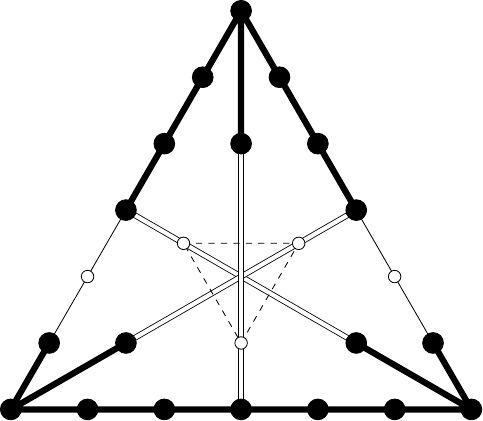}
        \vskip 10pt
        \includegraphics[width=\linewidth]{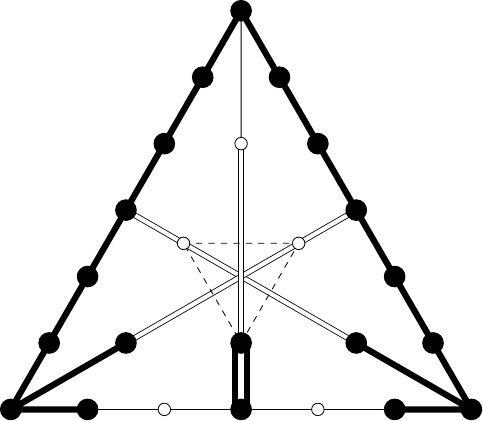}
      \end{minipage}
    \end{minipage}
  \end{center}
  \caption{Coxeter diagram $G_\vin$ and its maximal parabolic subdiagrams
  $\wA_{17}$, $\wD_{10}\wE_7$, $\wE_8^2\wA_1$, $\wD_{16}\wA_1$}
  \label{fig:vinberg}
\end{figure}

\begin{definition}
  A \emph{subdiagram of $G_\vin$} is a subgraph $G\subset G_\vin$
  induced by a subset $V\subset V(G_\vin)$ of the vertices, i.e. a
  subset of the 24 roots $r_i$. It defines a vector subspace
  $\bR V = \la r_i,\ i\in V\ra \subset N_\bR$.

  A subdiagram is called \emph{elliptic} if the restriction of the
  quadratic form of $N$ to $\bR V$ is negative definite. It is called 
  \emph{parabolic} if it negative semi-definite. \emph{Maximal
    parabolic} means maximal by inclusion among the parabolic diagrams.
\end{definition}

Vinberg described the faces of the fundamental polytope $P$, see
\cite[Thm.3.3]{vinberg1973some-arithmetic}. In our situation this gives:

\begin{theorem}\label{thm:faces-fund-chamber}
  The correspondence
  \begin{displaymath}
    F \mapsto G(F) = \{i \mid F\subset r_i^\perp\},
    \quad
    G = \{r_i,\ i\in G\}\mapsto F(G)= \cap_{i\in I} r_i^\perp
  \end{displaymath}
  defines an order reversing bijection between the faces of the
  fundamental chamber $\fund$ and the elliptic and maximal parabolic
  subdiagrams $G\subset G_\vin$. The chamber itself corresponds to
  $G=\emptyset$.

  Type III cones (meeting the interior $\pC$) of dimension $d>0$
  correspond to elliptic subdiagrams of rank $r=19-d$. These
  are disjoint unions of Dynkin diagrams $G_i$ of ADE type with
  $\sum |G_i| = r$.

  Type II rays $\bR_{\ge0}v$ with $v^2=0$ 
  correspond to maximal parabolic subdiagrams of $G_\vin$. These are
  disjoint unions of affine Dynkin diagrams $\wG_i$ with
  $\sum |G_i| = 17$.
\end{theorem}

\begin{lemma}\label{lem:cones-fcox}
  The cones of the Coxeter fan $\fF^\cox$ mod $W(N)$ are in a
  bijection with the faces of the fundamental chamber.  The cones of
  $\fF^\cox$ mod~$O^+(N)$ are in a bijection with elliptic and maximal
  parabolic subdiagrams of $G_\vin$ mod~$S_3$.
\end{lemma}
\begin{proof}
  This follows since $\fund$ is a fundamental domain for the
  $W(N)$-action and $O^+(N) = S_3\ltimes W(N)$.
\end{proof}

The following two lemmas are proved by direct
enumeration.

\begin{lemma}\label{lem:parabolic-subdiagrams}
  Mod $S_3$, there are $4$ maximal parabolic subdiagrams of $G_\vin$,
  illustrated on the right in Fig.~\ref{fig:vinberg}.
  \begin{enumerate}
  \item $\wA_{17} = [ i,\ 0\le i<18]$,
  \item $\wD_{10}\wE_7 = [ 18,{17},0 , \dotsc,  6,7,19] \sqcup
    [ 9 , \dotsc,  {15},20]$,
  \item $\wE_8^2\wA_1 = [ {13}, \dotsc, 2,18] \sqcup
    [ 4 , \dotsc,  11, 19] \sqcup [ 20,23]$.
  \item $\wD_{16}\wA_1 = [ 19,5,6 , \dotsc,  0,1,18] \sqcup
    [ 3,23]$. 
  \end{enumerate}
\end{lemma}

\begin{lemma}\label{lem:numbers-strata-Ftor}
  Mod $S_3$, the numbers of elliptic subdiagrams of $G_\vin$
  that have ranks $r=1,\dotsc, 18$ are
  $6$, $51$, $328$, $1518$, $5406$, $14979$, $33132$, $59339$,
  $87077$, $105236$, $105078$,
  $86505$, $58223$, $31564$, $13371$, $4209$, $883$, $99$.
  In particular, in $\fF^\cox$ mod $O^+(N)$ there are $4 + 99 = 103$ rays.
\end{lemma}

For each of the extended Dynkin diagrams $\wA_k$, $\wD_k$, $\wE_k$,
there is a unique primitive positive integer combination of the roots
which is null in the affine root lattice.
The coefficients for the first $k$ nodes are the fundamental weights
of the corresponding Lie algebra and the coefficient of the extended
node is $1$. Alternatively, these are labels of the extended Dynkin
diagram such that each label is half the sum of its neighbors. For
example, for the first $\wE_8$ diagram in case (3) above this vector
is
\begin{displaymath}
  n\big(\wE_8^{(1)}\big) = r_{13}+2r_{14}+3r_{15}+4r_{16}+5r_{17}+6r_{0} +
  4r_1+2r_2+3r_{18}.
\end{displaymath}

\begin{lemma}\label{lem:N-linear-rels}
  For each maximal parabolic subdiagram of $G_\vin$ the square-zero
  vectors of its connected components coincide:
  \begin{displaymath}
    n\big(\wD_{10}\big) = n\big(\wE_7\big), \quad
    n\big(\wE_8^{(1)}\big)= n\big(\wE_8^{(2)}\big)= n\big(\wA_1\big), \quad
    n\big(\wD_{16}\big) = n\big(\wA_1\big).
  \end{displaymath}
  The six $\wE_8^2\wA_1$ equations generate all the relations between the
  $24$ roots $r_i$. The unique syzygy between them is that the sum of
  the three $\wE_8^2$ differences is zero.
\end{lemma}
\begin{proof}
  An easy direct check.
\end{proof}

\subsection{Connected Dynkin subdiagrams of $G_\vin$}
\label{subsec:connected-subdiags}

We adopt the notation of \cite{alexeev17ade-surfaces} for the
connected subdiagrams of $G_\vin$ using decorated Dynkin diagrams.

\begin{definition}\label{def:irr-subdiagram}
  The subdiagrams of $G_\vin$ with the vertices entirely contained in
  the subset $\{18, 19, 20, 21, 22, 23\}$ are called \emph{irrelevant}. A
  diagram is \emph{relevant} if it has no irrelevant connected
  components. For each $G\subset G_\vin$ its \emph{relevant content}
  $G^\rel$ is the subdiagram obtained by dropping all irrelevant
  connected components.
\end{definition}

We list the connected subdiagrams of $G_\vin$  in
Table~\ref{tab:vin-conn-subdiags}.  The indices $0\le i<18$ are taken
in $\bZ_{18}$. We first give the elliptic subdiagrams, then parabolic,
then irrelevant elliptic and finally irrelevant parabolic.  The
diagrams are considered up to the $S_3$-symmetry if they do not lie in
the outside $18$-cycle. The ones that are contained in the $18$-cycle
are considered up to the dihedral symmetry group $D_9$.

\begin{table}[h!]
  \centering
  \begin{tabular}[h!]{lll|lll}
    Type &Vertices & &Type &Vertices  \\
    \hline
    $A_{2n+1}$ & $ {2i+1},\dotsc, {2i+2n+1}, $ & $n\leq 8$ &
          $\wA_{17}$ & $ i,\ 0\le i<18$  \Tstrut \\
    $A_{2n}^-$ & $ {2i+1},\dotsc, {2i+2n} ,$ & $n\leq 8$  &
          $\wD_{10}$ & $ 18,17,0 , \dotsc,  6,7,19$ \\
    $\mA_{2n+1}^-$ & $ {2i},\dotsc, {2i+2n} ,$ & $n\leq 8$ &
          $\wE_7$ & $ 9 , \dotsc,  {15},20$  \\
    $\pA_{2n+1}$ & $ 18, 0,1,\dotsc, {2n-1} ,$ & $n\leq 8$ &
          $\wE_8^-$ & $ {13}, \dotsc, 2,18$  \\
    $\pA_{2n}^-$ & $ 18, 0,1,\dotsc, {2n-2}, $ & $n\leq 8$ &
          $\wD_{16}$ & $ 19,5,6 , \dotsc,  0,1,18$  \\
    $\pA_9'$ & $ 18, 0, \dotsc, 6, 19 $ &  &
          $\wA_1^*$ & $  3,23$ \\
    $\pA_{15}'$ & $ 18, 0, \dotsc, {12}, 20 $ &  \\

    $D_{2n}$ & $ 18, 17, 0,1,\dotsc, {2n-3}, $ & $n\leq 8$ &
         $A_1^\irr$ & $18$ \\
    $D_{2n+1}^-$ & $18, 17, 0,1\dotsc, {2n-2}, $ & $n\leq 8$ &
         $\mA_1^{-\,\irr}$ & $21$ \\
    $D'_{10}$ & $ 18, 17, 0,\dotsc, 6, 19 $ &  \\

    $D'_{16}$ & $ 18, 17, 0,\dotsc, {12}, 20 $ &   &
          $\wA_1^\irr$ & $  20,23$ \\
    $\phmi E_6^-$ & $ 18, {16}, 17, 0, 1, 2$ & \\

    $\phmi E_7$ & $ 18, {16}, 17, 0, 1, 2, 3$ &  \\

    $\phmi E_8^-$ & $ 18, {16}, 17, 0, 1, 2, 3, 4$ &  \\
  \end{tabular}
  \smallskip
  \caption{Connected elliptic and parabolic subdiagrams of $G_\vin$}
  \label{tab:vin-conn-subdiags}
\end{table}

The parabolic subdiagrams of $G_\vin$ are shown in
Fig.~\ref{fig:vinberg}. 

\begin{definition}
The {\it skeleton} of a diagram its intersection with the cycle $0,1,\dotsc, 17$. \end{definition}

In the shortcut notation of Table \ref{tab:vin-conn-subdiags}, a minus or prime on the left (resp. right)
implies that the clockwise (resp. counterclockwise) vertex adjacent to the skeleton is odd.
The absence of a marking implies the vertex is even. The prime indicates
that an extra leaf of the subdiagram has entered the interior vertices 
$\{18,19,20,21,22,23\}$ of Fig.~\ref{fig:vinberg}. 

\begin{definition}\label{def:stable-type} The {\it stable type} of an elliptic or maximal parabolic subdiagram
$G=\sqcup \, G_k\subset G_\cox$ is its relevant content $G^\rel$,
with diagrams notated as in Table \ref{tab:vin-conn-subdiags},
listed in cyclic order around the $18$-cycle. We introduce symbols
$A_0^-$ or $\mA_0$ to indicate, respectively, that both $2i, 2i+1$ or both $2i+1,2i+2$ (mod $18$)
do not lie in the skeleton of $G$.
   \end{definition}
   
   The insertion of the symbols $A_0^-$ or $\mA_0$ is necessary to determine the spacing
   between the relevant connected components. Two examples are shown in Fig.~\ref{fig:diagrams}.
   Note that the $S_3$ or $D_9$ action cyclically rotates and/or flips the diagram labels in the stable
   type, and orientation reversing symmetries flip which sides of a symbol
   are decorated with a $-$ sign.

\begin{figure}[!htp]
  \begin{center}
    \begin{minipage}{175pt}
      \centering
      \includegraphics[width=\linewidth]{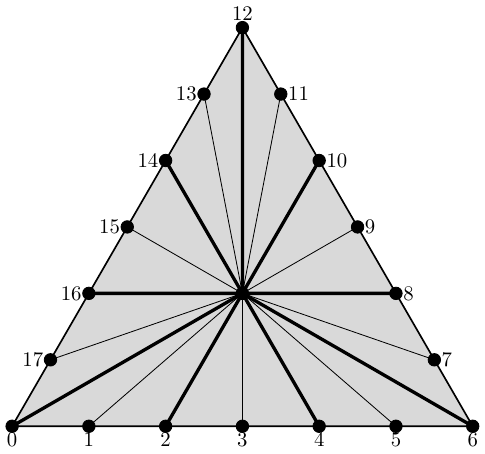}
    \end{minipage}
    \begin{minipage}{175pt}
      \centering
      \includegraphics[width=\linewidth]{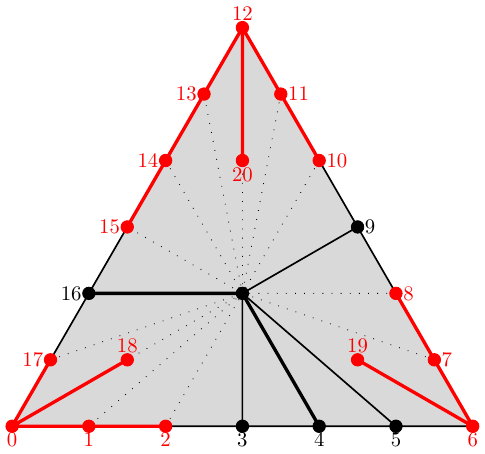} 
    \end{minipage}
    \caption{Stable types $(A_0^-\mA_0)^9$ and  $D_5^- \mA_0 A_0^- \pA_4^- \phmi E_7$.}
    \label{fig:diagrams}
  \end{center}
\end{figure}

\subsection{A toroidal compactification}
\label{sec:toroidal-compn}

\begin{definition}
  The toroidal compactification $\oF_2^\tor=\oF_2^{\fF_\cox}$ we consider in this
  paper is the one corresponding to the Coxeter fan $\fF^\cox$.
\end{definition}

We describe the strata of $\oF_2^\tor$ which by
\eqref{thm:faces-fund-chamber}, \eqref{lem:cones-fcox} correspond to
elliptic and maximal parabolic subdiagrams of $G_\vin$ mod $S_3$.

\begin{notation}
  An elliptic subdiagram $G=\sqcup\, G_k$ is a union of ADE Dynkin
  diagrams. We denote by $R_G$ the corresponding root system and
  $W(G)$ its Weyl group. Let $S_G \ltimes W(G)\subset O(R_G)$ be the
  extension by the symmetries $S_G\subset S_3$ of the subdiagram.
  A parabolic subdiagram $\wG = \sqcup\, \wG_k$ is a union of affine ADE
  Dynkin diagrams. In this case, let $G=\sqcup\, G_k$ be the
  union of the corresponding ordinary (not extended) Dynkin diagrams.
\end{notation}

\begin{proposition}\label{prop:strata-in-F2tor}
  The type III and II strata in $\oF_2^\tor$ are as follows:
  \begin{enumerate}
  \item For an elliptic diagram $G$, $\Str(G)$ is the quotient by $S_G\ltimes W(G)$
  of the torus $\Hom(M_G, \bC^*)$ where $M_G$ is the saturation
  of the root torus $R_G$ in $M=N^*$.
  \item For a maximal parabolic diagram $\wG$, $\Str(\wG)$ is
  the quotient by $S_G\ltimes W(G)$ of $\Hom(M_G, \cE) \simeq \cE^{17}$,
    where $\cE^{17}\to \bA^1_j$ is the
    self fiber product of the universal family of elliptic curves
    $\cE\to\cM_{1,1}$ over the moduli stack.
  \end{enumerate}
\end{proposition}
\begin{proof}
  The strata of Type III are contained in the fiber of
  $\oF_2^\tor\to\oF\ubb$ over the unique Type III point and can be
  described purely in terms of toric geometry.  We have two lattices
  $N = I^\perp/I = H\oplus E_8^2\oplus A_1$ and $M =
  N^*$. Using the quadratic form on $N$, we can present $N^*$ as an
  overlattice with $N^*/N=\bZ_2$. The lattice $N$ is generated by the
  24 roots $r_i$ in Coxeter diagram. Thus,
\begin{displaymath}
  M= N^* = \big\{ v\in \textstyle{\frac12} N  \mid (v,r_i)\in \bZ
  \big\} = N + \frac12 r_{21}.
\end{displaymath}
For each Type III cone $\sigma=\sigma(G)$ of $\fF^\cox$, we have a
cone $\sigma\subset N_\bR$ and a toric variety $U_\sigma$ with a
unique closed orbit $O_\sigma$, which is a torus itself.  It is
standard in toric geometry that
$O_\sigma = \Hom(\sigma^\perp\cap M, \bC^*)$, and we have
$\sigma^\perp = M_G=R_G^\sat$, the saturation of the root lattice
$R_G$ in $M$. In the toroidal compactification we divide an infinite
toric variety by $\Gamma = O^+(N)$. The orbit $T_G$ is divided by
its stabilizer in $O^+(N)$, which is $S_G\ltimes W(G)$. The description in Type
III follows.

The exact structure of a
Type II boundary divisor is determined by
the parabolic group ${\rm Stab}_\Gamma(J)$
stabilizing the corresponding rank $2$ isotropic lattice $J$.
This parabolic group acts on the period domain
$\mathbb{H}\times \C^{17}$ of Type II mixed Hodge structures,
and the quotient is the boundary divisor. The unipotent
subgroup $U_J$ is the kernel of the map
${\rm Stab}_\Gamma(J)\xrightarrow{q} \SL(J)\times O(J^\perp/J)$ 
and induces the full group of translations
$J^\perp/J\otimes (\Z\oplus \Z\tau)\simeq (\Z\oplus \Z\tau)^{17}$
on the second factor $\C^{17}$. Quotienting by $U_J$ first
gives $J^\perp/J\otimes (\C/\Z\oplus \Z\tau)\rightarrow \mathbb{H}$ 
on which the image of $q$ further acts.

We claim that $q$ is surjective. First we show that for any isotropic $J$,
there is a complementary isotropic subspace $J'$,
i.e. a lift of $h^\perp/J^\perp$ to an isotropic plane in $h^\perp$
such that the pairing between $J$ and $J'$ realizes $J'={\rm Hom}(J,\Z)$.
For instance, let $e_1,e_2$ be a basis of $J$. There is an isotropic $f_1$ such
that $e_1\cdot f_1=1$. Taking the perpendicular of $\{e_1,f_1\}$ we get a sublattice 
of $h^\perp$ isometric to $N$ because there is a unique
$0$-cusp. We claim that there is an isotropic
$f_2\in N$ such that $e_2\cdot f_2=1$. Observe that $e_2$ is primitive
in $N^*$---it is primitive in $N$ and is not of the form
$r_{21}+2n$ for any $n\in N$
because the norm of any such element is nonzero.
Hence there is an $f_2$ such that $e_2\cdot f_2=1$.
Since $N$ is even, we can modify $f_2$ by a multiple of $e_2$ to ensure
it too is isotropic. Then we choose $J'=\{f_1,f_2\}$.

We can now realize any element $(\gamma,g)\in \GL(J)\times O(J^\perp/J)$
by an isometry of $h^\perp$: We declare the action on $J$ to be
$\gamma$, on $J'={\rm Hom}(J,\Z)$ to be the transpose action $\gamma^T$, and
the action on the lattice summand $(J\oplus J')^\perp\simeq J^\perp/J$ to
be $g$. Thus, the type II boundary divisor is the quotient of $\cE^{17}
\rightarrow \mathbb{H}$ by all of $\SL_2(\Z)\times O(J^\perp/J)$---we only get
$\SL_2(\Z)$ because the isometry must have spinor norm $1$.  \end{proof}

\begin{lemma}\label{lem:orbs-conn-ell}
  For the connected elliptic subdiagrams $G$ one has
  $M_G=R_G$ except for the following diagrams given up
  to $S_3$, where the quotient $M_G/R_G$ is
  \begin{enumerate}
  \item $\bZ_2$ for $A_1^\irr=[ 23]$;
    \quad $\pA_9'$, $\pA_{15}'$, $D_{10}'$, $D_{16}'$; 
    \newline $A_{17}=[ 3, \dotsc, {1}]$, $\mA_{17}^-=[ 4, \dotsc, {2}]$,
    and $D_{18} = [ 18, {17}, 0, \dotsc, {15}]$.
  \item $\bZ_6$ for $A_{17}=[ 1, \dotsc, {17}]$.
  \end{enumerate}
\end{lemma}
\begin{proof}
  For a vector $u\in M_\bQ$, one has $u\in M \iff (u, v)\in\bZ$ for
  all $v\in N$, i.e. iff $(u,r_i)\in\bZ$ for the 24 roots $r_i$. Now,
  for each of the lattices $\Lambda = R_G$ we check the finitely
  many vectors in $\Lambda^*/\Lambda$ and see for which of them all
  the intersection numbers with the 24 roots $r_i$ are integral. As
  usual, $A_n^*/A_n = \bZ_{n+1}$, $D_n^*/D_n = \bZ_2^2$ or $\bZ_4$
  for $n$ even or odd, and $E_n^*/E_n = \bZ_{9-n}$.
\end{proof}

\begin{example}
  For the lattice $\pA_9'$ the vector
  $u = \frac12(r_{18}+r_1+r_3+r_5+r_{19})\in R_G\otimes\bQ$ in
  fact lies in $M$ because $(u,r_j)\in\bZ$ for all roots $r_j$.  Note
  that $u \equiv \varpi_5 \mod A_9$, the fundamental weight of the $A_9$
  lattice for the middle node.
  Similarly for the $A_{17}$ diagram in (2), the vector
  $u = \frac16 \sum_{i=1}^{17} ir_i$ is in $(R_G\otimes\bQ) \cap M$.
\end{example}

\subsection{Generalized Coxeter semifan}
\label{sec:semifan}

We start with a more general situation and then specialize to our
case. Let $N$ be a hyperbolic lattice of signature $(1,r-1)$,
$\cC\subset N_\bR$ the positive cone and $\opC$ its rational closure.
Let $W\subset O(N)$ be a discrete group generated by reflections in
vectors $\{r_k \in N \mid k\in K\}$ such that $r_k^2<0$ and
$r_k\cdot r_{k'}\ge0$ for $k\ne k'$. Let
\begin{displaymath}
  \fund = \{v \mid v\cdot r_k\ge0 \} \cap \opC = \cap_k H^+_{r_k} \cap\opC.
\end{displaymath}
be the fundamental domain of $W$. Then $P=\bP(\fund)$ is a polytope in
a hyperbolic space whose faces by Vinberg
\cite{vinberg1985hyperbolic-groups} admit a description similar to 
\eqref{thm:faces-fund-chamber}.

\begin{definition}\label{def:semifan}
  We split the set $K=I\sqcup J$ into two subsets of \emph{active} and
  \emph{inactive mirrors}. We call a face
  $\fund \cap_{j\in V} r_j^\perp$ of $\fund$ \emph{irrelevant}
  if $V\subset J$.
  Let $W_J = \la w_j, \ j\in J\ra$.  We define a bigger chamber 
  \begin{math}
    \Ch = \cup_{h\in W_J} \ h(\fund)
  \end{math}
  and a \emph{generalized Coxeter semifan} $\fF^\semi$ as the one
  whose maximal cones are $g(\Ch)$ for $g\in W$.
\end{definition}

\begin{proposition}\label{prop:semifan-general} ${}$
  \begin{enumerate}
  \item One has
    $\Ch = \cap_{i\in I,\ h\in W_J} H^+_{h(r_i)}$.  In particular, $\Ch$ is
    convex and locally finite.
  \item The stabilizer group of $\Ch$ in $W$ is $W_J$.
  \item The support of $\fF^\semi$ is $\opC$.
  \item The cones in $\fF^\semi$ are $g(F)$ for $g\in W$ and the
    relevant faces $F$ of~$\fund$.
  \end{enumerate}
\end{proposition}
\begin{proof}
  Consider a single reflection $w_j$ in a vector $r_j$, $j\in J$ and a
  neighboring chamber $w_j(\fund) = \cap_i H^+_{r'_k}$, where
  $r'=w_j(r_k)$. Then for $i\ne j$ and $v\in \fund$ one has
  \begin{displaymath}
    r_i \cdot w_j(v) =
    r_i  \cdot \big(v - \frac{2 r_i\cdot r_j}{r_j^2} r_j \big)
    \ge r_i\cdot v = w_j(r_i) \cdot w_j(v).
  \end{displaymath}

  A product of two generators of $W_J$ is
  $w_j w_{j'} = (w_jw_{j'}w_j\inv) w_j$ which is the same as the
  reflection in the inactive mirror $r_j^\perp$ followed by the
  reflection in an inactive mirror of the neighboring chamber
  $w_j(\fund)$. In the same way, any element $h\in W_J$ is a product
  of reflections in inactive mirrors in a sequence of neighboring
  chambers.
  
  By induction we get 
  \begin{math}
    r_i\cdot h(v) \ge h(r_i) \cdot h(v) = r_i \cdot v \ge0
  \end{math}
  for any $i\in I$ and $h\in W_J$.  Thus, $\Ch \subset H^+_{h(r_i)}$.

  Vice versa, suppose $v\in \opC$ is such that $h(r_i)\cdot v\ge0$ for
  all $i\in I$ and $h\in W_J$. Let $\rho\in N$ be a vector in the
  interior of $\fund$. Then $\rho \cdot r_k \in \bZ_{>0}$ for all
  $k\in K$. If there exists $j\in J$ such that $v\cdot r_j <0$ then
  \begin{displaymath}
    \rho \cdot w_j(v) =
    \rho \cdot \big(v - \frac{2 v \cdot r_j}{r_j^2} r_j \big)
    < \rho \cdot v.
  \end{displaymath}
  Both $\rho\cdot v$ and $\rho\cdot w_j(v)$ are positive integers and
  the set of vectors $v$ with $0< \rho\cdot v \le \text{const}$ is
  finite. Therefore, after finitely many reflections in $h'(r_j)$,
  $h'\in W_J$, we arrive at an element $h(v)$, $h\in W_J$, such that
  $r_j\cdot h(v)\ge0$ for $j\in J$. For all $i\in I$ we already have
  $r_i\cdot h(v) = h\inv(r_i) \cdot v\ge0$. Thus, $h(v) \in \fund$ and
  $v\in h\inv(\fund)$. This proves (1).
  Parts (2,3) are immediate.

  For (4), clearly each face of $\Ch$ is
  of the form $g(F)$ for some face $F$ of $\fund$. A face
  $F= \fund \cap_{i\in V} r_i^\perp$ is \emph{not} a face of $\Ch$ if
  the images $g(\fund)$ for $g\in W_J$ cover its open
  neighborhood. This happens when $W_V \subset W_J$, i.e. $V\subset
  J$ and $F$ is irrelevant.
\end{proof}

We now apply this to our lattice $N=H\oplus E_8^2\oplus A_1$, the 24
roots $r_k$, and the sets $I=\{0,\dotsc,17\}$ and
$J=\{18,\dotsc,23\}$. In this situation the cone $\Ch$ 
has infinitely many faces and an infinite stabilizer group
in $O(N) = S_3\ltimes W(N)$. This explains the name \emph{semifan}
that we use for the generalized Coxeter semifan $\fF^\semi$.

\begin{corollary}\label{cor:Fsemi-Fcox}
  The semifan $\fF^\semi$ is a coarsening of the Coxeter fan
  $\fF^\cox$. For two elliptic or maximal parabolic subdiagrams
  $G_1,G_2\subset G_\vin$ the corresponding cones of $\fF^\cox$ map to
  the same cone in $\fF^\cox$ iff $G_1^\rel = G_2^\rel$.
\end{corollary}

\begin{remark}
  The same construction applies to an elliptic lattice or parabolic
  ambient diagram $G_\cox$.
  When the subdiagram $J$ is elliptic, the Weyl group $W_J$
  is finite. In this case the resulting semifan is a fan and it can be
  alternatively defined as the normal fan of a permutahedron.
		
  The fan $\fF^\cox$ itself is the normal fan of the
  permutahedron, an infinite polyhedron $\Conv(W.p)$ for a point $p$
  in the interior of $\fund$. If $q$ is chosen to be on a
  lower-dimensional face of $\fund$ for an elliptic subdiagram $J$,
  with a finite Weyl group $W_J$, then $\fF^\semi$ is again the normal
  fan of the permutahedron $\Conv(W.q)$. This is basically the
  ``Wythoff construction'' for the uniform polytopes in Coxeter
  \cite{coxeter1935wythoffs-construction}.
\end{remark}

Looijenga \cite{looijenga1985semi-toric,
  looijenga2003compactifications-defined2} has constructed a
generalization of both the Baily-Borel and toroidal compactifications
of an arithmetic quotients $\Gamma\backslash\bD$ of a symmetric
Hermitian domain. The starting data is a semifan supported on $\opC$
in which the cones are not assumed to be finitely generated or to have
finite stabilizers. For example, the Baily-Borel compactification
corresponds to the semifan consisting only of the cone $\opC$ itself
and its null rays.

\begin{definition}
  Let $\oF_2^\semi$ be the semi-toric compactification for the
  generalized Coxeter semifan~$\fF^\semi$.
\end{definition}

\begin{theorem}\label{tor-to-semi}
  There is a morphism $\oF_2^\tor\to \oF_2^\semi$, an isomorphism on $\Gamma\backslash \bD$, whose induced map on strata
  is isogenous to the natural map of tori $$\Hom(M_G, \bC^*)\to \Hom(M_{G^\rel}, \bC^*),\quad \Hom(M_G, \cE)\to \Hom(M_{G^\rel}, \cE)$$ in Types III, II, respectively.
  \end{theorem}
\begin{proof}
  This follows directly from Proposition \ref{prop:semifan-general}, Corollary \ref{cor:Fsemi-Fcox},
  and the functoriality of the semi-toric construction under refinement of semifans. 
\end{proof}

Note that $|G|-|G^\rel|\leq 3$, with the maximum achieved when there are
three $A_1^\irr$ diagrams in $G$. So the largest fiber dimension of the morphism
in Theorem \ref{tor-to-semi} is $3$. For the Type II boundary strata, $G=G^\rel$ except
when $G= \wE_8^2\wA_1$ in which case, the morphism of Theorem \ref{tor-to-semi} loses
$1$ dimension.

\section{The mirror surfaces}
\label{sec:mirror-k3}

In Dolgachev-Nikulin-Voisin mirror symmetry for K3 surfaces
\cite{dolgachev1996mirror-symmetry}, the 19-dimensional moduli space
$F_2$ of polarized K3 surfaces with a rank-1 Picard lattice $\bZ h$,
$h^2=2$, is mirror-symmetric to the 1-dimensional moduli space of
lattice-polarized K3 surfaces with a primitive rank 19 sublattice
$H\oplus E_8^2\oplus A_1 \subset \Pic S$. We describe the latter
explicitly, and show that for a general surface $S$ its nef cone can be
identified with the fundamental chamber $\fund$ of the Coxeter fan
$\fF^\cox$. 

The K3 surfaces in this family admit several elliptic fibrations, one
of which contains an $I_{18}$ Kodaira fiber. It turns out that they
also come with an involution that fixes this $I_{18}$ fiber, and the
quotient surface $T = S/\iota$ is a non-minimal rational elliptic
surface with an $I_9$ Kodaira fiber in its minimal form.

\begin{figure}[!htp]
  \begin{center}
    \begin{minipage}[c]{175pt}
      \centering
      \includegraphics[width=\linewidth]{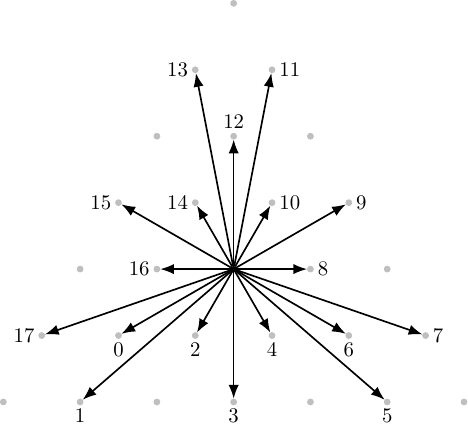}
    \end{minipage}
    \begin{minipage}[c]{175pt}
      \centering
      \includegraphics[width=\linewidth]{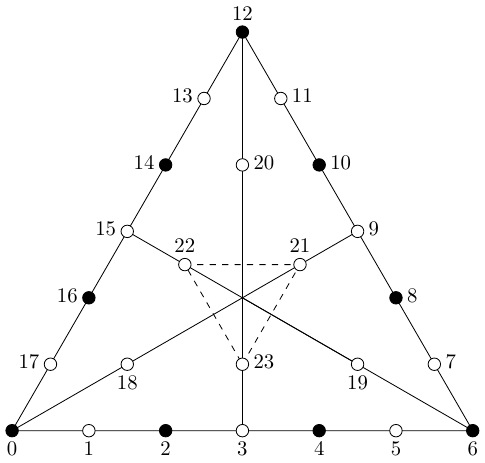}
    \end{minipage}
    \caption{Fan of the toric surface $\oT$ and the dual graph of
      negative curves on the surface $T=\Bl_{p_0,p_6,p_{12}} (\oT)$}
    \label{fig:fan}
  \end{center}
\end{figure}

\subsection{A toric model}
\label{subsec:toric-model}

We begin with a toric surface $\oT$ whose fan is depicted in
Fig.~\ref{fig:fan} on the left. It is easy to see that $\oT$ is smooth
and projective. For each ray we have a boundary curve $\oF_i$. One has
$\oF_i^2=-3$ for $i=0,6,12$, $\oF_i^2=-4$ for other even $i$, and
$\oF_i^2=-1$ for odd $i$. The Picard rank is $\rho(\oT)=16$.  There
are three toric rulings $\oT\to\bP^1$ corresponding to the opposite pairs of rays
numbered $0,9$, resp. $6,15$ and $12,3$.

\subsection{A rational elliptic surface}
\label{sec:rational-elliptic}

We define $T$ as the blowup of $\oT$ at three points $P_i\in \oF_i$,
$i=0,6,12$, each corresponding to the identity $1\in\bP^1$ under the
torus action. Let the exceptional divisors of this blowup be $F_{18}$,
$F_{19}$, $F_{20}$, and let $F_i$ for $0\le i<18$ be the strict
transforms of the divisors $\oF_i$ on $\oT$. 

The fiber over $P_0$ in the first ruling defined above is, after pullback,
$F_{18}+F_{21}$, where $F_{21}$ is a $(-1)$-curve intersecting
$F_9$. Similarly, the pulled back fiber of the second fibration over $P_6$ is
$F_{19}+F_{22}$, and the pulled back fiber of the third fibration over $P_{12}$ is
$F_{20}+F_{23}$. One has $F_i^2=-4$ for the even $0\le i<18$, and $-1$
for all other $i$. The intersection graph of $F_i$'s is given in
Fig.~\ref{fig:fan} on the right. The black vertices correspond to the
$(-4)$-curves and white vertices to the $(-1)$-curves.
For the solid edges one has $F_i\cdot F_j=1$, and for the
dashed edges $F_i\cdot F_j=3$.

The divisor $F = \sum_{i=0}^{17} F_i$ satisfies $\cO_F(F)\simeq\cO_F$
and defines an elliptic fibration $T\to\bP^1$. Contracting the nine
$(-1)$-curves $F_1,F_3,\dots,F_{17}$ gives a relatively minimal
elliptic fibration with an $I_9$ Kodaira fiber and three $I_1$
fibers. This is the extremal elliptic surface $X_{9111}$ in the
terminology of \cite[Thm.4.1]{miranda1986on-extremal-rational}; it has
three sections and three bisections, given by $F_i$ for $18\leq i< 24$.
The exceptional curves not lying in the fibers are precisely the
sections. Thus, $F_i$ for $0\le i<24$ are all the negative curves
on $T$.

\subsection{An elliptic K3 double cover}

Let $\pi\colon S\to T$ be the double cover ramified in the nine
$(-4)$-curves $F_0,F_2,\dotsc, F_{16}$ and another fiber $F'$ of the
elliptic fibration. Since there are three special $I_1$ fibers, one
gets a 1-parameter family of such surfaces,
with three members of the family having a rational
double point. For a very general choice of $F'$ one has
$\rho(S)=19$. A more detailed discussion of the moduli space of these
mirror K3s may be found in \cite[Sec.5]{doran2015families-of-lattice}.

For the preimages $E_i$ of the exceptional curves one has
$\pi^*(F_i)=2E_i$ for the even $0\le i<18$ and $\pi^*(F_i)=E_i$ for
all other $i$. Then $E_i^2 = -2$ for all $0\le i<24$ and the
intersection graph of $E_i$'s is the Coxeter graph of
Fig.~\ref{fig:vinberg}. Thus, $E_i$ generate a 19-dimensional lattice
$N=H\oplus E_8^2\oplus A_1$.  Since $\det N=2$ is square-free, it
follows that $\Pic S = N$. Thus, $S$ is a 2-elementary K3 surface
described by Nikulin and Kondo. Note that the graph of the $(-2)$-curves in
\cite[Fig.1]{kondo1989algebraic-k3} is exactly our Coxeter
graph.

The elliptic fibration on $T$ induces an elliptic fibration on $S$
with an $I_{18}$ fiber, which is $\wA_{17}$ in Dynkin notation. The
preimage of a ruling on $T$ for the rays $0,9$
(or $6,15$ or $12,3$) gives an elliptic fibration on $S$ with
$\wE_8\wE_8\wA_1$ fibers). The preimage of a ruling for the
rays $2,10$ (or $4,14$ or $8,16$) gives an elliptic fibration with
$\wD_{10}\wE_7$ fibers. The three subdiagrams $\wD_{16}\wA_1$ give yet
three more elliptic fibrations on $S$ which also double cover
rulings on $T$, see section~\ref{sec:second-toric}.

\subsection{The nef cones of the rational and K3 surfaces}

\begin{lemma}\label{lem:nef-cone}
  For a surface $S$ as above with $\rho=19$ the nef cone is a finite,
  polyhedral cone equal to
  \begin{displaymath}
    \Nef(S) = \{\lambda \mid \lambda\cdot E_i\ge0 \mid 24 \text{ curves } E_i \}
  \end{displaymath}
  Under the identification $\Pic(S)=N$, it maps isomorphically to a
  fundamental chamber $\fund$ of the Coxeter fan $\fF^\cox$.  The double cover
  defines identifications
  \begin{displaymath}
    \pi^*\colon\Pic(T)_\bQ\isoto\Pic(S)_\bQ, \qquad
    \pi^*\colon\Nef(T)\isoto\Nef(S).
  \end{displaymath}
  However the lattice structures on $\Pic(T)$ and $\Pic(S)$ are
  different.
\end{lemma}
\begin{proof}
  The nef cone of an algebraic surface
  is the intersection of the closure
  of the positive cone
  $\pC = \{\lambda \mid \lambda^2>0,\ \lambda\cdot h>0\} \subset \Pic(S)_\bR$ with the half
  spaces $\lambda \cdot E\ge0$ for the irreducible curves $E$ with $E^2<0$.  By
  \cite{kondo1989algebraic-k3} the 24 curves $E_i$ are the only
  negative curves on $S$. We thus get the same inequalities that define a
  fundamental chamber $\fund$ of $\fF^\cox$.

  The pullback of a negative curve is a sum of negative curves. Thus,
  $F_i$ for $0\le i<24$ are the only negative curves on $T$, and
  $\pi^*\colon\Nef(T)\isoto\Nef(S)$.
\end{proof}

\subsection{A second toric model}
\label{sec:second-toric}

\begin{figure}[!h]
  \begin{center}
      \includegraphics[width=300pt]{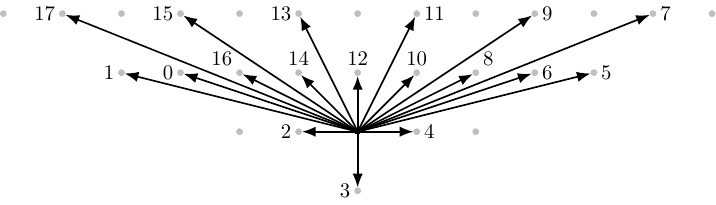}
      \caption{Fan of the toric surface $\ov{\oT}$}
    \label{fig:fan-unigonal}
  \end{center}
\end{figure}

Let $T\to \ov{\oT}$ be the contraction of the disjoint $(-1)$-curves
$F_{19},F_{20},F_{21}$. Just as $\oT$, the surface $\ov{\oT}$ is also
a smooth projective toric surface with $\rho(\ov{\oT})=16$. Its fan is
shown in Fig.~\ref{fig:fan-unigonal}.


\section{Family of $\ias$ over the Coxeter fan}
\label{sec:ias-over-coxeter}

We now define a
family of polarized $\ias$ over the Coxeter fan (our ``Voronoi" decompositions).
We motivate the construction with mirror symmetry.

As we saw in
Section \ref{sec:reflection-fan}, a compactification of $F_2$
is governed by a fan decomposition $\mathfrak{F}$ of the rational closure
of the positive cone of $N=H\oplus E_8^{\oplus 2}\oplus A_1$. Each $\lambda\in N$,
$\lambda^2\geq 0$ determines a Picard-Lefschetz transformation of a one-parameter
degeneration of complex structure, whose logarithm is given by
$(\log T)\cdot x = (x\cdot \delta)\lambda - (x\cdot \lambda)\delta$.
Mirror symmetry dictates that the complex moduli of $F_2$ is interchanged with the
K\"ahler moduli of the Dolgachev-Nikulin-Voisin mirror K3 surface $S$ from
 Section \ref{sec:mirror-k3}. This is instantiated in the isomorphisms ${\rm Pic}(S)=N$,
 ${\rm Nef}(S)=\fund$.

To make the mirror correspondence more precise, consider some $\lambda\in N$,
$\lambda^2>0$. The symplectic geometry of $(S,\omega)$ in K\"ahler class
$[\omega]=\lambda$ should be interchanged with the complex geometry of
a degenerating family of degree $2$ surfaces, whose monodromy vector is $\lambda$.
We have a mechanism for this interchange---the Monodromy Theorem
of Section \ref{sec:monodromy}. It states that the $\ias$ on the base
of a Lagrangian torus fibration $\mu\colon (S,\omega)\to B$ 
should be identified with the dual complex $B=\Gamma(\cX_0)$ of a Kulikov degeneration
$\cX\to (C,0)$ whose monodromy vector is $\lambda$.

Finally, we recall the construction $S\to T$ as a double cover of a rational surface. This motivates
a construction of $B$ for any monodromy vector $\lambda$, and thus any 
Type III degeneration:
We should produce a Symington polytope $P$ for the rational surface $T$, 
then glue two copies $B= P\cup P^{\op}$ together
to form an $\ias$ which is the base of a Lagrangian torus fibration $\mu\colon (S,\omega)\to B$
satisfying $[\omega]=\lambda$.

We also give an explicit description of the Type II degenerations
corresponding to the cusps of $\fund$, when $\lambda^2=0$.

\subsection{Construction of $\ias$} 
\label{subsec:def-ias}

Let $\pi\colon S\to T$ be the double cover of a special K3 rational
by a special K3 surface as defined in
Section~\ref{sec:mirror-k3}.  Let $L \in\Pic(T)\otimes\bR$ be a nef
class. Let $a_i = \pi^*(L)\cdot E_i$ for the $(-2)$-curves $E_i\subset S$
and let $b_i = L \cdot F_i$ for the $(-1)$- and $(-4)$-curves on $T$.  Thus,
$a_i = b_i$ for the even $0\le i<18$, and $a_i = 2b_i$ for all other~$i$.

Let $\phi\colon T\to\oT$ be the blowup of the first toric model, which
contracts exceptional curves $F_{18}$, $F_{19}$, $F_{20}$
meeting $F_0$, $F_6$, $F_{12}$. Set $\bar b_i = \oL\cdot \oF_i$. Then
\begin{eqnarray*}
  &&L = \phi^*(\oL) - b_{18}E_{18} - b_{19}E_{19} - b_{20}E_{20},\\
  &&b_0 = \ob_0 - b_{18}, \quad
  b_6 = \ob_6 - b_{19}, \quad
  b_{12} = \ob_{12} - b_{20}, \quad b_i=\ob_i
  \text{ for other } i.
\end{eqnarray*}

\begin{figure}[!htp]
  \begin{center}
    \includegraphics[width=290pt]{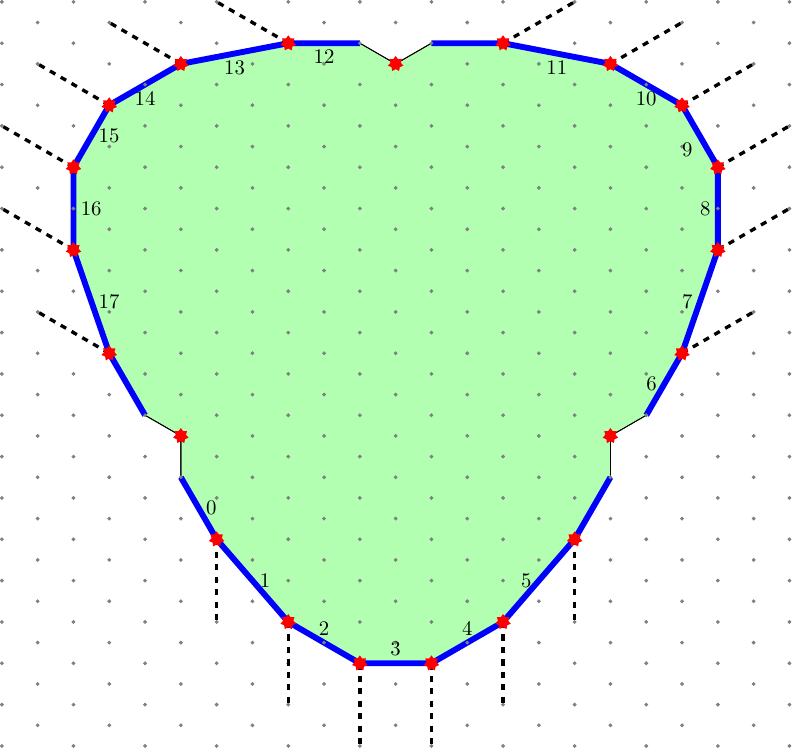}
    \vskip .5cm
    \includegraphics[width=290pt]{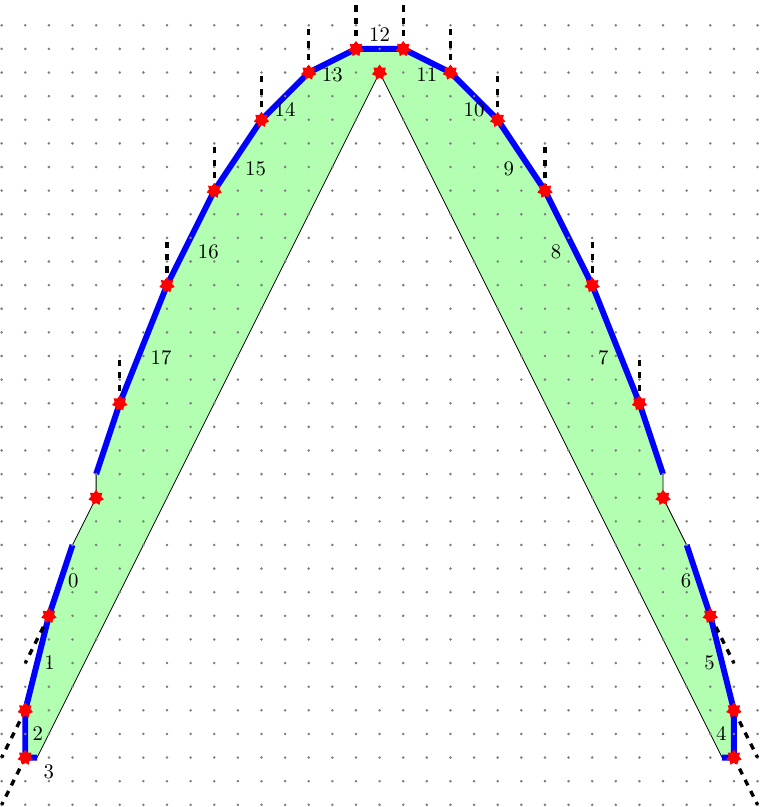}
    \caption{An example of the same $\ias$ glued in two different ways}
    \label{fig:ias}
  \end{center}
\end{figure}

\begin{construction}\label{con:ias}
  In Lemma~\ref{lem:nef-cone} we identified the nef cones of $S$ and
  $T$ with a fundamental chamber $\fund$ of the Coxeter fan $\fF^\cox$.
  So let $L=\vec{b} \in \Nef(T)=\fund$ be a nef
  $\bR$-divisor.

  First assume that all $b_i>0$; a fortiori all $\ob_i>0$. Let $P$
  be the Symington polytope obtained from the moment polytope $\oP$
  for $\overline{L}$ by three almost toric blow-ups of size
  $b_{18},b_{19},b_{20}$ on sides 0, 6, 12 as shown in Figure
  \ref{fig:ias}. This introduces three $I_1$ singularities in the
  interior of $P$ whose monodromy-invariant lines parallel the side
  from which the surgery triangle was removed. So $P$ is an
  integral-affine disc with 18 boundary components.
  The location of the cut on the sides 0, 6, 12 can be chosen
  arbitrarily; ultimately choices will produce Kulikov models
  differing by ``nodal slides'' defined below, which do not affect
  anything.
We make \emph{the symmetric choice:} with the cut centered
  around the midpoint of the side.

  By \cite[Thm.5.3]{engel2021smoothings} the class $L$ on $T$ is
  nef iff it is possible to fit surgery triangles of the appropriate
  size inside the polytope for a toric model without overlapping. In
  our case, this is also easy to see directly.

Define an integral-affine sphere $B:=P\cup P^\op$ by gluing together
two copies of~$P$. This requires introducing an $I_1$ singularity at
each corner of $P$, whose monodromy-invariant lines are shown dashed
in Fig.~\ref{fig:ias}. More precisely, we can take top figure for $P$ in Fig.~\ref{fig:ias},
and take its isometric reflection along the edge $3$ (with edges $9$ or $15$, it is similar).
This produces a copy of $P^{\op}$ attached to $P$ along $3$, but there is a gap between
edge $4$ and its reflection. This gap is closed exactly by gluing edge $4$ and its
reflection with a unit shear in the dotted direction. Once this gluing is made, we must
introduce another singularity to glue edge $5$ and its reflection. And so on for
edges 0, 1, 2, 3, 4, 5, 6 (and similarly for the other edges).

The general case is obtained as a
limit of the above construction by sending some of the $b_i$'s to zero.
\end{construction}

\begin{definition}\label{def:ias-voronoi}
  For any real vector $\vec a=(\lambda\cdot r_i)_{i\in \{0,\dots,23\}}$ with
  $\lambda\in\fund$, $\lambda^2>0$, this construction defines an integral affine
  structure $B(\vec a)$ on a sphere with 24 singularities,
  some of which may coalesce, an $\ias$
  for short. We sometimes suppress the dependence on $\vec{a}$.

  When all $a_i>0$, we define an integral-affine divisor $R_\ia$ whose
  supporting graph is the equator, that is, the common boundary of $P$ and
  $P^\op$. The multiplicities are $2$ for the even sides and $1$
  for the odd sides. The assumption $a_i>0$ implies that
  the $\ias$ has 18 isolated $I_1$ singularities on
  the equator. By Remark \ref{balancing1}, this suffices to define
$R_\ia$ uniquely.

When some $a_i=0$, the definition of 
$R_\ia$ is quite subtle. It is delayed until Section \ref{sec:ias-polarization},
but the supporting graph is still the equator,
  and the multiplicities are the same values, 1 and 2, for the odd and even sides
  $i\in\{0,\dots,17\}$ with $a_i\ne0$.
\end{definition}

The pair $(B, R_\ia)$ is an analogue of a Voronoi decomposition in the
case of abelian varieties. As $\vec a$ varies continuously, so do
they. 

\begin{lemma}\label{lem:ias-volume}
  One has $(\pi^*L)^2 = 2L^2 = \vol(B)$, where the latter is the lattice
  volume, twice of the Euclidean one.
\end{lemma}
\begin{proof}
  By definition, $\vol(B)=2\vol(P)$ and
  $\vol(P) = \vol(\oP)- b_{18}^2-b_{19}^2-b_{20}^2$. It is easy to see
  that $L^2=\oL^2-b_{18}^2-b_{19}^2-b_{20}^2$.  For any toric variety
  with a nef class, its volume is the lattice volume of the moment
  polytope; this gives $\oL^2=\vol\oP$.
\end{proof}

\begin{remark}
  By definition, $b_{18}$ is the lattice distance from the singularity
  to the side 0. The linear relation
  $n\big(\wE_8\big)= n\big(\wA_1\big)$ of \eqref{lem:N-linear-rels}
  implies that $b_{21}$ is the lattice distance to the opposite side
  9.  Similarly for $b_{19},b_{22}$ and $b_{20},b_{23}$.
\end{remark}

\begin{example}
  Fig.~\ref{fig:ias} shows a concrete example with
  \begin{flalign*}
    &\ob_0 = \ob_6=\ob_{12}=3, \quad b_2=b_4=\dotsb=b_{16}=2,
    \quad b_1=b_3=\dotsb=b_{17}=1, \\
    &b_{18} = b_{19} = b_{20}=1, \quad b_0=b_6=b_{12}=2,
    \quad b_{21} = b_{22} = b_{23} = 29.
  \end{flalign*}
  The green interior region is an open chart for the integral-affine
  structure on the disc $P$. In the $a$-coordinates, $a_i=2\cdot 1$ for
  $0\le i<21$ and $a_{21}=a_{22}=a_{23}=2\cdot 29$.
  
  The second picture gives an alternative way of presenting the same
  $\ias$, using the second toric model $\ov{\oT}$. It is obtained by
  cutting a different ray emanating from the $I_1$ singularity, which, instead
  of hitting the edge $12$, goes in the opposite direction, towards edge $3$. \end{example}

Recall that in \eqref{lem:N-linear-rels} we defined the vectors
$n(\wt\Lambda)$ for the affine Dynkin diagrams $\wA_{17}$, $\wE_8$,
$\wA_1^\irr$, $\wD_{10}$, $\wE_7$, $\wD_{16}$ and $\wA_1^*$, using
notations of Table~\ref{tab:vin-conn-subdiags}.

\begin{lemma}\label{lem:height-width}
  The circumference in the vertical direction, that is twice the
  lattice distance in $P$ and in $\oP$ between the sides $3$ and $12$ is
  $\ev(n_{3,12})$, where
  \begin{displaymath}
    n_{3,12} = n\big(\wE_8^{(1)}\big)= n\big(\wE_8^{(2)}\big)=
    n\big(\wA_1^\irr\big)
    \quad\text{and}\quad \ev(r_i)=a_i    
  \end{displaymath}
  is the evaluation map. Similarly, the circumference in the $8$-$16$
  direction is $\ev(n_{8,16})$ for
  $n_{8,16} = n\big(\wD_{10}\big) = n\big(\wE_7\big)$; the
  circumference in the $2$-$4$ direction in the second presentation
  (i.e. around a singularity, close to the side 12)
  is $\ev(n_{2,4})$  with
  $n_{2,4} = n\big(\wD_{16}\big) = n\big(\wA_1^*\big)$; and the
  circumference along the equator is $\ev\big( n(\wA_{17}) \big)$.
\end{lemma}
\begin{proof}
  This follows by observation using Lemma~\ref{lem:N-linear-rels}.
\end{proof}

\begin{example}
  In the example of Fig.~\ref{fig:ias} all the $a_i=2$ for
  $i\ne 21,22,23$. It follows and is indeed very amusing to observe
  that the projections of sides 13, 14, 15, 16, 17, 0, 1, 2 and 18 to
  a vertical line have lattice lengths 1, 2, 3, 4, 5, 6, 4, 2 and 3,
  which are the multiplicities of the simple roots in 
  $n(\wE_8)$.  Similarly, the projections of the sides 18, 17, 0,
  \dots, 6, 7, 19 to a horizontal line have lattice lengths 1, 1,
  2, \dots, 2, 1, 1, which are the multiplicities for $\wD_{10}$, and
  similarly for $\wE_7$. 
\end{example}

\begin{corollary}
  Near the rays $L^2=0$ of $\fund=\Nef(S)$ the sphere $B$ with its
  integral-affine structure degenerates to an interval as follows:
  \begin{enumerate}
  \item $\wA_{17}$. The Symington polytope $P$ degenerates to a
    segment from the boundary of $P$ to the north pole, and $B$
    degenerates to a longitude.
  \item $\wD_{10}\wE_7$. Both $P$ and $B$ degenerate to the side $8$,
    identified with the side $16$.
  \item $\wE_8^2\wA_1$. Both $P$ and $B$ degenerate to the side $3$,
    identified with the side $12$.
  \item $\wD_{16}\wA_1$. Both $P$ and $B$ degenerate to the side $2$,
    identified with the side $4$.
  \end{enumerate}
  In the cases (2,3,4) the interval lies in the equator.
\end{corollary}

\begin{definition}\label{def:family-ias}
  We define the family of $\ias$ over the fundamental chamber $\fund$
  by Construction~\ref{con:ias}.  By Lemma~\ref{lem:ias-volume}, this
  gives a family outside of the boundary rays with $L^2=0$, where
  $\ias$ degenerates to an interval.
  
  We then extend it to a family over $\opC$ by reflections in the Weyl
  group $W(N)$. This is well defined because $\fund$ is a fundamental
  domain of the reflection group and because on the boundary of $\fund$
  where some $a_i=0$ the limits of the structures from both sides coincide.
\end{definition}

\begin{remark}
  As we mentioned, the locations of the cuts on the sides 0, 6, 12 are
  quite arbitrary and may be moved by a ``nodal slide''.
  Instead of the symmetric choice for the cuts, one could also make a
  ``vertex-preferred" choice: For this
  choice, if $b_i$ are integral then the coordinates of the three
  internal singularities are also integral.  For the symmetric choice
  they are only half-integral.
  
  This is quite similar to the case of abelian varieties where, given
  an integral positive-definite symmetric bilinear form
  $B\colon M\times M\to \bZ$ on $M\simeq\bZ^g$, the Voronoi
  decomposition $f_B(\Vor B)$ in $N_\bR=M^*_\bR$ has only
  half-integral coordinates but in low dimensions there is a
  ``vertex-preferred'' 
  linear shift $\ell$ so that $\ell + f_B(\Vor B)$ has
  vertices in the lattice $N=M^*$.
\end{remark}

\subsection{Collisions of singularities in $\ias$}
\label{sec:sings-ias}

We now describe how the 24
singularities collide and the resulting singularities of the
integral-affine structures.

\begin{theorem}\label{thm:collapsing-ias}
  For a big and nef class $L\in\Nef(S)$, the possibilities for
  the collisions of the 24 singular points are in bijection with the
  elliptic subdiagrams $G$ of the Coxeter diagram
  $G_\vin$, excepting \eqref{num:exceptions}.
  Each collision point, excepting \eqref{num:exceptions}, is in bijection with
   a connected component $G_k$ of $G$.
\end{theorem}
\begin{proof}
  In \eqref{lem:nef-cone} we made an identification of $\Nef(S)$ with
  the fundamental chamber~$\fund$. Now we simply apply
  Theorem~\ref{thm:faces-fund-chamber}. With the noted insignificant
  exceptions, the collisions correspond to the collections
  of indices $\{i \mid a_i=0\} \subset \{0,\dotsc, 23\}$, i.e. to the
  faces of $\fund$, by virtue of Construction \ref{con:ias}.
\end{proof}

\begin{num}\label{num:exceptions}
  The exceptions, which play no role in the end, occur as artifacts
  of the ``symmetric choice" of cuts for the Symington polytope $P$. A collision 
  is insignificant if a different choice of cut would get rid
  of the collision, for instance, when two cuts are made that have the same
  apex in the interior of $P$.
\end{num}

\begin{lemma}
  The singularities appearing $B(\vec{a})$ as some collection of $a_i\to 0$ for
  $i\in G_k$ (such subdiagrams are listed in Table \ref{tab:vin-conn-subdiags})
  are exactly those listed in Table~\ref{tab:ias-sings-without-ram} with the same Dynkin
  label.
\end{lemma}	

Decorations $-$ from Table \ref{tab:vin-conn-subdiags}
are dropped in Table~\ref{tab:ias-sings-without-ram} as they do not affect 
the resulting integral-affine singularity.
 
\begin{proof}
A singularity resulting from a collision as $a_i\to 0$ is determined by (and in fact,
is defined by, see Definition \ref{intaffsing}) tracking the monodromy directions of the $I_1$ singularities as they
coalesce. This presents the coalesced singular point in the form $I(n_1\vec{v}_1,\dots,n_k\vec{v}_k)$.
The results are determined by direct geometric examination of Fig.~\ref{fig:ias},
and tabulated in Table \ref{tab:ias-sings-without-ram}.

For example, the $E_8$ diagram formed from nodes $i\in \{18,16,17,0,1,2,3,4\}$ of $G_\cox$
corresponds to the coalescence where these lengths $a_i$ all approach zero. This results in the collision
of $10$ total $I_1$ singularities. Choosing monodromy-invariant cut directions for each singularity
in a counterclockwise fashion about the center
of edge 0 (like a windmill) and letting $a_i\to 0$, we see that this collision can be presented
as $I(5,1,3,1)\sim I(2,3,5)$, which is the ``$E_8$" integral affine singularity.
\end{proof}

\begin{table}[H]
  \centering
  \begin{tabular}{llllll}
    Definition&
               Name
                &Charge 
                \\
    $I(n+1)$  &$A_{n}$  &$n+1$ 
    \\
    $I(2,2,n-2)$ &$D_n$
    
    &$n+2$ 
    \\
    $I(2,3,3)$ &$E_6$ 
    &$8$ 
    \\
    $I(2,3,4)$ &$E_7$ 
    &$9$ 
    \\
    $I(2,3,5)$ &$E_8$ 
    &$10$ 
    \\
    $I(2,3,n-3)$ &$D'_{n-1}$ &$n+2$ 
    \\
    $I(n+1,1)$ &$\pA_n$ &$n+2$ 
    \\
    $I(n,n,2)$ &$\pA'_{2n-1}$ &$2n+2$ 
    \\
  \end{tabular}\vspace{4pt}
  \caption{Possible integral-affine singularities on $B(\vec{a})$ for some $\vec{a}\in \fund$}
  \label{tab:ias-sings-without-ram}
\end{table}


\subsection{Polarization of the $\ias$}\label{sec:ias-polarization}
We now define the polarizing divisor $R_{\ia}$
on $B$, when some singularities have collided.
By Definition \ref{IAdivisor} and Remark \ref{generic},
the data of $R_{\ia}$ is specified by a nef line bundle $L_i$
on an anticanonical pair $(V_i,D_i)$ for which $\mathfrak{F}(V_i,D_i)$
models each integral-affine singularity. This line bundle
is furthermore required to have intersection numbers $n_{ij} = L_i\cdot D_{ij}$ 
agreeing with the weighted balanced graph in $B$.

The graph underlying $R_{\ia}$ is supported on
the equator and has exactly two nonzero weights $n_{ij}\in \{1,2\}$
emanating from an equatorial vertex $v_i\in \mathfrak{F}(V_i,D_i)$.
These weights are notationally incorporated into 
the decorations of the corresponding Dynkin subdiagram $G_k$ by the $-$ and $'$
decorations, see the discussion following Table \ref{tab:vin-conn-subdiags}.
For each singularity, we must give an anticanonical pair $(V_i,D_i)$
in the c.b.e.c. describing the singularity, and the appropriate
line bundle $L_i\to V_i$.

\begin{theorem}\label{eq-invo}
Let $\iota_{\ia}$ be the involution of $B$ switching the hemispheres $P, P^\op$ and fixing
the equator pointwise. For each singularity $v_i$ on the equator of $B$,
consider the deformation class of ``involution pairs" $(\oV_i,\oD_i+\epsilon \oR_i)$, see
\cite{alexeev17ade-surfaces} and Section \ref{sec:ade-wade-surfaces},
notated in {\rm loc. cit.} by exactly the same decorated Dynkin symbol
of the corresponding subdiagram, listed in Table \ref{tab:vin-conn-subdiags}.

Let $\overline{\iota}_i$ be the involution, so $\oR_i = {\rm Fix}(\overline{\iota}_i)$, and let
$\iota_i$ be the induced involution on the minimal resolution $\pi_i\colon (V_i,D_i)\to (\oV_i,\oD_i)$.
Then, $v_i=\mathfrak{F}(V_i,D_i)$ as integral-affine singularities,
and $\iota_i$ induces the same the action as $\iota_{\ia}$.
Furthermore, denoting $R_i=\pi_i^*(\oR_i)$,
the nef line bundle $L_i:=\mathcal{O}_{V_i}(R_i)$
has intersection numbers $n_{ij}=L_i\cdot D_{ij}$ which give the weighted
graph on the equator described in Definition \ref{def:ias-voronoi}.
 \end{theorem}

\begin{proof} Essentially, the proof is by direct calculation
of all the cases. We simply check
that $\mathfrak{F}(V_i,D_i)$ is the correct integral-affine
singularity, and $L_i\cdot D_{ij}$ are the correct values.
We perform this check below for some representative examples.
\end{proof}

\begin{remark} The proposition should not come as a surprise---the notation
for subdiagrams of $G_\cox$ was reverse-engineered so that Theorem
\ref{eq-invo} becomes true. \end{remark}

For notational convenience, we drop the index $i$.

\begin{example}[$\mA_0$ and $A_0^-$]\label{example0} $(V,D)=(\mathbb{P}^2,D_1+D_2)$ is a projective plane
with a line $D_1$ plus conic $D_2$ as anticanonical divisor. The singularity of $\mathfrak{F}(V,D)$
is an $I_1$ singularity, and the degrees for $R_\ia$ must be $1$, $2$ on the components $D_1$, $D_2$ corresponding to equatorial edges, respectively.

The pair $(V,D)$ admits an involution $\iota$ fixing another line $R$,
and an isolated point on $D_1$. The line bundle $L=\cO_V(R)=\cO_{\mathbb{P}^2}(1)$, 
which gives the correct multiplicities $1,2$ on the two equatorial edges of $R_\ia$. The two
cases $\mA_0$ and $A_0^-$ are, respectively, distinguished by whether the line is
on the left (clockwise) or right (counterclockwise) side of the equator.
\end{example}

\begin{example}[$A_{2n-1}$]\label{example1} As the singularity is $I(2n)=I(n,0,n,0)$,
see Remark \ref{allpseudo}, 
we can model the c.b.e.c. as the blow-up $$(V,D)=(V,D_1+D_2+D_3+D_4)\to (\bP^1\times\bP^1, s_1+f_1+s_2+f_2),$$
at $n$ points on $s_1$ then $n$ points on $s_2$. The edges corresponding to the
equator of $B$ correspond to the two fibers $f_1$, $f_2$ and are required
to intersect the polarization with degree $2$. There is an involution
$\iota\colon (V,D)\to (V,D)$ preserving $f$ and switching the strict transforms of $s_1$
and $s_2$, which have classes $D_1= s-\sum_{i=1}^n e_i$ and
$D_3= s-\sum_{i=n+1}^{2n} e_i$. Here $e_i$ are the
  exceptional divisors.

Assuming the points blown up are chosen generically,
the ramification divisor of $\iota$ is the strict transform of the divisor on
$\bP^1\times\bP^1$ in the linear system $|2s+nf|$ which passes through
all the $2n$ points. It has the class
$R = 2s+nf - \sum_{i=1}^{2n} e_i$ with $R^2 = 2n$.
The line bundle $L=\mathcal{O}_V(R)$
has intersection numbers $L\cdot D_2=L\cdot D_4 =2$
and $L\cdot D_1 = L \cdot D_3=0$ with the boundary. Thus, it gives
the correct intersection numbers for $R_{\ia}$ as it passes
through an $I(2n)$ singularity on the equator.

Finally, the stable model $\oV$ is the result of contracting
$D_1$ and $D_3$ which are the only curves on which $L$ has degree $0$.
The involution descends and defines the $A_{2n-1}$ involution pair
from \cite{alexeev17ade-surfaces}.
\end{example}

\begin{example}[$\pA'_{2n-1}$]\label{example2}
  As the singularity is $I(n,n,2)=I(n,1,n,1)$, the c.b.e.c. is represented
  by blowing up an $A_{2n-1}$ pair at one point on each of $f_1$, $f_2$
  with the exceptional classes $g_1$, $g_2$.
  We blow up at a point in $R\cap f_1$ and $R \cap f_2$ respectively.
  So the resulting pair $(V,D_1+D_2+D_3+D_4)$ still admits an involution
  $\iota$ lifting the involution of the $A_{2n-1}$ pair.
  
  The boundary curves have classes
  $D_1= s-\sum_{i=1}^n e_i$, $D_3= s-\sum_{i=n+1}^{2n} e_i$,
  $D_2= f-g_1$, $D_4= f-g_2$ and $R = 2s+nf-\sum_{i=1}^{2n} e_i - (g_1+g_2)$.
  The polarization is defined to be
  $L = \mathcal{O}_V(R)$ and note that $L\cdot D_2 = L\cdot D_4 = 1$
  and $L\cdot D_1=L\cdot D_3=0$ as desired.
  The stable model is again the result of contracting $D_1$ and $D_3$
  and gives the $\pA'_{2n-1}$ involution pair.
\end{example}

\begin{example}[$D_{2n}$]\label{example3}
  The easiest model for $D_{2n}$ is a pair $(V,D_1+D_2)$, whose components
   are a fiber $f$ and the $2n$-fold blow-up of
   a bisection in class $2s+f$, on $\bP^1\times\bP^1$. Taking some corner blow
  ups and a toric model, one finds the pseudo-fan is
  $\mathfrak{F}(V,D)=I(2,2,2n-2)$ as desired, and that $D_1$ and $D_2$
  correspond to the edges emanating from $v$ along the equator.
  
  There is an involution $\iota$ preserving $f$ and switching the two sheets of 
  the bisection. Its ramification divisor has class
  $R = 2s + nf - \sum_{i=1}^{2n} e_i$, $R^2 = 2n$. Setting $L=\cO_V(R)$, one has
  $L\cdot D_1 = L\cdot D_2 = 2$ as desired. In this case, $L$ is already ample,
  so the stable model is the same surface and it is the $D_{2n}$ involution pair.
\end{example}

\begin{example}[$D'_{2n}$]\label{example4}
  The $D'_{2n}$ surface is obtained from the $D_{2n}$ surface by an
  additional blowup at one of the two points $D_1\cap R$. This gives
  the singularity $I(2,3,2n-2)$, which is the same as for $E_n$, but
  in these two cases, the equator sits differently with respect to the shearing rays. The involution
  on the $D_{2n}$ pair lifts to give an involution.
\end{example}

\begin{example}[$E_n$]\label{example5} For $E_8$, the singularity 
is $I(2,3,5) = I(5,1,3,1)$, which can be represented by blowing up $5$, $1$,
$3$, $1$ points respectively on the four sides of an anticanonical square in
$\mathbb{P}^1\times \mathbb{P}^1$. Contracting the two
boundary exceptional curves gives the blow-up of a nodal cubic in $\mathbb{P}^2$
at $8$ smooth points, and at the node. Then $R$ is the fixed locus
of the Bertini involution, which intersects each boundary component with 
degree $1$, as desired. Here $L=\cO(R)$ is already ample.

For $\phmi E_7$ there are 7 blowups
  on the cubic and an additional blowup at the node, and for
  $\phmi E^-_6$ there are two more blowups at the node. The involution
  is the Geiser involution. See \cite[\S4.5]{alexeev17ade-surfaces} for
  more details.
\end{example}

\begin{remark}\label{ram-not-r} In Examples
 \ref{example1}, \ref{example2}, \ref{example3},  \ref{example4},
 \ref{example5}, the description of $R$ as ${\rm Fix}(\iota)$ is valid
only when the blow-ups are chosen generically.
This is because as we vary the blow-up points on a smooth anticanonical pair,
the fixed loci ${\rm Fix}(\iota)$ do not vary in a flat manner. 
The resolution of this issue is to work with the ADE surfaces of \cite{alexeev17ade-surfaces},
on which $\oR_i={\rm Fix}(\overline{\iota}_i)$ {\it does} vary in a flat manner. Then, the pullback
$R_i$ to the minimal resolutions also varies in a flat manner.

The same phenomenon occurs even for smooth degree $2$ K3 surfaces
acquiring a $(-2)$-curve,
such as the minimal resolution of a double cover of $\mathbb{P}^2$ ramified
over a nodal sextic. The divisor $R$ on such a smooth K3 is {\it not}
${\rm Fix}(\iota)$ but rather the pullback of ${\rm Fix}(\overline{\iota})$ for the involution
$\overline{\iota}$ of the ADE K3 surface.
    \end{remark}
 
\begin{proposition}[Reconstructing the polarization]\label{reconstructing}
The line bundles $L_i$ defining the polarization $R_\ia$ at a singularity $v_i=\mathfrak{F}(V_i,D_i)$
are uniquely characterized by:
  \begin{enumerate}
  \item the intersection numbers $n_{ij}=L_i\cdot D_{ij}\in \{0,1,2\}$,
  \item the $\iota_i$-invariance of the class of $L_i$,
  \item $L_i^2=$ the number of equatorial $I_1$ singularities
  involved in the collision.
  \end{enumerate}
\end{proposition}

\begin{proof} As for Theorem \ref{eq-invo}, this simply requires a direct check in all cases.\end{proof}

This completes the construction of a family $(B(\vec{a}), R_\ia)$ of polarized $\ias$,
varying over $\oC$, which is combinatorially constant exactly along the cones of $\mathfrak{F}_\cox$.

\subsection{Kulikov degenerations and their monodromy}

The goal of this section is to verify that the monodromy invariant
of a Kulikov model $\cX\to (C,0)$ whose central fiber satisfies
$\Gamma(\cX_0)=B(\vec{a})$ is in fact $\vec{a}\in \fund$.

\begin{definition}[The parity condition] \label{def:parity} We say
  that $\vec{a}\in\bZ^{24}$ {\it satisfies the parity condition} if
  $a_i\equiv 0\textrm{ mod }2$ for $i$ odd, and all $i\geq
  18$. Equivalently that all $b_i\in\bZ$.
\end{definition}

Let
$N=H\oplus E_8^2\oplus A_1$ be our standard lattice of signature
$(1,18)$ as in Section~\ref{sec:reflection-fan}. For each vector
$\vec a\in\bZ^{24}_{\ge0}$ coming from an integral point in $\fund$ and satisfying the
parity condition, we define a combinatorial type of
polarized Kulikov surface. Then we prove that a Kulikov
degeneration with this central fiber has the
monodromy ~$\lambda$.

\begin{construction}\label{con:ias-to-kulikov3}
  Suppose that $\vec{a}\in \fund$ satisfies the parity
  condition, so that $B(\vec{a})$ has singularities only at integral points.
  Let $\iota_\ia$ be the orientation-reversing involution of
  $B(\vec{a})$ which switches $P$ and $P^{\op}$, fixing the equator
  pointwise. Choose an $\iota_\ia$-invariant triangulation $\cT$ of
  $B(\vec{a})$ into triangles of lattice volume $1$ which contains the
  equator $R_\ia$ as a set of edges.
  
  We now apply Proposition \ref{spheretok3} to produce a Kulikov surface
  $\cX_0=\bigcup_i\,(V_i,D_i)$ for which $\mathfrak{F}(V_i,D_i)=\Star(v_i)$
  as an integral-affine surface, and $\Gamma(\cX_0)=B(\vec{a})$. This specifies
  a unique deformation type of $\cX_0$ but not its continuous moduli.
  
  To choose from the continuous moduli, first, we pick an anticanonical
  pair $(V_i,D_i)$ on the equator admitting an involution $\iota_i$ which
  induces $\iota_\ia$ on $\mathfrak{F}(V_i,D_i)$. This is possible by Theorem \ref{eq-invo}.
  Then, we glue the equatorial edges of $\cX_0$ by ensuring that $R_i$ glue
  into a Cartier divisor, i.e. $R_i\cap D_{ij} = R_j\cap D_{ji}$ as multisets.
  Finally we glue the northern and southern hemispheres of $\cX_0$
  onto this equatorial band of surfaces, in an arbitrary involution-invariant manner.
\end{construction}

  The resulting Kulikov surface $\cX_0$ admits an involution which we
  denote $\iota_0$ and which acts on $\Gamma(\cX_0)$ by $\iota_\ia$.
  Furthermore, by construction there is a Cartier divisor $R\subset \cX_0$ 
  given by $R=\cup_i R_i$.
  We show in Section \ref{sec:defs-kulikov-invo} that it is possible to 
glue so that this involutive surface $\cX_0$ is also $d$-semistable (see Section
\ref{subsec:kulikov}), but for the moment, assume this. In particular, $\mathcal{X}_0$ is smoothable by
\cite{friedman1983global-smoothings}.

\begin{definition} We write $\cX_0(\vec{a})$ for the Kulikov
surface defined in Construction \ref{con:ias-to-kulikov3} and $\cX(\vec{a})$
for a smoothing of it. \end{definition}

\begin{theorem}\label{thm:monodromy2}
  Let $\vec{a}$ satisfy the parity condition and suppose $B(\vec{a})$
  is generic.
\begin{enumerate}
\item Let $\cX(\vec{a})\to (C,0)$ be a Kulikov degeneration defined above.
\item Let $\mu\colon(S,\omega)\to B(\vec{a})$ be a Lagrangian torus
  fibration over $B(\vec{a})$.
\item Let $\phi\colon S\to \cX_t$ be the diffeomorphism of Theorem
  \ref{thm:monodromy}.
\end{enumerate}
Define $v:= \phi_*[\Sigma_{\ia}]\in \delta^\perp/\delta$ where $\Sigma_{\ia}:=\Sigma_{R_\ia}$. Then
$\{v,\delta\}^\perp/\delta$ is isometric to $N$ and the monodromy
invariant is $\lambda = \vec{a}\textrm{ mod }O^+(N).$
\end{theorem}

\begin{proof}
  By construction of $\phi$, we have that
  $\phi_*[\Sigma_\ia]\in \delta^\perp/\delta$ is invariant under the
  Picard-Lefschetz transformation, hence perpendicular to the
  monodromy invariant $\lambda$. So
  $\lambda\in \{v,\delta\}^\perp/\delta$.

  We describe a collection of $24$ spheres $\{E_i\}$ of
  self-intersection $-2$ in $(S,\omega)$, which intersect according to
  the Coxeter diagram $G_{\vin}$. They are all presented as {\it
    non-Lagrangian}
  visible surfaces. Let $\gamma_i$ for
  $i=0,\dots,17$ be the $i$-th edge of $P$. Then, the
  monodromy-invariant vectors $\alpha_i$ at the two endpoints of
  $\gamma_i$ are parallel. By Construction \ref{visibleconstruction}
  and Example \ref{minus-two-visible},
  there is a visible surface $$E_i:=\Sigma_{(\gamma_i,\alpha_i)}$$
  fibering over $\gamma_i$. Now, let $i=18$ (resp. $19,20$). Define
  $\gamma_i$ as a path which connects the singularity of $P$ over the
  edge $0$ (resp. $6,12$), to the mirror singularity in $P^{\op}$,
  crossing the edge $0$ (resp. $6,12$). As before, let
  $E_i:=\Sigma_{(\gamma_i,\alpha_i)}$ where $\alpha_i$ is the (common)
  vanishing cycle at the two endpoints of $\gamma_i$. Finally, we
  define $E_i$ for $i=21,22,23$ similarly, but this time connecting
  the singularity in $P$ to the mirror one in $P^{\op}$ via a path
  $\gamma_i$ which crosses the edge $9$ (resp. $15,3$). It is an easy
  check that if the $E_i$ are properly oriented, the intersection
  numbers $E_i\cdot E_j$ give exactly a system of roots as in
  $G_{\vin}$.

  We directly compute by perturbing and counting signed intersection
  points that $[\Sigma_\ia]\cdot [E_i] = 0$. Since the classes $[E_i]$
  generate a lattice of determinant $2$ and rank $19$, we conclude
  that $\phi_*[E_i]$ generate the rank $19$ lattice
  $\{\delta,v\}^\perp/\delta$ over $\Z$ and that this lattice is
  isometric to $N$, with the isometry identifying $\phi_*[E_i]$ and
  $r_i$.

  Finally, we wish to show $\lambda$ and $\vec{a}$ define the same
  element of $N$ modulo $O^+(N)$. We have the following formula
  for the symplectic area of a visible surface: \begin{align*}
    [\omega]\cdot E_i = \int_0^1 \det(\alpha_i ,\gamma_i'(t))\, dt& =
    a_i \end{align*} for all $i$. Hence $\lambda\cdot \phi_*[E_i]=a_i$
  for all $i$. This shows that $\lambda$ and $\vec{a}$ represent the
  same lattice point in $\fund$, i.e.
  $\lambda = \vec{a}\textrm{ mod }O^+(N)$. \end{proof}

\begin{corollary}\label{multiple3}
  The vector $v\in \delta^\perp/\delta$ is imprimitive with $v= 3w$
  and $w^2=2$.
\end{corollary}

\begin{proof}
  By taking a generic perturbation of $\Sigma_{\ia}$ off itself and counting
  signed intersections, we see that $v^2=18$. Also, $v$ lies in
  span$\{\phi_*[E_i]\}^\perp\subset \delta^\perp/\delta$, a
  one-dimensional lattice of determinant $2$, which is necessarily
  generated by a vector $w$ with $w^2=2$. Hence $v=3w$.
\end{proof}

\begin{theorem}\label{thm:monodromy3} Theorem \ref{thm:monodromy2} holds,
even when $B(\vec{a})$ is not generic. \end{theorem}

\begin{proof} The primary issue with the proof of Theorem \ref{thm:monodromy2} 
in the non-generic case is that there is no smooth Lagrangian
  torus fibration $\mu\colon (S,\omega)\to B(\vec{a})$ when $B(\vec{a})$ has more complicated
  singularities such as $\pA_{2n-1}'$. So we cannot directly apply Theorem \ref{thm:monodromy}.

  Let $\vec{a}(t)$ be a one-parameter family over
  $t\in [0,1]$ such that $a_i(t)>0$ for all $t>0$ and $a_i(0)$ results
  in a collision of $I_1$ singularities. Let $N>0$ be a large integer.
  For all $t$, let $B'(\vec{a}(t))$ be the result of performing nodal slides (Def. \ref{nodal-slide})
  of fixed length lying in $N^{-1}\Z$, to every singularity involved in a collision. Then,
  as $t\to 0$, the singularities no longer collide, and rather the collection of singularities
  of $B(\vec{a})$ are factored into $I_1$ singularities by the nodal slides. Let
\begin{align*}
  \mu(t)\colon (S(t),\omega(t))
  & \rightarrow B(\vec{a}(t))\textrm{ for }t\in(0,1] \\
  \mu'(t) \colon (S'(t),\omega'(t))
  &\rightarrow B'(\vec{a}(t))\textrm{ for }t\in[0,1]
\end{align*}
be the corresponding families of almost toric fibrations.
The fibration $\mu(0)$ doesn't exist unless $B(\vec{a})$ has all
unprimed singularities, but $\mu'(0)$ does. Define
$$\sigma(t)=[\Sigma_\ia]\in H^2(S(t),\Z).$$
The $\sigma(t)$ are identified under the
Gauss-Manin connection on the fiber bundle $S(t)\to (0,1]$. Define $\sigma'(t)$
by parallel transport of $\sigma(t)$ along the nodal slide connecting
$B(\vec{a}(t))$ to $B'(\vec{a}(t))$. It is easy to see that
$\sigma'(t)$ is also represented by a visible surface $\Sigma'_\ia(t)$
which fibers over $R_\ia(t)$ and the segments along which the nodal
slides were made. Since $\mu'(0)$ exists (as the $I_1$ singularities
no longer collide) we have that $\sigma'(0)=[\Sigma'_\ia(0)]$ is the
parallel transport of $\sigma'(t)$.

As the slide parameters lie in $N^{-1}\Z$, these parameters are integral on
the order $N$ refinement.
So $B'(a(0))[N]$ admits a triangulation into
lattice triangles. By Proposition \ref{spheretok3},
we get a Kulikov degeneration
$\mathcal{X}'[N]\rightarrow (C,0)$ such that
$\Gamma(\mathcal{X}_0'[N])=B'(a(0))[N]$.

The nodal slides
destroy the involution symmetry of $B(a(0))[N]$ and any chance of
$\mathcal{X}$ having an involution. But we may apply the first half of
Theorem \ref{thm:monodromy} to $B'(a(0))[N]$ to conclude that the
vanishing cycle is identified with $[\mu'(0)^{-1}(p)]$ and the
monodromy invariant $\lambda'[N]$ is identified with
$N[\omega]'(0)$. Furthermore, the class $\phi_*\sigma'(0)$ is
invariant under Picard-Lefschetz, and the conclusion of Theorem
\ref{thm:monodromy2} $$\lambda'[N]= N\vec{a}(0)\textrm { mod }O^+(N)$$
holds by continuity: We have $[\omega(t)]=[\omega'(t)]$ because
nodal slides leave the class of the symplectic form invariant. Hence
$$[\omega'(0)]=\lim_{t\rightarrow 0}[\omega'(t)]=\lim_{t\rightarrow
  0}[\omega(t)]=\lim_{t\rightarrow 0}a(t) = a(0).$$

Integral length nodal slides correspond to M1 modifications
of $\mathcal{X}'[N]$ by Proposition \ref{nodal-slide-dual}.
Thus there is a sequence of such
modifications after which we have a Kulikov degeneration
$\mathcal{X}''[N]\rightarrow C$ satisfying
$\Gamma(\mathcal{X}_0''[N])=B(a(0))[N]$. After a series of M2
modifications corresponding to retriangulation
(again Proposition \ref{nodal-slide-dual}), we can produce a Type III degeneration
$\mathcal{X}[N]\rightarrow C$ such that the dual complex
is the standard order $N$ refinement of a triangulated $\ias$
$B(a(0))$. We conclude by Proposition
\ref{refine-base-change} that $\mathcal{X}[N]\rightarrow C$ is
in fact an order $N$ base change of a Kulikov degeneration
$\mathcal{X}\rightarrow C$ such that
$\Gamma(\mathcal{X}_0)=B(\vec{a}(0))$, whose vanishing cycle is the same, 
and whose monodromy-invariant $\lambda$ is $\vec{a}(0)$.

Furthermore, we have produced not just a class $\phi_*\sigma'(0)$, but
an actual surface $\phi(\Sigma'_\ia(0))$ on the general fiber
$\mathcal{X}_t$ (note the M1 and M2 modifications act trivially on
the punctured family). Under the Clemens collapsing map
$\mathcal{X}_t\rightarrow \mathcal{X}_0$ the surface
$\phi(\Sigma'_\ia(0))$ is pinched at the double curves to produce a
union of surfaces $R_i \subset V_i$ on the equator such that
$R_i \cap D_{ij} = R_j\cap D_{ij}$. 

We note that the involution
is restored in the limit $a_i\to 0$ when undoing the nodal
slides. The class $[R_i]$ is invariant under the involution on
an anticanonical pair of deformation type $(V_i,D_i)$. We also know the
values of $[R_i]^2$ and $[R_i]\cdot D_{ij}$ by continuity, so we can apply the
Reconstruction Proposition~\ref{reconstructing} to determine $R_i$ uniquely as the class
of the ramification locus.
\end{proof}

\subsection{An example: the $A_{18}'$ ray}
\label{sec:example}

Consider the $A_{18}'$ subdiagram of $G_{\vin}$
where $a_i=0$ for $i\in [18,0,1,\dots,16]$.  The corresponding cone
in $\fF^{\cox}$ is a ray. Take $\vec{a}$ to be twice the integral generator,
so that it satisfies the parity condition \ref{def:parity}. Then
$a_{17}=6$. Relations in $N$ determine
$(a_{19},a_{20},a_{21},a_{22},a_{23})=(10,8,30,14,22)$.
Recall from Section~\ref{sec:ias-over-coxeter} that $\vec{a}$
corresponds to line bundle $M$ in the nef cone $\textrm{Pic}\, S$ of
the mirror K3 satisfying $M\cdot E_i = a_i$. Letting
$\pi\colon S\to T$ be the double cover of the rational elliptic
surface, $M=\pi^*L$ where $L\cdot F_i = b_i$ with
$\vec{b} = (0,\dots,0, 3, 0,5,4,15,7,11)$. Then, we may further
blowdown $\phi\,:\,T\rightarrow \oT$ to the first ``6-6-6" toric
model. The values $\overline{b}_i=(\phi_*L)\cdot \overline{F}_i$ are
$\overline{b}_6=5$, $\overline{b}_{12}=4$, $\overline{b}_{17}=3$ with
all other $\overline{b}_i=0$.

\begin{figure}
  \begin{center}
    \includegraphics[width=200pt]{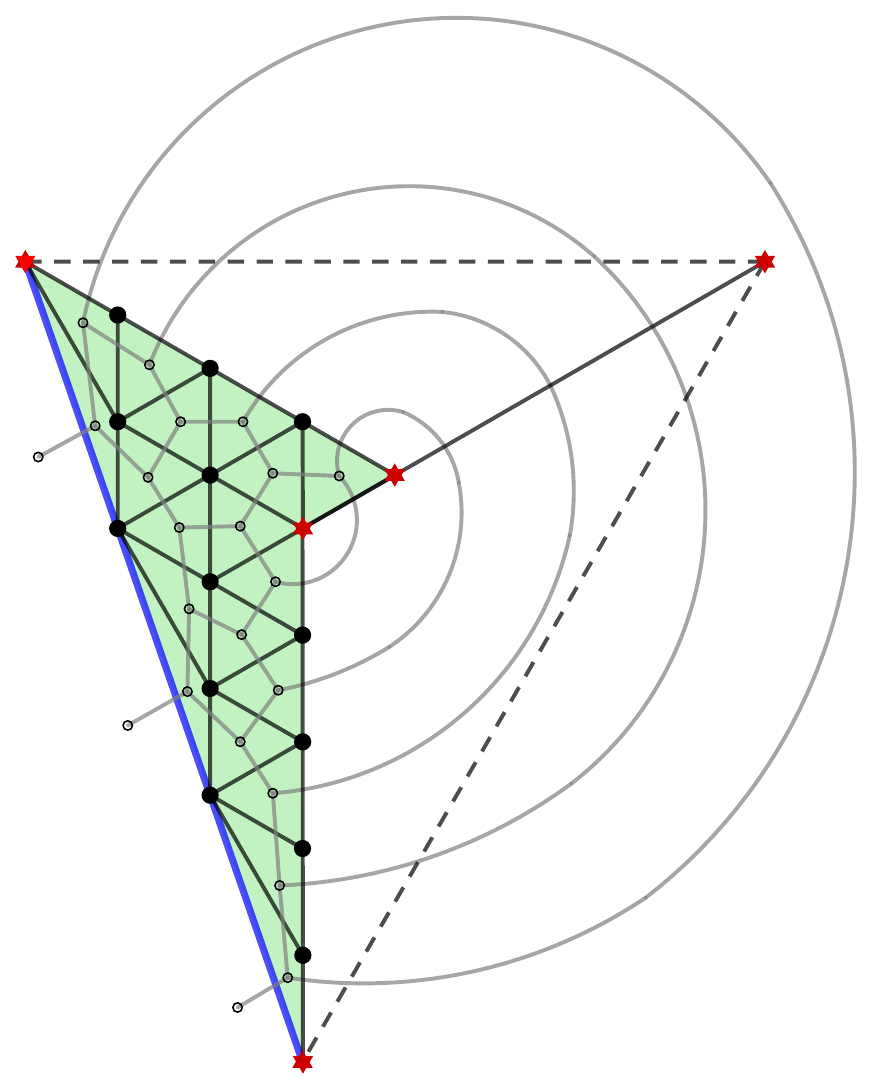}
    \caption{Above: $P(\vec{a})$ for the $\vec{a}\in \fund$ generating the $A_{18}'$ ray. Below:
    A Kulikov surface $\cX_0(\vec{a})$ with $\Gamma(\cX_0(\vec{a}))=B(\vec{a})$.} \label{fig:A18'} \vspace{5pt}
    \includegraphics[width=370pt]{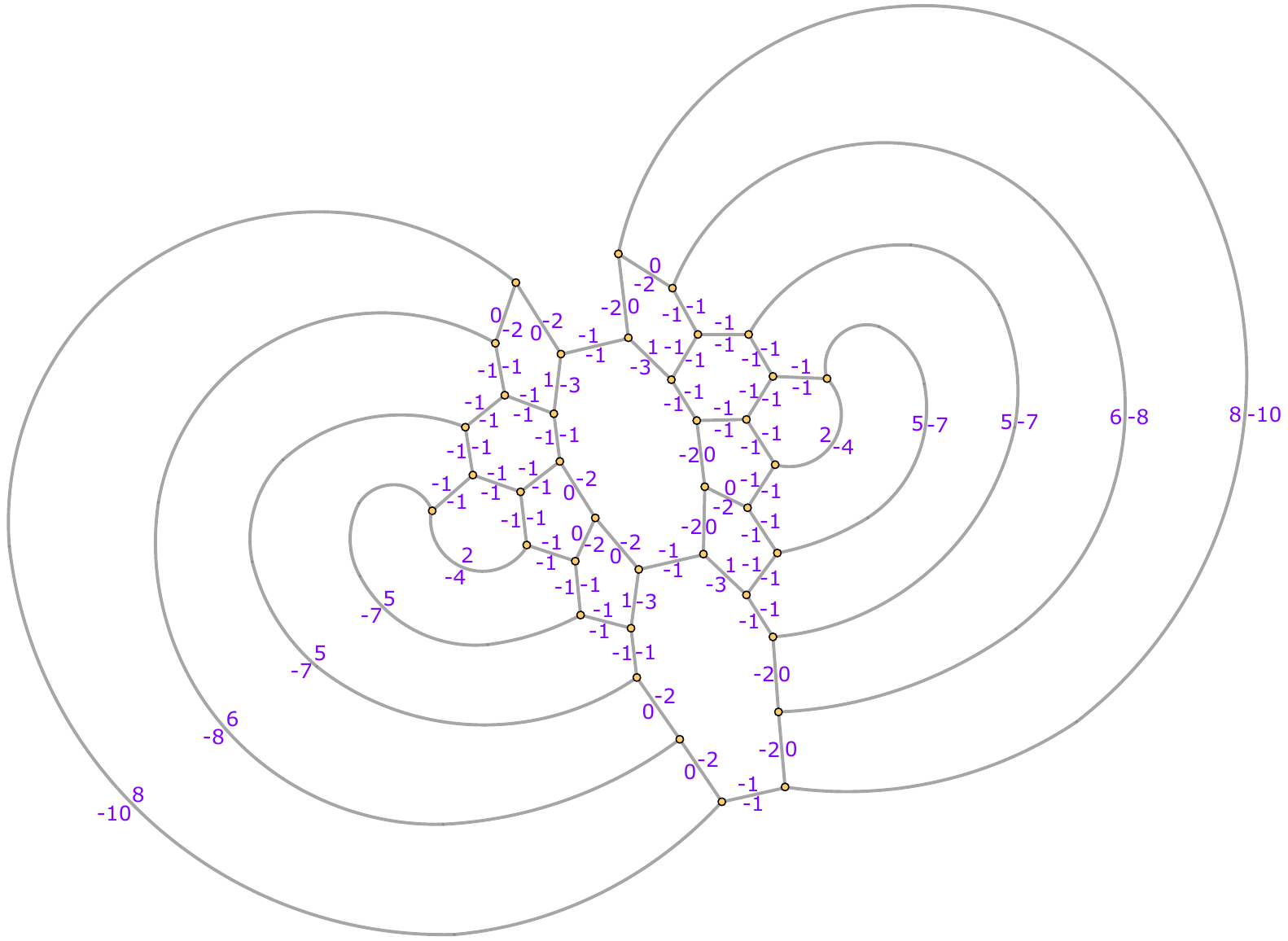}
  \end{center}
\end{figure}

Take a moment polygon of $\overline{T}$ with polarization
$\oL=\phi_*L$ and apply two Symington surgeries of size $5$ and $4$ on
the edges associated to $\overline{b}_6$ and $\overline{b}_{12}$
respectively, producing the green integral-affine disc $P(\vec{a})$ with blue
boundary depicted in the upper part of Figure \ref{fig:A18'}. We double the disc, so that the
blue edge becomes the equator of the $\ias$ $B(\vec{a})$.

We triangulate $B(\vec{a})$
into lattice triangles in an involution-invariant manner,
respecting the blue edge. The singular red stars and non-singular
black points form the vertices $v_i$. Finally, we interpret each vertex as
the pseudo-fan $\mathfrak{F}(V_i,\sum_jD_{ij})$ an anticanonical pair and glue according to the
combinatorics of the triangulation. The resulting Kulikov surface $\cX_0$ is shown in the lower part of Figure \ref{fig:A18'}, with double curves in
gray, self-intersections in purple, and triple points in
yellow. 

The line bundles $R_i$ are trivial on all but three components, those
along the blue line. On the outer component, $R_i$ is the fixed locus
on involution on a resolution of the type $A_{18}'$ involution
pair. On the two other equatorial components, $R_i$ is the fiber of a
ruling along the equator. These glue to form the Cartier divisor
$R$. Taking the image of a multiple of $\mathcal{O}_{\cX_0}(R)$, we
get the stable model: This contracts the northern and southern
hemispheres to a point, contracts a ruling on two equatorial
components, and is a birational morphism of the outer component to the
polarized $A_{18}'$ involution pair.

The resulting stable model is irreducible, and is the contraction of an
anticanonical pair with cycle of self-intersections
$(-10,-2,-1,-2,-10,-1)$ to a singular surface with two boundary
components glued. It has two cyclic quotient singularities at the
north and south poles whose resolutions are a chain of rational curves of
self-intersections $(-10,-2)$ and the images of the $(-1)$-curves which
are glued.

\subsection{Type II Kulikov models} It remains to determine
the Kulikov models corresponding to the rational cusps of $\fund$.

\begin{construction}\label{con:ias-to-kulikov2}
  To a vector $\vec a\ne 0\in \fund$ with $\vec a^2=0$ we associate a
  Type~II Kulikov surface $\cX_0(\vec{a})$ with an involution $\iota_0$
  and a stable surface $(\overline{\cX}_0(\vec{a}),\epsilon \oR)$.

   For the
  relevant connected parabolic diagrams we have the Type II $\wA\wD\wE$
  involution pairs $(X_k,D_k)$ of \cite{alexeev17ade-surfaces}  which glue
  to a stable surface $\overline{\cX}_0(\vec{a})$ with an involution
  $\overline{\iota}_0$ and fixed locus $\oR$.  For the diagrams $\wE_8^2\wA_1$,
  $\wD_{10}\wE_7$, $\wD_{16}\wA_1$ where there are two components, the
  elliptic curves $D_k$ must be isomorphic to a fixed $E$.

  Now we describe the Kulikov models. If $\vec a= m \vec a_0$ with primitive $\vec a_0$ then the dual 
  complex $\Gamma(\cX_0(\vec{a}))$ will be an interval $[0,m]$. A triangulation in this case
  is the subdivision into intervals of length~1.
  
  For $\wD_{10}\wE_7$ and $\wD_{17}\wA_1$, the surface $\cX_0(\vec{a})$ is constructed
 by taking the minimal resolution of each component of $\overline{\cX}_0(\vec{a})$
  and gluing these components, with a chain of $m-1$ $\mathbb{P}^1$-bundles over $E$
  inserted between them, like an accordion.

  In the $\wE_8^2\wA_1$ case we assume after an order $2$ base change
  that $m$ is even. At the ends we put the minimal resolutions of two $\wE_8$ involution pairs.
  We build a chain of surfaces as in the previous case, but on the middle component,
  we blow up a pair of points on the boundary of
  $\mathbb{P}^1\times E$ switched by the involution.
  This corresponds to the $\wA_1=\wA_1^\irr$ diagram.
      
    In the $\wA_{17}$ case, resolve the two simple elliptic singularities
   of the $\wA_{17}$ involution pair $\overline{\cX}_0(\vec{a})= (\oV,\overline{p}_1,\overline{p}_2)$
   to obtain a surface $(V,D_1+D_2)$ which
   is ruled over an elliptic curve with $18$ broken
   fibers, and whose anticanonical divisor $D_1\sqcup D_2$ is the disjoint
   union of two elliptic curves $E$ of self-intersection $-9$. 
   We again assume $m$ is even, and put the anticanonical pair $(V,D_1+D_2)$
  at the $m/2$ vertex. We add ruled surfaces over
  the same elliptic curve $E$ for the integral points
  $l\ne 0, m/2, m$ and cap both ends of the segment
  by the rational anticanonical pair $(\bP^2,E)$.
\end{construction}

\begin{remark}
  The Type II Kulikov models can also be constructed directly from the
  segment, together with the data of how it degenerated from an
  $\ias$; an analogue of the Monodromy Theorem \ref{thm:monodromy}
  likely holds. This requires first generalizing pseudo-fans to allow
  blow-ups of $E\times \mathbb{P}^1$, corresponding to the surfaces in
  the interior of the interval.  The ends of the interval are
  anticanonical pairs $(V,D)$ with $D$ smooth.  These obviously have
  no toric models, but can be constructed via {\it node smoothing
    surgeries}, cf. \cite[Prop.~3.12]{engel2018looijenga}. For example,
  in the $\wA_{17}$ case, the three surgery triangles consume all of
  $\overline{P}$. At the north pole, this dictates three node
  smoothing surgeries on the anticanonical pair
  $(\mathbb{P}^2,L_1+L_2+L_3)$, giving the pair
  $(\mathbb{P}^2,E)$, as in Construction
  \ref{con:ias-to-kulikov2}.
\end{remark}

\subsection{Deformations of Kulikov models with involution}
\label{sec:defs-kulikov-invo}

We now prove that the Kulikov surfaces $\cX_0(\vec{a})$ constructed:
\begin{enumerate}
\item can be made $d$-semistable and admit a smoothing into $F_2$, and
\item the union $R\subset \cX_0(\vec{a})$ of the curves $R_i\subset (V_i,D_i)$
from Theorem \ref{eq-invo} is the flat
limit of the ramification divisor. 
\end{enumerate}

Once these are demonstrated, we show:
\begin{enumerate}
\item[(3)] Every degeneration $C^*\to F_2$ admits a Kulikov limit
of the form $\cX_0(\vec{a})$ with $R$ the flat limit of the ramification
divisor.
\end{enumerate}

We first recall the basic statements about $d$-semistable Kulikov
surfaces from \cite{friedman1983global-smoothings,
  friedman1986type-III, laza2008triangulations,
  gross2015moduli-of-surfaces}. Let $\cX_0$ be a Type III Kulikov
surface with $v$ irreducible components $V_i$, $e$ double curves
$D_{ij}=V_i\cap V_j$, and $f$ triple points
$T_{ijk}=V_i\cap V_j\cap V_k$.
One defines an important lattice of ``numerical Cartier divisors''
\begin{displaymath}
  \widetilde{\Lambda} = \ker\big( \oplus_i \Pic V_i \to \oplus_{i<j} \Pic D_{ij} )
\end{displaymath}
with the homomorphism given by restricting line bundle and applying
$\pm1$ signs.
The map is surjective over $\bQ$ by \cite[Prop. 7.2]{friedman1986type-III}.
The set of
isomorphism classes of not necessarily $d$-semistable Type III
surfaces of the combinatorial type $\cX_0$ is isogenous to
$\Hom(\widetilde{\Lambda}, \bC^*)$.

For a given $\psi \in \Hom(\widetilde{\Lambda},\bC^*)$ the Picard group of the
corresponding surface is $\ker(\psi)$. The surface is $d$-semistable
iff the following $v$ divisors are Cartier:
$\xi_i = \sum_j D_{ij}-D_{ji} \in \widetilde{\Lambda}$. Note that $\sum_i
\xi_i=0$. Thus, the $d$-semistable surfaces correspond to the points
of multiplicative group $\Hom(\Lambda,\bC^*)$, where 
\begin{displaymath}
  \Xi = \frac{ \oplus_i \bZ\xi_i }{ (\sum_i \xi_i)},
  \qquad
  \Lambda = \coker(\Xi\to \widetilde{\Lambda}).
\end{displaymath}

By \cite{friedman1986type-III, friedman1984a-new-proof}, the Clemens-Schmid exact
sequence identifies $\Lambda$ as isometric to
$\{\delta,\lambda\}^\perp/\delta = \lambda^\perp\subset I^\perp/I$ or $J^\perp/J$
in Types III or II, respectively.

The dimension of the space of the $d$-semistable surfaces is
\begin{displaymath}
  \textstyle \sum_i \rho(V_i) - e - (v-1) = (2e-2v+24) - e - (v-1) = e-3v + 25 = 19
\end{displaymath}
because $e-3v = -6$ for a triangulation of a sphere.

\begin{lemma}\label{at-least-one} For all $\vec{a}$, there
is at least one $d$-semistable Kulikov surface
 $\cX_0(\vec{a})$ which admits an involution acting
 by switching the hemispheres of $B(\vec{a})$, and acts in the prescribed
 way on the equatorial components (cf. Theorem \ref{eq-invo}).\end{lemma}
 
 \begin{proof} Within any deformation type,
 the Kulikov surface $\cX_0$ for which $\psi=1$ is the one for which
 all moduli of components and gluing data is trivial: only $-1$
 (in toric coordinates) is blown up on a toric model of a component $(V_i,D_i)$, and
all double curves $D_{ij}$ and $D_{ji}$ are identified by gluing
 (in toric coordinates) by $-1$.
 
 For this surface, it is automatic that the equatorial edges $D_{ij}$ are glued in such a way
 that $R_i\cap D_{ij} = R_j\cap D_{ji}$. Thus, the union of the equatorial components
 admits an involution, and by uniqueness of this Kulikov surface, the involution
 extends across the two hemispheres.
 
 Finally, since $\psi=1$, the $d$-semistability condition is automatic.\end{proof}

 \begin{lemma}\label{d-involutive} Any $d$-semistable equisingular
 deformation of the Kulikov surface $\cX_0(\vec{a})$
from Lemma \ref{at-least-one} keeping $[R]$ Cartier smooths to a
degree $2$ K3 surface.  The space of such deformations
is isogenous to ${\rm Hom}(\Lambda/\tfrac{1}{3}\Z[R],\C^*)$ in Type III 
and ${\rm Hom}(\Lambda/\tfrac{1}{3}\Z[R],\cE)$ in Type II.\end{lemma}

\begin{proof} To prove the second part, observe that Corollary \ref{multiple3}
implies $[\Sigma_{\ia}]$ is $3$-divisible in $\{f,[\omega]\}^\perp/f$ 
and therefore $[R]$ is $3$-divisible in $\{\delta,\lambda\}^\perp/\delta=\Lambda$.
Since, $\tfrac{1}{3}[R]$ is Cartier on the surface with $\psi=1$, 
any deformation keeping $[R]$ Cartier, also keeps $\tfrac{1}{3}[R]$ Cartier.
Thus, any deformation keeping $[R]$ Cartier admits a line bundle $L$, with $L^2=2$.

 By \cite{friedman1983global-smoothings}, the analytic smoothing component $S$
of $\cX_0$ is $20$-dimensional and analytically-locally isomorphic to an extension 
of vector spaces
  \begin{displaymath}
    0\rightarrow \Hom(\Lambda,\C)\to S\to
    H^0(\mathcal{E}xt^1(\Omega^1_{\mathcal{X}_0},\mathcal{O}_{\mathcal{X}_0}))
    \rightarrow 0.
  \end{displaymath}
The first space forms the tangent space
to equisingular $d$-semistable deformations and by $d$-semistability, the third space has dimension one.
The hyperplane $S_{[R]}\subset S$ keeping $[R]$ Cartier fits into an exact subsequence
  \begin{displaymath}
    0\rightarrow \Hom(\Lambda/\tfrac{1}{3}\Z[R],\C)\to S_{[R]}\to
    \C
    \rightarrow 0.
  \end{displaymath}
and a deformation is first-order smoothing iff it has nonzero image in $\C$.
 So, there are smoothings keeping $[R]$ Cartier and admitting a line bundle $L$, $L^2=2$.
 The first part of the lemma follows. 
\end{proof}

\begin{lemma}\label{d-invo-2} Any equisingular deformation as in Lemma \ref{d-involutive}
admits an involution $\iota_0$ and a Cartier divisor $R$ representing the deformation
of $[R]$, realizing it as a Kulikov surface $\cX_0(\vec{a})$ 
coming from Constructions \ref{con:ias-to-kulikov3},  \ref{con:ias-to-kulikov2}.
\end{lemma}

\begin{proof} It suffices to prove that the deformations which keep the class $[R]$ Cartier
also admit an involution $\iota_0$ acting in the desired way on $\cX_0$ and are, therefore,
instances of Construction \ref{con:ias-to-kulikov3} ({\it caveat lector}: $R$ and
${\rm Fix}(\iota_0)$ need not be equal, see Remark \ref{ram-not-r}).

First, we suppose $B(\vec{a})$ is generic. In this case, we prove that a deformation
keeps $[R]$ Cartier iff it deforms the involution $\iota_0$. The reverse implication
is easy, as the Cartier divisor $R$ can be reconstructed from $\iota_0$---it is the pullback of
$\oR_i={\rm Fix}(\overline{\iota}_i)$ from the stable models of the equatorial components.

Next we show that the first-order $d$-semistable equisingular
deformations of $\cX_0$ keeping $[R]$
Cartier preserve the involution. The tangent space to the 
$d$-semistable equisingular deformations is
  $ \Hom(\Lambda,\bC)$. Here, the target vector
  space $\bC$ depends on the orientation of $\Gamma(\cX_0)$,
  so the involution $\iota_0$ acts on
  it as $(-1)$. Thus, the tangent space to deformations preserving the involution
  is $\Hom({\Lambda}/{\Lambda}_+,\C)$, where ${\Lambda}_+$ is the $(+1)$-eigenspace
  of $\iota_0$ on ${\Lambda}$.
  Obviously, $[R]\in {\Lambda}_+$. So all we have to show is that
  $(n_+,n_-):= (\rank {\Lambda}_+, \rank {\Lambda}_-) = (1,18)$.  We now compute the
  ranks of the $(+1)$ and $(-1)$-eigenspaces for all the lattices
  involved.

  Let us denote by $e_E$, $e_N$ the edges of the triangulation of the
  sphere that appear on the equator and in the northern
  hemisphere. One has $e=e_E+2e_N$. Similarly, we have $v=v_E+2v_N$
  for the vertices and $q=q_E+2q_N = 18+6$ for the charges.
  
  For an irreducible component the Picard rank is $\rho(V_i) = e^i+q^i-2$,
  where we only count the edges and charges belonging to $V_i$. For a
  symmetric pair of surfaces in the northern and southern hemispheres
  this gives $(e^i_N+q^i_N-2,\ e^i_N+q^i_N-2)$. For a surface in
  the equator: $(e^i_E+e^i_N+q^i_N-1,\ e^i_N+q^i_E+q^i_N-1)$.
  Adding up the eigenspaces for $\oplus \Pic V_i$ gives:
  \begin{displaymath}
    (2e_E+2e_N - v_E-2v_N + q_N, \quad  2e_N -v_E-2v_N + q_E + q_N).    
  \end{displaymath}
  For the lattice $\oplus \Pic D_{ij}$ the answer is of course
  $(e_E+e_N, e_N)$, and for $\Xi$ it is $(v_E+v_N-1, v_N)$.
  Putting
  this together, the ranks $(n_+,n_-)$ are 
  \begin{displaymath}
    \textstyle    \left(
    \frac12(e-3v) + \frac12(e_E-v_E) + q_N + 1,\quad
    \frac12(e-3v) - \frac12(e_E-v_E) + q_E + q_N
    \right).
  \end{displaymath}
  Using $e-3v=-6$ and $e_E = v_E$, $(n_+,n_-)=(q_N-2, q_E+q_N-3) = (1,18)$.

   When $\Gamma(\cX_0)$ is non-generic, the computation has an additional
  subtlety: The action of the involution on ${\rm Pic}(V_i)$ for an equatorial
  component varies (see Remark \ref{ram-not-r}) as one varies the involution pair 
  $(\oV_i,\oD_i+\epsilon \oR_i)$. But choosing a generic member of the
  space of $(V_i,D_i)$ admitting an involution $\iota_i$ gives
  a specified action on ${\rm Pic}(V_i)$, and for this generic choice,
  $(n_+,n_-)=(1,18)$.
  
  We now lift to higher order deformations, noting that these higher order lifts
  form a torsor over the first-order deformation space $\Hom(\Lambda,\C)$.
  Thus, the involution $\iota_0$ acts on higher order deformations by
  an affine-linear transformation, whose linear part fixes
  an $18$ dimensional subspace. It follows that the involution fixes
  an $18$-dimensional affine-linear subspace. So the involution
  can be lifted to higher order. Furthermore, these lifts are exactly
  those preserving the line bundle, since the fixed locus of the
  involution is Cartier. We conclude that deformations over an analytic   
  open subset of $\Hom(\Lambda,\C^*)$ have an involution. This open subset
  is Zariski dense, since the condition of having such an
  involution is algebraic.
  
  We now specialize from this sublocus of $\Hom(\Lambda,\C^*)$ of Kulikov surfaces with involution,
  observing that a limiting Kulikov surface $\cX_0$ still admits an involution, and the limiting
  class $[R]$ is still Cartier. Alternatively, we can cite \cite[Thm.~B]{alexeev17ade-surfaces}---the
  spaces of ADE surfaces are parameterized by tori $(\C^*)^n$,
 so by varying the moduli of the equatorial components and the edge gluings,
 the space of $\cX_0(\vec{a})$ fills out all of $(\C^*)^{18}$, as opposed to some
 Zariski open subset. 
 
In the Type II case, a dimension count shows that varying moduli of the ADE surfaces
and gluings from Construction \ref{con:ias-to-kulikov2}, with a fixed elliptic curve $E$,
produces an abelian variety isogenous to $E^{17}$. Thus, additionally varying $j(E)$
fills out the entire abelian variety fibration ${\rm Hom}(\Lambda/\tfrac{1}{3}\Z[R],\cE)$
over the modular curve. 
\end{proof}

\begin{theorem}\label{flat-limit-good} Let $\cX_0(\vec{a})$ be a $d$-semistable
Kulikov surface from Constructions \ref{con:ias-to-kulikov3}, \ref{con:ias-to-kulikov2}.
The Cartier divisor $R\subset \cX_0(\vec{a})$ is the flat limit of the
ramification divisors $\cR^*\subset \cX^*$ on any smoothing $\cX\to (C,0)$
keeping $[R]$ Cartier.

Furthermore, every degenerating family $(\cX^*,\cR^*)\to C^*$
of degree $2$ K3 surfaces with ramification divisor admits, after some finite base change,
a Kulikov model $\cX\to (C,0)$ of this form. \end{theorem}

\begin{proof} Let $\cX_0(\vec{a})$ be a generic element of ${\rm Hom}(\Lambda/\tfrac{1}{3}\Z[R],\C^*)$.
Then, each anticanonical pair $(V_i,D_i)$ with involution $\iota_i$ is generic, and the involution
$\iota_0$ acts on $\Lambda$ with eigenspaces of dimension $(1,18)$.
The argument of Lemma \ref{d-invo-2} shows that any smoothing
keeping $[R]$ Cartier preserves the involution, because $\iota_0$ acts on
$H^0(\mathcal{E}xt^1(\Omega^1_{\mathcal{X}_0},\mathcal{O}_{\mathcal{X}_0}))$ by negation.

So there is an involution $\iota$ on any Kulikov model $\cX$ smoothing $\cX_0$ which
keeps $[R]$ Cartier. This implies $\lim_{t\to 0}{\rm Fix}(\iota_t)={\rm Fix}(\iota_0)$ because
$\cX$ is smooth, so ${\rm Fix}(\iota)$ consists of only $0$- and $2$-dimensional components.
In particular, each $2$-dimensional component is irreducible and forms a flat family of divisors.
Furthermore since we are in the generic case, ${\rm Fix}(\iota_t)=R_t$ for all $t$ including $0$.
 
For $\cX_0(\vec{a})$ non-generic, i.e. having $(-2)$-curves in equatorial components,
and a general smoothing $\cX\to (C,0)$, the involution
$\iota_0$ does not extend to a regular involution $\iota$ of $\cX$. Instead, $\cX^*$ admits
an involution $\iota^*$ which only extends as a birational involution
$\iota\colon \cX\dashrightarrow \cX$ whose locus of indeterminacy is the $(-2)$-curves
in the equatorial components, and the restriction $\iota\big{|}_{\cX_0(\vec{a})}$ extends
to $\iota_0$. 

We conclude that the flat limit of $\cR^*$ differs from
${\rm Fix}(\iota_0)$ at most along the equatorial $(-2)$-curves, as does $R$, by construction.
So $\lim_{t\to 0}\cR_t = R+\sum a_iC_i$ for $C_i$ these $(-2)$-curves.
On the other hand $\cR_t^2=R^2$ by construction, $R\cdot C_i=0$, and $C_i$
span a negative-definite lattice, so we conclude that $a_i=0$ for all $i$. This completes
the proof of the first paragraph in the theorem.

To prove the second paragraph, we observe that after a finite base change,
any degeneration $\cX^*\to C^*$ has unipotent monodromy, and thus has 
some monodromy invariant $\lambda\in \fund$. After a further order
$2$ base change, we can ensure vector $\vec{a} \in \Z_{\geq 0}^{24}$ defined
by $(\lambda\cdot r_i)_{i\in \{0,\dots,23\}}$ satisfies the parity condition.
Let $\cX_0(\vec{a})$ be one of the corresponding Kulikov surfaces. By 
Theorem \ref{thm:monodromy3}, the monodromy invariant of a smoothing
$\cX(\vec{a})\to (C,0)$ is, in fact, equal to $\lambda$.

It remains to show that we can vary the continuous moduli of $\cX(\vec{a})$ 
until our given family $\cX^*\to C^*$ agrees with $\cX^*(\vec{a})$.
By Lemmas \ref{d-involutive}, \ref{d-invo-2} the $d$-semistable surfaces $\cX_0(\vec{a})$
keeping $[R]$ Cartier form a family $\mathfrak{X}_0(\vec{a})$ over (a variety isogenous to) $(\C^*)^{18}$ or $\cE^{\times 17}$ in Types III and II, respectively.

A result of Friedman-Scattone \cite[5.5,
  5.6]{friedman1986type-III} shows that the smoothing components of the fibers
  of $\mathfrak{X}_0(\vec{a})$ keeping $[R]$ Cartier
  can be glued together, to form a family $\mathfrak{X}(\vec{a})$ of smooth and Kulikov K3 surfaces with line bundle. The base of $\mathfrak{X}(\vec{a})$ is $19$-dimensional, and up to the action
  of a finite group, is identified with the toroidal extension $F_2\hookrightarrow F_2^\lambda$
  whose only cones are the $\Gamma$-orbits of the ray $\R_{\geq 0}\lambda$. The boundary
  divisor is exactly the base of $\mathfrak{X}_0(\vec{a})$, parameterizing the $d$-stable
  equisingular deformations of $\mathfrak{X}_0(\vec{a})$ keeping $[R]$ Cartier.
  Proposition \ref{right-monodromy} implies that $\cX^*\to C^*$ is a subfamily of $\mathfrak{X}(\vec{a})$, because the period map approximates a co-character $\lambda\otimes \C^*$
  which is completed at $0$ in $F_2^\lambda$.  The theorem follows. 
   \end{proof}

\section{Determination of stable models}
\label{sec:degenerate-k3s}

The goal of this section is to determine the KSBA stable limit of
 any one parameter degeneration $(\cX^*,\epsilon\cR^*)\to C^*$ in $F_2$,.
We describe
the components which will
appear on any stable limit of degree 2 K3 pairs $(X,\epsilon R)$, and
how they are glued.

\subsection{$ADE$ and $\wA\wD\wE$ surfaces}
\label{sec:ade-wade-surfaces}

In this section, we describe the classification of involution
pairs of \cite{alexeev17ade-surfaces} in more detail.

\begin{definition}
  Let $X$ be a normal projective surface with a reduced boundary divisor $D$
  and an involution $\iota\colon X\to X$, $\iota(D)=D$ such that 
  \begin{enumerate}
  \item $K_X+D\sim 0$ is a Cartier divisor linearly equivalent to 0,
  \item the ramification divisor $R$ is Cartier and ample, and
  \item the pair $(X, D+\epsilon R)$ has log canonical singularities
    for $0<\epsilon\ll1$. 
  \end{enumerate}
  Such pairs were called \emph{$(K+D)$-trivial polarized involution
    pairs} in \cite{alexeev17ade-surfaces}, where they are classified
  in terms of the decorated $ADE$ diagrams in Type III and extended
  \wade{} diagrams in Type II.
\end{definition}

The classification in \cite{alexeev17ade-surfaces} is done in terms of
the quotients $(Y,C)=(X,D)/\iota$ and the branch divisors
$B\subset Y$. The surface $X$ is recovered as a double cover
$\pi\colon X\to Y$ branched in $B$. Then $R=\pi\inv(B)$ and
$D=\pi\inv(C)$. 

In toric geometry a lattice polytope $P$ corresponds to a toric
variety $Y_P$ with an ample line bundle $L_P$.  For
many Dynkin diagrams there exists a polytope $P$ corresponding to it
in an obvious way. Then $Y$ is defined to be $Y_P$ and the branch
divisor $B$ to be a divisor in the linear system $|L_P|$.

For example, there are polytopes of shapes associated to $A^-_0$,
$D_5$, $E_7$ in Fig.~\ref{fig:diagrams}.  In general, the polytope $P$
has the following vertices:
\begin{enumerate}
\item $A_n$, $A_n^-$ for $n$ odd, resp. even: $(2,2)$, $(0,0)$, $(n+1,0)$.
\item $\mA_n^-$, $\mA_n$ for $n$ odd, resp. even:
  $(2,2)$, $(1,0)$, $(n+2,0)$.
\item $D_n$ and $D_n^-$: $(2,2)$, $(0,2)$, $(0,0)$, $(n-2,0)$.
\item $E_n$ ($\phmi E_6^-$, $\phmi E_7$, $\phmi E_8^-$):
  $(2,2)$, $(0,3)$, $(0,0)$, $(n-3,0)$. 
\item $\wD_{2n}$: $(0,2)$, $(0,0)$, $(2n-4,0)$, $(4,2)$.
\item $\wE_7$: $(0,4)$, $(0,0)$, $(4,0)$.
\item $\wE_8$: $(0,3)$, $(0,0)$, $(6,0)$.
\end{enumerate}

In the $ADE$ cases, the boundary $D$ has two components. 
In the \wade{} cases $D$ is an irreducible smooth elliptic curve.

The only nontoric initial cases are $\wA_{2n-1}$ and two small exotic
$\wA$ shapes:
\begin{enumerate}\setcounter{enumi}{7}
\item $\wA_{2n-1}$. The surface is cone $\Proj_E(\cO\oplus\cF)$ over
  an elliptic curve $E$, where $\cF$ is a line bundle of degree
  $n$. The boundary $C=0$ is empty and $B\in |-2K_Y|$.
\item $\wA_1^*$. Here, $Y=\bP^2$, the boundary $C$ is a smooth conic,
  and the branch curve $B$ is a possibly singular conic. If $B$ is
  smooth then $X=\bP^1\times\bP^1$; if $B$ is two lines then
  $X=\bP(1,1,2)$ with $R$ passing through the apex. Also included is a
  degenerate subcase when $\bP^2$ degenerates to $Y=\bP(1,1,4)$ with
  $R$ not passing through the apex.
\item $\wA_0^-$. Here, $Y=\bP(1,1,2)=\bF_2^0$.  The curve $C$ is the image
  of $\wC\in|s+3f|$ on~$\bF_2$.  The branch curve is a conic section
  disjoint from the apex.
\end{enumerate}

All other pairs are obtained from these by a process called
``priming'': making up to 4 weighted $(1,2)$-blowups $Y'\to Y$ at the
points of intersection of the branch divisor $B$ with the boundary
$C$, and then contracting parts of the boundary $C'$ on which
$-K_{Y'}$ is no longer ample. On the double cover $\pi\colon X\to Y$
this corresponds to an ordinary smooth blowup at a point of $R\cap D$
and then contracting parts of the boundary $D'$ on which $R'$ is no
longer ample.

These ``primed'' surfaces $Y'$ are usually not toric. But they are
toric in the $\pA_{2n-1}$, $\pA_{2n}^-$, $\pA'_{2n-1}$ and $D'_{2n}$
cases for which there are also lattice polytopes.
The polytope for $\pA_n$ is obtained from that for $A_n$ by
cutting a corner, a triangle with lattice sides $1,1,2$, which
corresponds to the weighted $(1,2)$ blowup. For the $\pA'_{2n+1}$
diagram the corners on both sides are cut. For the $D_{2n}'$ diagram,
the corner on the ``right'' side is cut.
See a concrete example of a polytope of shape $\pA_4$ in
Fig.~\ref{fig:diagrams}.

Examples \ref{example0}, \ref{example1}, \ref{example2}, \ref{example3}, \ref{example4}, \ref{example5}
describe explicitly the minimal resolutions of involution pairs $(X,D)$ of shapes
$A_0^-$, $A_{2n-1}$, $\pA'_{2n-1}$, $D_{2n}$, $D'_{2n}$, $E_n$
to smooth anticanonical pairs admitting an involution.

\begin{notation} In general, the involution pairs with elliptic diagram have two boundary components,
each isomorphic to $\mathbb{P}^1$, and meeting at two points to form a banana curve. 
We call the two nodes the {\it north} and {\it south poles}, and the two components
the {\it left} and {\it right components}. \end{notation}

\subsection{All degenerations of K3 surfaces of degree 2}
\label{sec:degn-types}

Recall that the {\it stable type} (Definition \ref{def:stable-type}) of an
elliptic or maximal parabolic subdiagram of $G_{\cox}$ was a
cyclically ordered list of its equatorial diagrams, with $\mA_0$ and $A_0^-$
diagrams inserted in the space between diagrams.

\begin{definition}\label{def:glued-pair}
Associated to every Type III stable type, we build a stable surface as follows:
For each diagram, we take an involution pair
$(X_k,D_k,\iota_k)$ with the corresponding label by \cite{alexeev17ade-surfaces} (see
  \ref{sec:ade-wade-surfaces}). Then, we successively glue the surfaces together
$(X,\iota)=\cup_k\,(X_k,D_k,\iota_k)$ along their boundary components,
identifying the right component of $D_k$ to the left component of $D_{k+1}$
and identifying the two north poles and the two south poles.
The intersection complex of the resulting surface is a sphere, decomposed
like the slices of an orange.
We glue in such a way that the ramification divisors
$R_k$ glue to a Cartier, ample divisor $R$.

In Type II, we do something similar for stable types $\wE_8^2$,
  $\wD_{10}\wE_7$, $\wD_{16}\wA_1$ by gluing the two
  components along elliptic curves. Finally, the stable surface
  associated to $\wA_{17}$ is simply the $\wA_{17}$ involution pair.
  
  The scheme-theoretic structure of the surface $(X,\iota)$
  is uniquely determined by the requirement that the gluing be seminormal
  with SNC double locus, see \cite[Prop.~5.3, Cor.~5.33]{kollar2013singularities-of-the-minimal}.
\end{definition}

\begin{example} Consider the empty subdiagram of $G_\cox$, corresponding
to the open cell of $\fund$. Its stable type is $(A_0^-\mA_0)^9$, see Fig.~\ref{fig:diagrams},
and the corresponding stable surface is the result of taking $18$ copies of
$(\mathbb{P}^2,L+C)$ with a line and conic, and successively gluing
conics to conics, and lines to lines, in such a way that
the fixed divisors, which are lines in each $\mathbb{P}^2$, glue into
a Cartier divisor.

This will be the unique maximal degeneration
of $\oF_2^\slc$. \end{example}

\begin{theorem}\label{thm:irred-comps}
  The stable limits of K3 pairs $(X,\epsilon R)$ of degree 2, polarized
  by the ramification divisor, are exactly the stable surfaces of
  Definition~\eqref{def:glued-pair}, formed from the union of involution pairs
  associated to a stable type of an elliptic or maximal
  parabolic subdiagram of $G_\cox$.
 
 More precisely, if the monodromy-invariant $\lambda$ of a one-parameter
 degeneration $\cX^*\to C^*$ lies in the relative interior
$\lambda\in \sigma^o$ of a cone $\sigma\in \mathfrak{F}_\cox$,
the stable limit is a stable surface associated to the stable type of
the subdiagram defining $\sigma$.
\end{theorem}

\begin{proof} Let $\cX^*\to C^*$ be a degeneration of degree $2$
K3 surfaces with monodromy invariant $\lambda$. By Theorem \ref{flat-limit-good},
there is some finite base change and an extension to a Kulikov model
$\cX\to (C,0)$ for which the central fiber $\cX_0=\cX_0(\vec{a})$ arises
from Constructions  \ref{con:ias-to-kulikov3}, \ref{con:ias-to-kulikov2}.
Here $\vec{a}\in \Z^{24}_{\geq 0}$ is the vector corresponding
to $\lambda\in \fund$ via $\lambda\cdot r_i=a_i$. The flat limit of $R$
is then a Cartier divisor on $\cX_0(\vec{a})$ which is empty
in the hemispheres of $\cX_0(\vec{a})$ and is the pullback of
$R_k ={\rm Fix}(\iota_k)$ on the involution pair $(X_k,D_k)$
which is the contraction of an equatorial component
(see Theorem \ref{eq-invo}, but note that
the involution pair is notated there as $(\oV_i,\oD_i)$).

In particular, $\cR\subset \cX$ is a relatively big and nef Cartier divisor
not containing any strata of $\cX_0$. By the proof of \ref{thm:extend-family-stable-k3},
the stable limit of $\cX^*\to C^*$ can be computed as
${\rm Proj}_C \bigoplus_{n\geq 0} \pi_*\cO_{\cX}(n\cR)$ which contracts
each hemisphere of $\cX_0$ to a single point, contracts the edges along
the equator by rulings, and contracts each equatorial vertex to the involution
pair $(X_k,D_k)$. The resulting stable surface is exactly that described
 in Definition \ref{def:glued-pair}.
 
 It is worth remarking that when a subdiagram of $G_{\cox}$ has an
 $A_1^\irr$ or $\wA_1^\irr$ component, there is an equatorial surface in $\cX_0$
 receiving two internal blow-ups switched by the involution, but the information
 of the location of these blow-ups is lost on the stable model,
 because they are contracted to points
 on a double curve.
 \end{proof}
 
 \subsection{Moduli of stable strata}

The following proposition should be compared with
\eqref{prop:strata-in-F2tor}. 

\begin{proposition}\label{prop:strata-in-F2slc}
  The strata in $\oF_2^\slc$ are as follows:
  \begin{enumerate}
  \item For a Type III stable type $G^\rel$,
    $\Str(G^\rel)$ is, up to an isogeny and a $W(G^\rel)$ action,
    the root torus $\Hom(R_{G^\rel}, \bC^*)$.
  \item For a Type II stable type $\wG^\rel$,
    $\Str(\wG^\rel)$ is, up to an isogeny
    and a $W(G^\rel)$ action, $\Hom(R_{G^\rel}, \cE) \simeq \cE^{17}$,
    where $\cE^{17}\to\cM_1$ is the self fiber product of the
    universal family of elliptic curves $\cE\to\cM_1$ over its moduli
    stack. 
  \end{enumerate}
\end{proposition}

\begin{proof}
  The parameter space for a Type III stratum is, up to a
  finite group, the product of the parameter spaces for the
  irreducible components $(X_k,D_k+\epsilon R_k)$, because
  the gluings of double curves which make $\cup_k R_k$ Cartier
  is, up to a finite group, unique. By
  \cite{alexeev17ade-surfaces}
   each of these is a quotient of torus
  isogenous to the root torus $\Hom(R_{G_k},\bC^*)$ by the Weyl group
  $W(G_k)$.
  Without the additional data of an involution,
  this result is essentially due to Gross-Hacking-Keel
  \cite{gross2015moduli-of-surfaces}.
\va{I added this at Alan's suggestion}

  The same works for Type II strata. Such a stratum is
  a finite quotient of the fiber product over $\cM_1$
  of the period domains of involution pairs with smooth elliptic
  anticanonical divisor. The {\it period point} of an
  anticanonical pair $(V_k,D_k)$
  is the element of ${\rm Hom}(D^\perp,D)$ which sends
  $L\mapsto L\big{|}_{D_k} \in {\rm Pic}^0(D_k)=D_k$. The $\iota$-invariant
  period points form an abelian subvariety isogenous to ${\rm Hom}(R_{G_k},D_k)$
  and the moduli space for each component is the quotient by
  the $\iota$-invariant ``admissible isometries" of $H^2(V_k,D_k)$,
  c.f. \cite{friedman2015on-the-geometry}, 
  which in our case is the Weyl group $W(G_k)$.
  
  Fixing an elliptic curve
  $D=D_k$ and taking the product of moduli of components
  produces the quotient by $W(G^{\rel})$ of
  an abelian variety isogenous to ${\rm Hom}(R_{G^\rel},D)$.
  Finally, we may vary the moduli of $D$ over $\cM_1$, giving
  the fiber product. \end{proof}

\section{Proof of main theorems}
\label{sec:family}

We now assemble the ingredients from the above sections
to prove the main theorems. In the proof of Theorem \ref{flat-limit-good},
we defined the toroidal extension $F_2\hookrightarrow F_2^\lambda$ whose
fan consists of the $\Gamma$-orbit of a ray $\R_{\geq 0}\lambda$, and
a family $\mathfrak{X}(\vec{a})\to \mathfrak{U}(\vec{a})$ of Kulikov and smooth K3 surfaces,
with $\mathfrak{U}(\vec{a})$ a finite cover of $F_2^\lambda$ and $\vec{a} = (\lambda\cdot r_i)_{i\in\{0,\dots,23\}}$ assumed to satisfy the parity condition.
Recall that the boundary divisor of $\mathfrak{U}(\vec{a})$ was isogenous to ${\rm Hom}(\Lambda/\tfrac{1}{3}\Z[R],\C^*)\simeq (\C^*)^{18}$
or ${\rm Hom}(\Lambda/\tfrac{1}{3}\Z[R],\cE)\simeq \cE^{\times 17}$, where $\Lambda$ (see Section \ref{sec:defs-kulikov-invo}) is the lattice $\lambda^\perp\subset I^\perp/I$
or $J^\perp/J$.

\begin{theorem}\label{thm:partial-comp}
  Let $\lambda\in \sigma_G^o\cap N$ lie in the relative interior
  of a cone of $\fF^\cox$ for a subdiagram $G\subset G_\vin$. Assume
  $\vec{a} = (\lambda\cdot r_i)_{i\in\{0,\dots,23\}}$ satisfies the parity condition.
  Then, the classifying map $$\mathfrak{U}(\vec{a})\dashrightarrow \oF_2^\slc$$
  is a morphism, and the induced morphism on the boundary divisor is (isogenous to)
  the restriction map ${\rm Hom}(\Lambda/\tfrac{1}{3}\Z[R],\C^*\textrm{ or }\cE)\to
  {\rm Hom}(R_{G^\rel},\C^*\textrm{ or }\cE)$
for the natural inclusion $R_{G^\rel}\hookrightarrow R_G\hookrightarrow \Lambda/\frac{1}{3}\Z[R]$,
followed by the quotienting by a finite group.
\end{theorem}

\begin{proof}
The proof is essentially the same as Theorem \ref{thm:irred-comps}, except we
don't restrict to a one-parameter subfamily. Let $\mathfrak{R}\subset \mathfrak{X}(\vec{a})$ be
the universal ramification divisor and let $\mathfrak{L} = \cO_{\mathfrak{X}(\vec{a})}(\mathfrak{R})$
be the corresponding line bundle, which is relatively big and nef. The family of divisors $\mathfrak{R}$ exists because the flat limit of the ramification divisor on any Kulikov model is $R\subset \cX_0(\vec{a})$ by Theorem
\ref{flat-limit-good}.

By Shepherd-Barron \cite{shepherd-barron1981extending-polarizations}
  the higher cohomology groups of $\mathfrak{L}^n$ are zero on every fiber,
  so for $n\ge4$, $\mathfrak{L}^n$ defines a contraction to a
  model with an ample line bundle. Since
  the divisors $\mathfrak{R}$ do not contain strata on any fiber by construction,
  the fibers in the contracted family are stable pairs
  $$(\overline{\mathfrak{X}}(\vec{a}),\epsilon \overline{\mathfrak{R}})\to \mathfrak{U}(\vec{a})$$
  and the fibers over the boundary divisor have stable type determined by $G^\rel$,
  by Theorem \ref{thm:irred-comps}. This proves that the classifying map is a morphism.
  
  So the classifying map induces a morphism from ${\rm Hom}(\Lambda,\C^*\textrm{ or }\cE)$
  to the slc stratum ${\rm Str}(G^\rel)$ of Proposition \ref{prop:strata-in-F2slc}.
  We claim that this morphism factors through the natural map of tori induced by
  the inclusions $R_{G^\rel}\hookrightarrow R_G\hookrightarrow \Lambda$---note that
  $R_G\hookrightarrow \Lambda$ because $\lambda\in \sigma\implies \sigma^\perp\subset \lambda^\perp\subset I^\perp/I \implies R_G\subset \Lambda.$

 Let $(V_i,D_i)$ be an equatorial component of $\mathcal{X}_0$, and define
 $$\Lambda_i := {\rm span}\{D_{ij}\}^\perp\subset H^2(V_i,\Z).$$ The 
 period domain \cite{gross2015moduli-of-surfaces, friedman2015on-the-geometry}
 for anticanonical pairs $(V_i,D_i)$ is ${\rm Hom}(\Lambda_i,\C^*)$ while
 the corresponding period domain for involution pairs \cite{alexeev17ade-surfaces}
 is a torus with character lattice isogenous to $R_{G_i}\subset \Lambda_i$
 (more canonically a quotient), consisting of periods of anticanonical pairs $(V_i,D_i)$ accepting an involution
 $\iota_i$. Finally, observe that there is an inclusion
 $\textstyle \bigoplus_i \Lambda_i\hookrightarrow \Lambda$ as every class in $\Lambda_i$
 can be extended by $0$ to a numerically Cartier class on $\cX_0$. This map
 induces the inclusion $R_G\hookrightarrow \Lambda/\tfrac{1}{3}\Z[R]$.
 We conclude that the map on moduli
 ${\rm Hom}(\Lambda/\tfrac{1}{3}\Z[R],\C^*\textrm{ or }\cE)\to {\rm Str}(G^\rel)$
 is induced by the claimed map of lattices. \end{proof}

\begin{theorem}\label{thm:tor-slc}
  The rational map $\varphi\colon\oF_2^\semi \ratmap \oF_2^\slc$ is regular,
  and is the normalization map.
\end{theorem}
\begin{proof}
  We first prove that the rational map $\varphi'\colon \oF_2^\tor
  \ratmap \oF_2^\slc$ is regular. For any ray $\bR_{\ge0}\lambda$ of
  the fan the map extends over the interior of the corresponding
  divisor of $\oF_2^\tor$ by Theorem~\ref{thm:partial-comp}. 
So, if there is any
  indeterminacy locus of $\varphi'$ then it is contained in the Type
  III locus.

  Suppose that $\varphi'$ is not regular. Let
  $\oF_2^\tor\xleftarrow{\ p\ } Z\xrightarrow{\ q\ }\oF_2^\slc$ be a
  resolution of singularities of $\varphi'$.  The preimage
  $Z_x=p\inv(x)$ of a point $x\in \oF_2^\tor$
  is projective. By \eqref{thm:irred-comps},
  \eqref{prop:strata-in-F2slc}  one has 
  $q(Z_x) \subset \Str(G^\rel)$. But by~\eqref{prop:strata-in-F2slc}
  every Type III stratum in $\oF_2^\slc$ is affine. So the map
  $Z_x\to \oF_2^\slc$ is constant. We conclude by
  Lemma~\ref{lem:extend-rat-map}.

  The morphism $\varphi'$ factors through
  $\varphi\colon\oF_2^\semi\to \oF_2^\tor$: In fact, by Theorems
  \ref{thm:partial-comp} and \ref{tor-to-semi}, the curves contracted
  by $F_2^\lambda\to \oF_2^\semi$ and $F_2^\lambda\to \oF_2^\slc$ are
  the same, giving the claim.  Then $\varphi$ is a birational morphism
  with finite fibers, so by Zariski's main theorem, it is the
  normalization.
\end{proof}

\begin{corollary}
 The Stein factorization of $\oF_2^\tor \to \oF_2^\slc$ is
  $\oF_2^\tor \to \oF_2^\semi \to \oF_2^\slc$.
\end{corollary}
\begin{proof}
  This follows from the fact that the fibers of $\oF_2^\tor\to \oF_2^\semi$ are
  connected.
\end{proof}

\begin{corollary}\label{stack}
  There is a regular map $\ocF_2^\semi \to \ocF_2^\slc$ of Deligne-Mumford stacks,
  for an appropriate choice of stack structure on $\oF_2^\semi$.\end{corollary}

\begin{remark} Corollary \ref{stack} is essentially a tautology by pulling
back the stack structure, but it is subtle from the perspective of arithmetic quotients:
\begin{enumerate}
\item Even the interior is not the stack quotient $[\bD:\Gamma]$.
The Heegner divisors associated to roots $\beta\in h^\perp$ have inertia in
$[\bD:\Gamma]$ but not in $F_2$. 
\item Due to the presence of generic automorphisms on $\slc$ strata,
we need a {\it stacky fan}: For each cone $\sigma\in \fF^\cox$ we
must choose a sublattice of ${\rm span}(\sigma)\cap N$, which introduces
inertia at the toric boundary components.
\end{enumerate}
  \end{remark}
  
Similar to the enumeration of the strata of $\oF_2^\tor$
in~\eqref{lem:parabolic-subdiagrams} and \eqref{lem:numbers-strata-Ftor}, 
by looking at the subdiagrams of $G_\vin$ without irrelevant connected
components only, mod $S_3$ or $D_9$, one can enumerate the strata of
$\oF_2^\semi$ or $\oF_2^\slc$. In particular, we have:

\begin{lemma}
  Both in $\oF_2^\semi$ and in $\oF_2^\slc$ there are $38$ boundary
  divisors, of which~$3$ are of Type II and $35$ are of Type III.
\end{lemma}

\begin{remark} The normalization map $\oF_2^\semi\to \oF_2^\slc$ is
not the identity map. For instance, when a diagram $G^\rel$ is entirely
contained in the $18$-cycle $0,\dots,17$, the resulting stable pair stratum
is the same for all diagrams in the $D_9$ dihedral group orbit of $G^\rel$.
For semi-toric strata, only diagrams related by $D_3\simeq S_3$ are identified.
\end{remark}

\bibliographystyle{amsalpha}

\def\cprime{$'$}
\providecommand{\bysame}{\leavevmode\hbox to3em{\hrulefill}\thinspace}
\providecommand{\MR}{\relax\ifhmode\unskip\space\fi MR }
\providecommand{\MRhref}[2]{%
  \href{http://www.ams.org/mathscinet-getitem?mr=#1}{#2}
}
\providecommand{\href}[2]{#2}

\end{document}